\documentclass[10pt]{amsart}
\usepackage{}
\usepackage{amssymb}
\usepackage{xcolor}
\usepackage{mathabx}

\usepackage{amsmath}

\usepackage[top=0.8in, bottom=0.8in, left=1.25in, right=1.25in]{geometry}

\input xy
\xyoption{all}
\usepackage{amsfonts}
\usepackage{mathrsfs}
\usepackage{tikz-cd}

\usepackage{stmaryrd}

\usepackage{graphicx}
\usepackage{amscd,amsbsy,amsthm}
\usepackage[all]{xy}
\usepackage[colorlinks,plainpages,urlcolor=blue]{hyperref}

\usepackage{verbatim}

\usepackage{enumitem}

\usepackage[normalem]{ulem}

\usepackage{tikz}
\usetikzlibrary{arrows,calc}
\tikzset{
>=stealth',
help lines/.style={dashed, thick},
axis/.style={<->},
important line/.style={thick},
connection/.style={thick, dotted},
}

\newcommand {\Omit}[1]{}

\usepackage{stackengine,scalerel}

\newcommand{\nc}{\newcommand}
\nc{\rnc}{\renewcommand}
\nc{\bb}[1]{{\mathbb #1}}
\nc{\bbA}{\bb{A}}\nc{\bbB}{\bb{B}}\nc{\bbC}{\bb{C}}\nc{\bbD}{\bb{D}}
\nc{\bbE}{\bb{E}}\nc{\bbF}{\bb{F}}\nc{\bbG}{\bb{G}}\nc{\bbH}{\bb{H}}
\nc{\bbI}{\bb{I}}\nc{\bbJ}{\bb{J}}\nc{\bbK}{\bb{K}}\nc{\bbL}{\bb{L}}
\nc{\bbM}{\bb{M}}\nc{\bbN}{\bb{N}}\nc{\bbO}{\bb{O}}\nc{\bbP}{\bb{P}}
\nc{\bbQ}{\bb{Q}}\nc{\bbR}{\bb{R}}\nc{\bbS}{\bb{S}}\nc{\bbT}{\bb{T}}
\nc{\bbU}{\bb{U}}\nc{\bbV}{\bb{V}}\nc{\bbW}{\bb{W}}\nc{\bbX}{\bb{X}}
\nc{\bbY}{\bb{Y}}\nc{\bbZ}{\bb{Z}}
\nc{\mbf}[1]{{\mathbf #1}}
\nc{\bfA}{\mbf{A}}\nc{\bfB}{\mbf{B}}\nc{\bfC}{\mbf{C}}\nc{\bfD}{\mbf{D}}
\nc{\bfE}{\mbf{E}}\nc{\bfF}{\mbf{F}}\nc{\bfG}{\mbf{G}}\nc{\bfH}{\mbf{H}}
\nc{\bfI}{\mbf{I}}\nc{\bfJ}{\mbf{J}}\nc{\bfK}{\mbf{K}}\nc{\bfL}{\mbf{L}}
\nc{\bfM}{\mbf{M}}\nc{\bfN}{\mbf{N}}\nc{\bfO}{\mbf{O}}\nc{\bfP}{\mbf{P}}
\nc{\bfQ}{\mbf{Q}}\nc{\bfR}{\mbf{R}}\nc{\bfS}{\mbf{S}}\nc{\bfT}{\mbf{T}}
\nc{\bfU}{\mbf{U}}\nc{\bfV}{\mbf{V}}\nc{\bfW}{\mbf{W}}\nc{\bfX}{\mbf{X}}
\nc{\bfY}{\mbf{Y}}\nc{\bfZ}{\mbf{Z}}
\nc{\bfa}{\mbf{a}}\nc{\bfb}{\mbf{b}}\nc{\bfc}{\mbf{c}}\nc{\bfd}{\mbf{d}}
\nc{\bfe}{\mbf{e}}\nc{\bff}{\mbf{f}}\nc{\bfg}{\mbf{g}}\nc{\bfh}{\mbf{h}}
\nc{\bfi}{\mbf{i}}\nc{\bfj}{\mbf{j}}\nc{\bfk}{\mbf{k}}\nc{\bfl}{\mbf{l}}
\nc{\bfm}{\mbf{m}}\nc{\bfn}{\mbf{n}}\nc{\bfo}{\mbf{o}}\nc{\bfp}{\mbf{p}}
\nc{\bfq}{\mbf{q}}\nc{\bfr}{\mbf{r}}\nc{\bfs}{\mbf{s}}\nc{\bft}{\mbf{t}}
\nc{\bfu}{\mbf{u}}\nc{\bfv}{\mbf{v}}\nc{\bfw}{\mbf{w}}\nc{\bfx}{\mbf{x}}
\nc{\bfy}{\mbf{y}}\nc{\bfz}{\mbf{z}}

\nc{\mcal}[1]{{\mathcal #1}}
\nc{\calA}{\mcal{A}}\nc{\calB}{\mcal{B}}\nc{\calC}{\mcal{C}}\nc{\calD}{\mcal{D}}
\nc{\calE}{\mcal{E}} \nc{\calF}{\mcal{F}}\nc{\calG}{\mcal{G}}\nc{\calH}{\mcal{H}}
\nc{\calI}{\mcal{I}}\nc{\calJ}{\mcal{J}}\nc{\calK}{\mcal{K}}\nc{\calL}{\mcal{L}}
\nc{\calM}{\mcal{M}}\nc{\calN}{\mcal{N}}\nc{\calO}{\mcal{O}}\nc{\calP}{\mcal{P}}
\nc{\calQ}{\mcal{Q}}\nc{\calR}{\mcal{R}}\nc{\calS}{\mcal{S}}\nc{\calT}{\mcal{T}}
\nc{\calU}{\mcal{U}}\nc{\calV}{\mcal{V}}\nc{\calW}{\mcal{W}}\nc{\calX}{\mcal{X}}
\nc{\calY}{\mcal{Y}}\nc{\calZ}{\mcal{Z}}
\nc{\fA}{\frak{A}}\nc{\fB}{\frak{B}}\nc{\fC}{\frak{C}} \nc{\fD}{\frak{D}}
\nc{\fE}{\frak{E}}\nc{\fF}{\frak{F}}\nc{\fG}{\frak{G}}\nc{\fH}{\frak{H}}
\nc{\fI}{\frak{I}}\nc{\fJ}{\frak{J}}\nc{\fK}{\frak{K}}\nc{\fL}{\frak{L}}
\nc{\fM}{\frak{M}}\nc{\fN}{\frak{N}}\nc{\fO}{\frak{O}}\nc{\fP}{\frak{P}}
\nc{\fQ}{\frak{Q}}\nc{\fR}{\frak{R}}\nc{\fS}{\frak{S}}\nc{\fT}{\frak{T}}
\nc{\fU}{\frak{U}}\nc{\fV}{\frak{V}}\nc{\fW}{\frak{W}}\nc{\fX}{\frak{X}}
\nc{\fY}{\frak{Y}}\nc{\fZ}{\frak{Z}}
\nc{\fa}{\frak{a}}\nc{\fb}{\frak{b}}\nc{\fc}{\frak{c}} \nc{\fd}{\frak{d}}
\nc{\fe}{\frak{e}}\nc{\fFf}{\frak{f}}\nc{\fg}{\frak{g}}\nc{\fh}{\frak{h}}
\nc{\fri}{\frak{i}}\nc{\fj}{\frak{j}}\nc{\fk}{\frak{k}}\nc{\fl}{\frak{l}}
\nc{\fm}{\frak{m}}\nc{\fn}{\frak{n}}\nc{\fo}{\frak{o}}\nc{\fp}{\frak{p}}
\nc{\fq}{\frak{q}}\nc{\fr}{\frak{r}}\nc{\fs}{\frak{s}}\nc{\ft}{\frak{t}}
\nc{\fu}{\frak{u}}\nc{\fv}{\frak{v}}\nc{\fw}{\frak{w}}\nc{\fx}{\frak{x}}
\nc{\fy}{\frak{y}}\nc{\fz}{\frak{z}}

\newtheorem{theorem}{Theorem}[section]
\newtheorem{lemma}[theorem]{Lemma}
\newtheorem{corollary}[theorem]{Corollary}
\newtheorem{prop}[theorem]{Proposition}

\newtheorem{setting}[theorem]{Setting}

\theoremstyle{definition}
\newtheorem{definition}[theorem]{Definition}
\newtheorem{example}[theorem]{Example}
\newtheorem{remark}[theorem]{Remark}

\newtheorem{conj}[theorem]{Conjecture}

\newtheorem{prop-defi}[thm]{Proposition-Definition}
\newtheorem{thm-defi}[thm]{Theorem-Definition}
\newtheorem{lem-defi}[thm]{Lemma-Definition}

\DeclareMathOperator{\Kos}{\mathrm{Kos}}
\DeclareMathOperator{\Fact}{\mathrm{Fact}}
\DeclareMathOperator{\qcoh}{\mathrm{qcoh}}
\DeclareMathOperator{\coh}{\mathrm{coh}}
\DeclareMathOperator{\Tot}{\mathrm{Tot}}
\DeclareMathOperator{\Spin}{\mathrm{Spin}}

 \DeclareMathOperator{\id}{id}
\DeclareMathOperator{\Image}{Im} 
 
 \DeclareMathOperator{\GL}{GL}
\DeclareMathOperator{\Hom}{{Hom}}

\DeclareMathOperator{\Hilb}{{Hilb}}

\DeclareMathOperator{\Spec}{{Spec}} \DeclareMathOperator{\tr}{tr}
\DeclareMathOperator{\Aut}{Aut}

\DeclareMathOperator{\Coh}{Coh}

\newcommand{\loc}{\mathrm{loc}}

\newcommand{\IndCoh}{\operatorname{IndCoh}}
\newcommand{\QCoh}{\operatorname{QCoh}}

\DeclareMathOperator{\Ind}{Ind}

\newcommand{\pt}{\text{pt}}

\newcommand{\Z}{\bbZ}
\newcommand{\C}{\bbC}

\newcommand{\QM}{\mathrm{QM}
}

\DeclareMathOperator{\Crit}{Crit}
\DeclareMathOperator{\Perf}{Perf}
\DeclareMathOperator{\bbf}{\textbf{f}}

\DeclareMathOperator{\fBun}{\fB un}
\DeclareMathOperator{\bCrit}{\textbf{Crit}}
\DeclareMathOperator{\bMap}{\pmb{\mathfrak{M}}ap}

\newcount\cols
{\catcode`,=\active\catcode`|=\active
 \gdef\Young(#1){\hbox{$\vcenter
 {\mathcode`,="8000\mathcode`|="8000
  \def,{\global\advance\cols by 1 &}
  \def|{\cr
        \multispan{\the\cols}\hrulefill\cr
        &\global\cols=2 }%
  \offinterlineskip\everycr{}\tabskip=0pt
  \dimen0=\ht\strutbox \advance\dimen0 by \dp\strutbox
  \halign
   {\vrule height \ht\strutbox depth \dp\strutbox##
    &&\hbox to \dimen0{\hss$##$\hss}\vrule\cr
    \noalign{\hrule}&\global\cols=2 #1\crcr
    \multispan{\the\cols}\hrulefill\cr%
   }}$}}}

\newcommand{\gufang}[1]{\textcolor{red}{$[$ Gufang: #1 $]$}}

\newcommand{\yc}[1]{\textcolor{blue}{  #1 }}

\newcommand{\yl}[1]{\textcolor{blue}{$[$ Yalong: #1 $]$}}

\newcommand{\yk}[1]{\textcolor{orange}{$[$ Yukinobu: #1 $]$}}

\DeclareFontFamily{U}{rsfs}{%
\skewchar\font127}
\DeclareFontShape{U}{rsfs}{m}{n}{%
<-6>rsfs5<6-8.5>rsfs7<8.5->rsfs10}{}
\DeclareSymbolFont{rsfs}{U}{rsfs}{m}{n}
\DeclareSymbolFontAlphabet
{\mathrsfs}{rsfs}
\DeclareRobustCommand*\rsfs{%
\@fontswitch\relax\mathrsfs}

\newdimen\argwidth
\def\db[#1\db]{
 \setbox0=\hbox{$#1$}\argwidth=\wd0
 \setbox0=\hbox{$\left[\box0\right]$}
  \advance\argwidth by -\wd0
 \left[\kern.3\argwidth\box0 \kern.3\argwidth\right]}

\newcommand{\D}{\mathrm{D}}

\newcommand{\eE}{\mathcal{E}}

\newcommand{\oO}{\mathcal{O}}

\newcommand{\dR}{\mathbf{R}}

\newcommand{\rk}{\mathop{\rm rk}\nolimits}

\newcommand{\cneq}{\mathrel{\raise.095ex\hbox{:}\mkern-4.2mu=}}
\newcommand{\eqcn}{\mathrel{=\mkern-4.5mu\raise.095ex\hbox{:}}}

\newcommand{\DT}{\mathop{\rm DT}\nolimits}

\newcommand{\Sing}{\mathop{\rm Sing}\nolimits}

\makeatletter
 
  \@addtoreset{equation}{section}
\makeatother

\title[{$K$-theoretic  pullbacks for Lagrangians on derived critical loci}]
{$K$-theoretic  pullbacks for Lagrangians \\ on derived critical loci}

\author{Yalong Cao}
\address{Morningside Center of Mathematics, Institute of Mathematics \& State Key Laboratory of Mathematical Sciences, Academy of Mathematics and Systems Sciences, Chinese Academy of Sciences, 55 Zhongguancun East Road, 100190, Beijing, China}
\email{yalongcao@amss.ac.cn}
\author{Yukinobu Toda}
\address{Kavli Institute for the Physics and Mathematics of the Universe (WPI), The University of Tokyo Institutes for Advanced Study, The University of Tokyo, Kashiwa, Chiba 277-8583, Japan}
\email{yukinobu.toda@ipmu.jp}
\author{Gufang Zhao}
\address{School of Mathematics and Statistics, University of Melbourne, Parkville VIC 3010, Australia}
\email{gufangz@unimelb.edu.au}


\begin{document}

\maketitle

\begin{abstract}

Given a regular function $\phi$ on a smooth stack, and a $(-1)$-shifted Lagrangian $M$ on the derived 
critical locus of $\phi$, under fairly general hypotheses,  we construct a pullback map from the Grothendieck group of coherent matrix factorizations of 
$\phi$ to that of coherent sheaves on $M$. This map satisfies a functoriality property with respect to the composition of Lagrangian correspondences, as well as the usual bivariance and base-change properties.  

We provide three applications of the construction, one in the
definition of quantum $K$-theory of critical loci (Landau-Ginzburg models), paving the way to  generalize works of 
Okounkov school from Nakajima quiver varieties to  quivers with potentials,
one in establishing a degeneration formula for $K$-theoretic Donaldson-Thomas theory of local Calabi-Yau 4-folds, the other
in confirming a $K$-theoretic version of Joyce-Safronov conjecture.

\end{abstract}

\setcounter{tocdepth}{1}
\hypersetup{linkcolor=black}
\tableofcontents

\section{Introduction}





A \textit{Landau-Ginzburg model} is a pair 
\begin{align*}
(X, \phi_X), \quad \phi_X \colon X \to \mathbb{C}
\end{align*}
of a smooth stack $X$ over $\mathbb{C}$ and a regular function $\phi_X$ on it. 
They are ubiquitous in enumerative geometry, geometric representation theory, singularity theory and mathematical physics, where one studies the critical locus $\mathrm{Crit}(\phi_X) \subset X$ of the function $\phi_X$. Many different considerations lead to the same notion of the \textit{category of} $B$-\textit{branes} in this setup, that of the \textit{category} $\mathrm{MF}(X, \phi_X)$ 
of \textit{matrix factorizations} \cite{Orl, Orl2, BFK1}. Its objects 
consist of \textit{matrix factorizations}:
  \begin{align}\notag
\xymatrix{
\mathcal{E}_{-1} \ar@/^8pt/[r]^-{d_0} &  \ar@/^8pt/[l]^-{d_{-1}} \mathcal{E}_0,  
} \quad
d_0 \circ d_{-1}=\cdot \phi_X, \,\,\,
d_{-1} \circ d_0=\cdot \phi_X,
\end{align}
where $\mathcal{E}_{-1}, \mathcal{E}_0$ are coherent sheaves on $X$. 

An example of such is
the theory of gauged linear sigma models (GLSM),  a curve counting theory for critical loci in GIT quotients, invented by Witten \cite{Witten} and studied from mathematical point of view intensively \cite{FJR, CFGKS, KL1, TX, FK}, which uses
matrix factorizations or their Chern characters as \textit{states} of the theory \cite{PV1, PV2}.
It provides powerful tools for computing Gromov-Witten invariants of Calabi-Yau 3-folds by linking them to Fan-Jarvis-Ruan-Witten invariants \cite{FJRW} 
of Landau-Ginzburg models with isolated singularities. 

Another example of matrix factorizations in Landau-Ginzburg models occurs in Donaldson-Thomas theory of (local) Calabi-Yau 3-folds, where the category of 
matrix factorizations and its 
$K$-theory or periodic cyclic homology provide categorified Donaldson-Thomas theory \cite{KS, Oko, Toda:localsurfaceZ/2, Toda2, PT}, which 
links to geometric representation theory via Hall algebras~(e.g.~\cite{DM, Pa, VV}), 
making hidden symmetries of invariants manifest.

As explored in \cite{CZ, CZZ},  definitions of both  GLSM invariants and  more recently studied  
relative Donaldson-Thomas invariants of log Calabi-Yau 4-folds,  fit into the setting of shifted symplectic geometry 
of Pantev-To\"en-Vaqui\'e-Vezzosi \cite{PTVV, CPTVV}.
Roughly speaking, both invariants can be interpreted as \textit{partition functions} associated to \textit{shifted Lagrangians} on \textit{derived critical loci} or more general $(-1)$-shifted symplectic targets. 

The main player in this paper is the Grothendieck group of the category $\mathrm{MF}(X, \phi_X)$:
\begin{align*}
K_0(X, \phi_X):=K_0(\mathrm{MF}(X, \phi_X)),
\end{align*}
called \textit{critical K-theory}. 
Given a flat map $\pi_X: B\to X$ with $\phi=\phi_X\circ \pi_X: B\to \C$, and a 
shifted Lagrangian 
$\textbf{\emph{M}} \to \mathbf{Crit}(\phi)$ (with fairly general hypotheses), by combining 
the method of matrix factorizations \cite{Orl, Orl2, BFK1, EP}, 
the theory of compact generators of ind-coherent sheaves \cite{DG, Gai}, 
and techniques recently developed from 
DT theory of Calabi-Yau 4-folds \cite{OT, Park1,Park2},  
we show the existence of a natural map:
\begin{align}\label{map:intro:crit}
      K_0(X, \phi_X) \to K_0(M, \mathbb{Z}[1/2]),
\end{align}
which we call \textit{critical pullback}.
This map is 
applicable to both examples mentioned above. 
When $M$ is (equivariantly) proper, by composing \eqref{map:intro:crit} with (equivariant) pushforward to a point, we obtain $K$-\textit{theoretic partition functions} on state spaces. 

The construction of 
critical pullbacks is important as it yields the following generalizations and applications: 
\begin{itemize}
    \item When $X$ is a point, the map \eqref{map:intro:crit} reproduces the $K$-theoretic virtual structure sheaves of 
$(-2)$-shifted symplectic derived schemes due to Oh-Thomas~\cite{OT}.
\item The map \eqref{map:intro:crit} gives a descent of Park's square root virtual pullback \cite[Thm.~5.2.2]{Park2}:  
\begin{align*}
    K_{0}(Z(\phi_X)) \to K_{0}(M,\mathbb{Z}[1/2])
\end{align*}
from the Grothendieck group of coherent sheaves on the zero locus $Z(\phi_X)$ to the critical $K$-theory (see Proposition~\ref{intro prop:compare_OT}). 

\item The map \eqref{map:intro:crit} is a $K$-theoretic
analogue of Joyce conjecture for DT perverse sheaves: 
$$H(X,\varphi_{\phi_X}\mathbb{Q})\to H(M),$$
see~\cite{Joy}, \cite[Rmk.~1.6]{CZZ}.
\item By taking the dual of the map \eqref{map:intro:crit}, we confirm  
a $K$-theoretic version of Joyce-Safronov conjecture (Theorem~\ref{cor:JSconj}). 
Here Joyce-Safronov conjecture roughly says that 
the category $\mathrm{MF}(X, \phi_X)$ may be regarded as a `Fukaya category' for 
shifted Lagrangians over $\textbf{Crit}(\phi)$. 

\item Using \eqref{map:intro:crit}, we define a linearization functor 
on the category of $(-1)$-shifted Lagrangians on derived critical loci (see Definition \ref{def of linear func}, Theorem \ref{thm:functor}).
\item Applied to examples above, we  provide foundations for 
 quantum critical $K$-theory (see Theorem~\ref{intro thm on glue in GLSM}), and establish a degeneration formula of $K$-theoretic Donaldson-Thomas invariants on local
Calabi-Yau 4-folds (see Theorem~\ref{intro thm on glue in DT4}). 

\end{itemize}

In what follows, we give a more detailed summary of results obtained in 
this paper.

\subsection{Definition of critical pullbacks}\label{subsec:intro_def}

Let $B$ be a smooth QCA\footnote{This means quasi-compact with affine stabilizer. } stack over $\mathbb{C}$ with a flat 
map $\pi_X\colon B\to X$ to a smooth global quotient stack $X$, and $\phi_X\colon X\to \C$ be a flat map with 
$\phi:=\phi_X\circ \pi_X\colon B\to \C$. We also regard $B$ as a derived stack by the inclusion functor. 

Consider a map of derived stacks 
\begin{equation}\label{intro equ on fmb}\textbf{\emph{f}}: \textbf{\emph{M}}\to B, \end{equation}
which has a  $(-2)$-\textit{shifted symplectic structure} $\Omega_{\textbf{\emph{f}}}$ 
such that the image of $\Omega_{\textbf{\emph{f}}}$ in the periodic cyclic homology is given 
by the function ${\phi}$. By a tautological argument (see \cite{Park2} or Proposition \ref{thm on equi of two}), this is equivalent to the structure of an exact Lagrangian 
$$\textbf{\emph{M}}\to \textbf{Crit}(\phi)$$
on the derived critical locus. As the title suggests, we shall construct a $K$-theoretic pullback map. 
Assume that $\textbf{\emph{f}}$ satisfies hypotheses in Setting \ref{setting of lag}. 
In particular, 
denote $\bbL_{\textbf{\emph{f}}}\,|_{M}$ to be the restriction of the cotangent complex of $\textbf{\emph{f}}$ to the classical truncation $M=t_0(\textbf{\emph{M}})$, 
there is a resolution
\begin{equation}\label{intro sym reso}\bbL_{\textbf{\emph{f}}}\,|_{M}\cong (V\xrightarrow{d}  E \xrightarrow{d^\vee}  V^\vee), \end{equation}
given by a symmetric complex of finite rank vector bundles, 
where $E$ is a quadratic bundle on $M$ with a non-degenerate quadratic form $q_E$.

The main construction of this paper (ref.~Definition \ref{defi of crit pb}) is the following pullback map. 
\begin{theorem}\label{intro:thm}
There is a well-defined map (called \textit{critical pullback}): 
\begin{align}\label{intro equ on crit pb der}f_{\pi_X}^!: K_0(X,\phi_X)\xrightarrow{\sigma_{\pi_X\circ f}} K_0(E,q_E)\xrightarrow{\Spin}K_0(M,\Z[1/2]) 
\xrightarrow{\otimes\sqrt{\det(V)^\vee}}  K_0(M,\Z[1/2]), \end{align} 
from the Grothendieck group $K_0(X,\phi_X)$ of the category of \textit{coherent matrix factorizations} 
of $(X,\phi_X)$ to the Grothendieck group of coherent sheaves on $M$.
\end{theorem}
We briefly explain the construction. 
Let $Z(\phi_X)$ denote the zero locus of $\phi_X$. There is a \textit{surjective canonical map}:
\begin{equation}\label{intro equ on kgps}can: K_0(Z(\phi_X)) \twoheadrightarrow  K^{}_0(X,\phi_X), \end{equation}
whose kernel is generated by perfect complexes on $Z(\phi_X)$ (Proposition \ref{prop:ab_vs_D}). 
One has a standard \textit{specialization map} for the zero locus:
\begin{equation}\label{eqn on sp of zr}K_0(Z(\phi_X))\to  K_0(Z(q_E)), \end{equation}
defined via deformation to the normal cone \eqref{sp map1}. 
We show \eqref{eqn on sp of zr} factors through \eqref{intro equ on kgps} and hence descends to a map $\sigma_{\pi_X\circ f}$ in \eqref{intro equ on crit pb der}. 
Such a descent is remarkable as there is in general no specialization map from open to closed for critical $K$-theories
(see Remark \ref{rmk on sp of cri ex} for a counterexample). To show the descent, 
we introduce \textit{relative singularity categories} \eqref{sp map2}  following ideas of Efimov-Positselski \cite{EP},
and compare them with the \textit{matrix factorization} category of $(E,q_E)$ (Proposition \ref{prop on compare sing and mf}). We use the theory of compact generators of ind-coherent sheaves \cite{DG, Gai} to handle \textit{stackyness} and \textit{singularity} of $M$ (and hence $E$) in this comparison. The details are in~\S \ref{sect on spe3}--\S \ref{sect on spe1}.
The map $\Spin$ in \eqref{intro equ on crit pb der} is defined in \S \ref{sec:spin}, which, roughly speaking, is given by the tensor product with 
 \textit{Koszul factorization} \eqref{equ on kos fa}.
The last map is the twisting by a square root line bundle,  
 making the map \eqref{intro equ on crit pb der} independent of the choice of the symmetric resolution \eqref{intro sym reso}. 

We remark that for applications considered in this paper, it is necessary to distinguish $B$ and $X$ in the setup of \S \ref{subsec:intro_def}, where the former is a very general smooth stack and the latter is a smooth global quotient stack, 
because our interested shifted Lagrangians appear on derived critical loci of functions on $B$, while  the fact that the functions are pulled back from a global quotient stack $X$ makes comparison between matrix factorization category and  singularity category more natural.


\subsection{Properties of critical pullbacks}
Let $f_Z$ be the base change of $f$ to the zero locus of $\phi$ \eqref{diag on MZPHI}.
Consider the composition 
$$\sqrt{f_Z^!}\circ \pi_X^*: K_0(Z(\phi_X))\xrightarrow{\pi_X^*} K^{}_0(Z(\phi)) \xrightarrow{\sqrt{f_Z^!}}K_0(M,\Z[1/2]), $$
where $\pi_X^*$ is the flat pullback and $\sqrt{f_Z^!}$ is the \textit{square root virtual pullback} \eqref{equ on k sqr pb} \cite{Park1}.
It is compatible with critical pullback  \eqref{intro equ on crit pb der} in the following sense:
\begin{prop}[Proposition \ref{prop:compare_OT}]
\label{intro prop:compare_OT}

We have a commutative diagram 
$$
\xymatrix{
K_0(Z(\phi_X))    \ar@{->>}[r]^{can}   \ar[rd]_{\sqrt{f_Z^!}\circ\pi_X^*\,\,\,\, }    &  K_0(X,\phi_X) \ar[d]^{f_{\pi_X}^!} \\
&  K_0(M,\Z[1/2]),
}
$$
where the horizontal map is the canonical map \eqref{intro equ on kgps}.
The critical pullback map is uniquely characterized by this property, due to the surjectivity of the canonical map.
\end{prop}

When $\phi_X \equiv 0$, the obstruction theory on $f$ is \textit{isotropic} (e.g.~\cite[Rmk.~5.2.5]{Park2}), the map $can$ is an isomorphism, and critical pullback $f_{\pi_X}^!$  coincides with that of \cite{Park1, OT}. In this sense, $f_{\pi_X}^!$ is a generalization of  {\it loc.\,cit.} to a  map $f$ that does not necessarily satisfy the isotropic condition.



We remark that the \textit{advantage} of considering $K_0(X,\phi_X)$ over $K_0(Z(\phi_X))$ is that the former
has a \textit{natural pairing} (Lemma\,\ref{lem on pairing}) which is \textit{non-degenerate} in many cases (Examples \ref{ex on nak}, \ref{ex on hilb3}), 
while the latter seldom does (unless $\phi_X\equiv 0$). 
The existence of a perfect pairing fits into the framework of \textit{cohomological field theories} \cite{KM} 
(or their $K$-theoretic variants),
making powerful computational tools \cite{Giv, Tel} possible to many examples \cite{P}. It is also pivotal for applications to geometric representation theory. 

By using Proposition \ref{intro prop:compare_OT}, we show that \eqref{intro equ on crit pb der} is well-defined (Proposition~\ref{prop:indep}), has natural bivariance properties (Proposition \ref{pullback comm with base change}) in intersection theory, and the following \textit{functorial} property.
\begin{theorem}
[Theorem~\ref{thm on funct}]
Given a commutative diagram of derived stacks: 
\begin{equation}
   \label{eqn on comm diag func}
   \xymatrix{
\textbf{{N}}\,' \ar[r]^{\textbf{{i}}} \ar[d]_{\textbf{{f}}}     &  \textbf{{N}} \ar[d]^{\textbf{{h}}} \\
\textbf{{M}} \ar[r]^{\textbf{{g}}} & \textbf{{B}},  }
\end{equation}
such that 
\begin{itemize}
\item  $\textbf{{g}}$ and $\textbf{{h}}$ are as \eqref{intro equ on fmb} for a common $\pi_X: B\to X$ and $\phi_X: X\to \C$,
\item $\textbf{{f}}$ is quasi-smooth and the classical truncation of $\textbf{{i}}$ is an isomorphism,
\item the obstruction theories are compatible $($see diagram \eqref{eqn:triang_obst} and Remark \ref{rmk:lag_cl}\,$)$,
\end{itemize}
then we have   
\begin{equation}\label{intro equ on fun of crit pb}h_{\pi_X}^!=f^!\circ g_{\pi_X}^!: K_0(X,\phi_X)\to K_0(N,\Z[1/2]), \end{equation}
where $g_{\pi_X}^!,h_{\pi_X}^!$ are critical pullbacks \eqref{intro equ on crit pb der} and $f^!: K_0(M)\to K_0(N)$ is the virtual pullback \cite{Qu}. 
\end{theorem}
The above functorial property is used to show how critical pullbacks behave under the composition of shifted Lagrangian correspondences (see Theorem \ref{thm:functorial}). 
As an application of this, we obtain a linearization functor 
on the category of $(-1)$-shifted Lagrangians on derived critical loci (Definition \ref{def of linear func}, Theorem \ref{thm:functor}).
It is also used in the proof of gluing formulae in examples mentioned below. 
 
\subsection{Applications}

The main examples considered in this paper are given by the theory of GLSM and Donaldson-Thomas theory of (local) Calabi-Yau 4-folds (Example \ref{ex of papers}).
We sketch the main results and  refer \S \ref{sec:appl} for details. 

\begin{theorem}
[Quantum critical $K$-theory, Theorem \ref{thm on glue in GLSM}]
\label{intro thm on glue in GLSM}
Consider a smooth GIT quotient
$W/\!\!/ G$ and a regular function $\phi_{W/\!\!/ G}: W/\!\!/ G\to \C$, which is invariant under a torus $T$-action, such that the torus fixed critical locus 
$\mathrm{Crit}(\phi_{W/\!\!/ G})^T$ is proper.
Then for any $g,n\geqslant 0$ $($such that $2g-2+n>0$$)$ and a curve class $\beta$,
there is a map 
$$\Phi^{\Crit,K}_{g,n,\beta}: K_0(\overline{\mathcal{M}}_{g,n})\otimes K_0^T\left((W/\!\!/ G)^n,\phi_{W/\!\!/ G}^{\boxplus n}\right)_{ } \to K_0^T\left(\mathrm{pt}\right)_{\loc}, $$
which satisfies the gluing formula 
\begin{align} 
\nonumber
&\quad \,\, \Phi^{\Crit,K}_{g,n,\beta}(\iota_*\alpha\boxtimes\gamma)\\  \nonumber
&=\sum_{k=0}^{\infty}(-1)^k\sum_{\beta_0+\beta_1+\cdots+\beta_k+\beta_\infty=\beta}\Phi^{\Crit,K}_{g_1,n_1+1,\beta_0}\otimes \Phi^{\Crit,K}_{0,2,\beta_1}\otimes \cdots \otimes\Phi^{\Crit,K}_{0,2,\beta_k}\otimes \Phi^{\Crit,K}_{g_2,n_2+1,\beta_{\infty}}(\alpha\boxtimes\gamma\boxtimes\oO_{\Delta^{k+1}}),  
\end{align}
where $\overline{\mathcal{M}}_{g,n}$ is the moduli stack of genus $g$, $n$-pointed stable curves,
$\iota: \overline{\calM}_{g_1,n_1+1}  \times \overline{\calM}_{g_2,n_2+1} \to \overline{\calM}_{g,n}$ is a gluing map, and the RHS of the equation is defined by \eqref{map phi tens}.
\end{theorem}
\begin{remark}
This theorem allows us to extend quantum $K$-theory from smooth varieties \cite{Lee} to the setting of \textit{critical loci}, which are possibly singular.
By virtue of this, one can lift the Okounkov package (e.g.~quantum connections, difference equations, Bethe ansatz, etc) from Nakajima quiver varieties \cite{Oko} to more general quivers with potentials \cite[\S 6-\S8]{CZ}. They have applications to geometric representation theory,~e.g.~\cite{BR, LY, RSYZ1, RSYZ2, Pa, VV}, along the line discussed in \cite{Oko, CZ}.  
\end{remark}

\begin{theorem}
[Degeneration formula of $K$-theoretic $\DT_4$ invariants, Theorem \ref{thm on glue in DT4}]
\label{intro thm on glue in DT4}
Let $X\to \mathbb{A}^1$ be a simple degeneration of local Calabi-Yau 4-folds with torus $T$-action, such that the central fiber 
$$X_0=Y_-\cup_D Y_+ $$
is the union of two (local) log Calabi-Yau pairs $(Y_{\pm},D)$ \eqref{equ on log cy}. 
Fix numerically equivalent $K$-theory classes $P_t \in K_{c,\leqslant 1}^\mathrm{{num}} (X/\mathbb{A}^1)$ for all $t\in \bbA^1$.
Assume conditions in Lemma \ref{lem on exi of reso2} hold for $(Y_{\pm},D)$.

Then there is a ``family class" and ``relative classes":
$$\Phi^{P_t}_{X/\bbA^1}; \quad \Phi_{Y_{-},D}^{Q_{0}}, \quad \Phi_{Y_{+},D}^{Q_{\infty}}, \quad \Phi_{\Delta^{\sim}}^{Q_i}, $$
such that the following gluing formula holds
\begin{align*} i_0^!(\Phi^{P_t}_{X/\bbA^1})=\sum_{k=0}^{\infty}(-1)^k\sum_{\begin{subarray}{c}\delta \\ k(\delta)=k \\ P(\delta)=P_0   \end{subarray}} \Phi_{Y_{-},D}^{Q_0}\otimes \Phi_{\Delta^{\sim}}^{Q_1}\otimes \cdots \otimes\Phi_{\Delta^{\sim}}^{Q_k}\otimes \Phi_{Y_{+},D}^{Q_\infty}(\oO_{\Delta^{k+1}}).
\end{align*}
\end{theorem}
\begin{remark} We comment on invariants obtained in Theorems~\ref{intro thm on glue in GLSM} and \ref{intro thm on glue in DT4}.
\begin{enumerate}
    \item A $K$-theoretic analogue of \cite[Def.~5.2]{CZ} defines a map 
$$K_0(\overline{\mathcal{M}}_{g,n})\otimes K_0^T\left(Z(\phi_{W/\!\!/ G}^{\boxplus n})\right)_{ } \to K_0^T\left(\mathrm{pt}\right)_{\loc}. $$
By Proposition \ref{intro prop:compare_OT}, it descends to the map in Theorem \ref{intro thm on glue in GLSM} through the canonical map. Similarly for the $K$-theoretic $\DT_4$ case (Theorem \ref{intro thm on glue in DT4} and \cite[Def.~5.9]{CZZ}).
    \item Explicit calculations of these invariants have been performed in several examples. For instance, {\it vertex functions} have been explicitly calculated in \cite[\S 7-\S 8]{CZ}, which consequently yields {\it Bethe equations}, demonstrating the nontriviality of the map \eqref{intro equ on crit pb der}\footnote{The calculation in \cite{CZ} is done in the cohomological case, it is straightforward to perform a $K$-theoretic one. In fact, \cite[Thm.~0.12]{CKM} shows that the $K$-theoretic invariant recovers the cohomological one by taking certain limit.}.
\end{enumerate}
\end{remark}

\subsection{$K$-theoretic Joyce-Safronov conjecture}
In \cite[Conj.~1.2]{JS}, Joyce-Safronov proposed the following conjecture. 
\begin{conj}
    Let $U$ be a smooth $\mathbb{C}$-scheme and $\phi \colon U \to \mathbb{C}$ be a 
    regular function. Let $\textbf{\textit{M}} \to \textbf{Crit}(\phi)$ be a $(-1)$-shifted Lagrangian
    such that $(\textbf{vdim}\,\textbf{\textit{M}}-\dim U)$ is even, with an {\it orientation data} and {\it spin structure}. 
    Then there exists an object $\mu_{\textbf{\textit{M}}} \in \mathrm{MF}(U, \phi)$ in the matrix factorization category $\mathrm{MF}(U, \phi)$, associated with $\textbf{\textit{M}}$, with prescribed local behavior spelled out in \cite{JS}.
\end{conj}
The above conjecture is important as it gives an interpretation of the matrix factorization
category $\mathrm{MF}(U, \phi)$ as a `Fukaya category' for $(-1)$-shifted symplectic 
derived schemes, and is also connected to the programme of Kapustin and Rozansky \cite{KR} for associating 2-categories to holomorphic symplectic manifolds, locally described using matrix factorization categories.
Note also that $\mathrm{MF}(U, \phi)$ is regarded as a local model of 
categorical DT theory for Calabi-Yau 3-folds~\cite{Toda:localsurfaceZ/2}, thus if global DT categories 
are constructed for more general $(-1)$-shifted symplectic derived schemes, we expect 
a similar statement there. 

As an application of Theorem~\ref{intro:thm}, we solve the following $K$-theoretic version of 
Joyce-Safronov conjecture. 
\begin{theorem}\label{cor:JSconj}
    Let $(U, \phi)$ be a smooth quasi-projective  scheme and a regular function with a torus $T$-action such that Setting~\ref{set of perf pair} holds, 
    and $\textbf{M} \to \mathbf{Crit}(\phi)$
    be a $T$-equivariant $(-1)$-shifted oriented exact Lagrangian
    such that the classical truncation $M$ is $T$-equivariantly proper, $M\to \Crit(\phi)$ is quasi-projective. 
    Then there is an associated element $$[\mu_\textbf{M}] \in K_0^T(U, \phi)_{\loc}. $$ 
    Here $(-)_{\loc}:=(-)\otimes_{K_0^T(\mathrm{pt})}K_0^T(\mathrm{pt})_{\loc}$,
    and $K_0^T(\mathrm{pt})_{\loc}$ is the field of fractions of $K_T(\mathrm{pt})$.  
\end{theorem}

Indeed under Setting~\ref{set of perf pair}, the natural pairing 
$\langle -, -\rangle$ on $K_0^T(U, \phi)_{\loc}$ is perfect. 
Therefore the element $[\mu_\textbf{\emph{M}}]$ is defined to be the unique element 
such that the pairing 
$\langle [\mu_\textbf{\emph{M}}], -\rangle$ equals to the composition 
\begin{align*}
    K_0^T(U, \phi)_{\loc} \stackrel{f_{\id}^!}{\to} K_0^T(M)_{\loc} \to K_T(\mathrm{pt})_{\loc},
\end{align*}
where the first map is the $T$-equivariant version of the map $f_{\id}^!$ in Theorem~\ref{intro:thm} and the second map is the equivariant push-forward to a point. 
In the above setting, all hypotheses in Setting \ref{setting of lag} are satisfied (where we take $\pi_X=\id$ to be the identity), so $f_{\id}^!$ is well-defined.


\subsection*{Related works and future directions}

In the application to GLSM, the construction of Theorem~\ref{intro thm on glue in GLSM} is expected to agree with the $K$-theoretic version of the work of Favero and Kim \cite{FK}. It is not clear, albeit interesting, to what extent the construction of \cite{FK}, which depends on the construction of global charts as well as sections and co-section of  obstruction bundles, can be carried out in the general case of Setting \ref{setting of lag}. For the same reason, although it is reasonable to expect invariants from \cite{FK} to be compatible with those from \cite{CZ} (see \cite[\S 1.5]{CZ}), a detailed comparison does not appear to be easy. We believe methods from the present paper will shade lights on this, which we hope to explore more in the future. 

There is a current project of gluing matrix factorization categories (or their variants) over $(-1)$-shifted symplectic target \cite{HHR}. 
It is an interesting question to construct critical pullbacks in this global setting. 
A further lift of the problem to the categorical level looks also natural to study.

\subsection*{Notations and Conventions}
Unless otherwise specified, all stacks are algebraic (Artin) stacks (over $\C$) as defined in the Stack Project \cite[Def.~Tag 026O]{Sta}, where no separation is imposed on the definition. 
All derived stacks are homotopically finitely presented derived Artin stacks \cite{Toen}.

\subsection*{Acknowledgments}

We are grateful to Jack Hall, Young-Hoon Kiem, Tasuki Kinjo, Henry Liu, Hiraku Nakajima, Andrei Okounkov, Hyeonjun Park, Feng Qu, Yongbin Ruan, Yang Zhou, Yehao Zhou, Zijun Zhou for helpful discussions and communications during the preparation of this work.
We particularly thank Hyeonjun Park for generously sharing the draft of \cite{Park2} in an early stage. 

During the  finalization stage of the present paper, we learned about the work in progress of Choa, Oh, and Thomas, which addresses a similar question of specialization to the normal cone in the critical $K$-theory. The authors thank Jeongseok Oh for interesting discussions on it.  
The construction of \textit{loc.\,cit.}\, on a local chart defines the specialization of a locally free factorization as a limit of a Fulton-MacPherson graph in a Grassmannian bundle. On a global $(-1)$-shifted Lagrangian, the resolution property of the Lagrangian is needed to globalize the construction (compared to the Setting \ref{setting of lag}).
We expect a comparison of the contruction in \textit{loc.\,cit.}\,with \eqref{equ on spe map on k} to be
 interesting, which on one hand simplifies the deduction of properties of \textit{loc.~cit.}, and on the other hand gives a geometric description of the effect of \eqref{equ on spe map on k} in terms of locally free factorizations.
 
G.~Z.~is partially supported by the Australian Research Council via DE190101222 and DP210103081. 
Y.~T.~is supported by World Premier International Research Center Initiative (WPI initiative), MEXT, Japan, and JSPS KAKENHI Grant Numbers JP24H00180.
Y.~T.~is also supported by the Distinguished Visiting Professor Program at 
 the Morningside Center of Mathematics, Chinese Academy of Sciences. 


\section{Exact Lagrangians on derived critical loci}
We collect some basic notations regarding Lagrangians on derived critical loci, and spell out the setup under which the present paper is concerned. We refer to Appendix~\ref{app on sympl} for the basic notions on shifted symplectic structures. 

\subsection{Exact Lagrangians and relative shifted symplectic structures}

\begin{definition}\label{defi of derived crit loci}
Let $\textbf{\emph{B}}$ be a derived stack with a function $\boldsymbol{\phi}: \textbf{\emph{B}}\to \C$ and 
$\textbf{Crit}(\boldsymbol{\phi})$ be the \textit{derived critical locus} defined by the homotopy pullback diagram of derived stacks: 
\begin{equation}\label{equ on sc cl}\begin{xymatrix}{
\bCrit^{}(\boldsymbol{\phi})  \ar[r]^{ } \ar[d]^{ } \ar@{}[dr]|{\Box} &\textbf{\emph{B}} \ar[d]^{d\boldsymbol{\phi}}  \\
\textbf{\emph{B}} \ar[r]^{0  \,\,} & \bbT^*\textbf{\emph{B}},
}\end{xymatrix}\end{equation}
where $\bbT^*\textbf{\emph{B}}$ is the cotangent bundle stack of $\textbf{\emph{B}}$. 

The derived critical locus is endowed with a $(-1)$-shifted symplectic structure $\Omega_{\bCrit^{}(\boldsymbol{\phi})}$ \cite[Thm.~2.9]{PTVV}. It has a canonical  \textit{exact} structure \cite[Ex.~3.1.2]{Park2},~i.e.~ a nullhomotopy 
\begin{equation}\label{equ on ex sym}  [\Omega_{\bCrit^{}(\boldsymbol{\phi})}]\sim 0 \end{equation}
of the image $[\Omega_{\bCrit^{}(\boldsymbol{\phi})}]$ of $\Omega_{\bCrit^{}(\boldsymbol{\phi})}$ in the periodic cyclic homology. 
\end{definition}

\begin{definition}
\label{def of exact lag}
Let $\textbf{\emph{M}}\to \textbf{Crit}(\boldsymbol{\phi})$
be a Lagrangian  with  defining nullhomotopy 
$$\Omega_{\bCrit^{}(\boldsymbol{\phi})}|_{M}\sim 0. $$ 
An \textit{exact} structure on the Lagrangian is a contraction of the loop in 
the periodic cyclic homology of $M$, where the loop is formed by  the image of the Lagrangian structure
$[\Omega_{\textbf{Crit}^{}(\boldsymbol{\phi})}|_{M}]\sim 0$ and the pullback to $M$ of the exact structure \eqref{equ on ex sym}.  
\end{definition}
By compositing with the inclusion $\textbf{Crit}(\boldsymbol{\phi})\to \textbf{\emph{B}}$, we obtain a map 
\begin{equation}\label{equ on -2 symp} \textbf{\emph{M}}\to \textbf{\emph{B}},  \end{equation}
which has a (relative) $(-2)$-shifted symplectic structure.
\begin{prop}\label{thm on equi of two}
\cite[Cor.~3.1.3]{Park2}
Let $\textbf{M}$ be a derived stack over $\textbf{B}$. Then there exists a canonical equivalence between spaces of the following objects: 
\begin{itemize}
\item relative $(-2)$-shifted symplectic structures $\Omega_{\textbf{M}/\textbf{B}}$ with $[\Omega_{\textbf{M}/\textbf{B}}]=\boldsymbol{\phi}|_{\textbf{M}}$ with $\boldsymbol{\phi}: \textbf{B}\to \C$,
\item morphisms $\textbf{M}\to \textbf{\emph{Crit}}^{}(\boldsymbol{\phi})$ over $\textbf{B}$ with exact Lagrangian structures. 
\end{itemize}
\end{prop}

\begin{remark}
    The relative $(-2)$-shifted symplectic form $\Omega_{\textbf{\emph{M}}/\textbf{\emph{B}}}$ 
    in Proposition~\ref{thm on equi of two} is called \textit{$\boldsymbol{\phi}$-locked form} 
    in~\cite{Park2}. 
\end{remark}

\subsection{Setting of the paper}
Recall \cite[Def.~1.1.8]{DG} that a derived stack $S$ is \textit{QCA} if it is quasi-compact with affine stabilizer, and the classical inertia stack is of finite presentation over 
the classical truncation of $S$. 

\begin{setting}\label{setting of lag}
In the present paper, we consider the following data.
\begin{enumerate}
\item (LG) Let $B$ be a smooth QCA classical stack with a flat morphism 
$\pi_X:B\to X$ to the quotient stack $X=[S/H]$ of a smooth quasi-projective scheme $S$ by a linear algebraic group $H$. 
Let $\phi_X: X\to \C$ be a flat regular function\,\footnote{We assume the classical critical locus $\Crit(\phi_X)\hookrightarrow Z(\phi_X)$  embeds into 
 the zero locus $Z(\phi_X)$. } with pullback function $\phi=\phi_X\circ \pi_X: B\to \C$.
\item (Oriented Lagrangian) A map of derived stacks
\begin{equation}\label{equ on -2 symp}\bbf: \textbf{M} \to B, \end{equation}
which has a (relative) $(-2)$-shifted symplectic structure $\Omega_{\bbf}$, and a relative orientation in the sense of Definition \ref{ori on even cy}. 
\end{enumerate}
The above data are assumed to satisfy the following constraints.
\begin{enumerate}
\item (Compatibility) The class
$[\Omega_{\bbf}]$ in periodic cyclic homology is given by ${\phi}|_{\textbf{M}}:={\phi}\circ \textbf{f}$. 
    \item (DM) The classical truncation $f: M\to B$ of $\bbf$ is a quasi-projective Deligne-Mumford morphism and $M=t_0(\textbf{M})$ is Noetherian.
     \item  (Resolution) The restriction $\bbL_{\bbf}\,|_{M}$ of the cotangent complex to the classical truncation 
     has a resolution by symmetric complex  
     $(V\xrightarrow{d}  E \xrightarrow{d^\vee}  V^\vee)$ of finite rank bundles, where $E$ is a 
     quadratic bundle with a non-degenerate quadratic form $q_E$.
\end{enumerate}
\end{setting}

In \S \ref{sec:linearization}, we further assume properness (or equivariant properness) of $M$ to do integration. This is not needed before that section.

\begin{example}\label{ex of papers}
The main examples considered in this paper are as follows: 
 \begin{enumerate}
\item
(Quasimaps to critical loci): 
Let $W/\!\!/ G$ be a smooth GIT quotient of a complex vector space $W$ by a linear algebraic group $G$, 
$\phi_{W/\!\!/ G}: W/\!\!/ G\to \C$ be a flat regular function, with derived critical locus $\bCrit(\phi_{W/\!\!/ G})$. Fix
$n,g\in \mathbb{Z}_{\geqslant 0}$ such that $2g-2+n>0$, let $\fBun_{H_R,g,n}^{R_{\chi}=\omega_{\mathrm{log}}}$ be the stack of principal $(H_R:=G\times \mathbb{C}^*)$-bundles $P$ on genus $g$, $n$-pointed prestable curves together with twists given by isomorphisms $P/G\times_{\C^*}R_{\chi}\cong \omega_{\mathrm{log}}$, 
where $R_{\chi}$ is a character of $\C^*$.
A \textit{curve class} is the degree of the principal $G$-bundle $P_G:=P/\mathbb{C}^*$:
$$\beta\in \Hom_{\mathbb{Z}}(\bbX(G), \bbZ), \quad \beta(\xi):=\deg_C(P_G\times_G \bbC_{\xi}). $$
For each $\beta$, let $\textbf{QM}_{g,n}^{R_{\chi}=\omega_{\mathrm{log}}}(\bCrit(\phi_{W/\!\!/ G}),\beta)$ be the derived moduli stack of twisted stable quasimaps to $\bCrit(\phi_{W/\!\!/ G})$ with 
natural morphism 
$$\emph{\textbf{f}}: \textbf{QM}_{g,n}^{R_{\chi}=\omega_{\mathrm{log}}}(\bCrit(\phi_{W/\!\!/ G}),\beta)\to \fBun_{H_R,g,n}^{R_{\chi}=\omega_{\mathrm{log}}}\times_{[\pt/G]^n}(W/\!\!/ G)^n, $$ 
given by forgetting the sections in quasimaps data and evaluation at marked points of curves. 
We refer to \cite[\S 3]{CZ} for more details about the construction. 


\item
(Relative Hilbert stacks on  local log Calabi-Yau 4-folds): Let $(Y,D)$ be a local log Calabi-Yau 4-fold \eqref{equ on log cy}, and $P$ be some 
numerical $K$-theory class on $Y$, with restriction $P|_D$ to $D$. Assume the derived Hilbert scheme $\textbf{Hilb}^{P|_D}(D)$ is given by the 
derived critical locus of a function on a smooth GIT quotient $W/\!\!/ G$ as in above. 
We take   
$$\emph{\textbf{f}}: \textbf{Hilb}^P(Y,D) \to  \calA^P\times (W/\!\!/ G)^n, $$
where the target  $\calA^P$ is the stack of $P$-decorated expanded pairs, 
and $n$ is the number of connected components of $D$. The domain is the derived relative Hilbert stack on $(Y,D)$.
We refer  to \cite{CZZ} for more details about the construction. 
\end{enumerate}
\end{example}
In Appendix \ref{app b}, we verify hypotheses in Setting \ref{setting of lag} in  examples above.

\section{Factorization categories and their $K$-theories}

Categories of factorizations have been extensively studied from mathematical point of view after the pioneer work of Orlov \cite{Orl, Orl2}. 
We refer to \cite{BFK1, BFK, EP, PV3} and references therein. 
In this section, we recall the basic notions and several properties of their $K$-theories. 

\subsection{Abelian categories of factorizations}\label{sec on ab fac}
\begin{definition}
Let $\calX$ be an algebraic stack, and 
$\phi$ a regular function on $\calX$. 
A \textit{(quasi-)coherent factorization} of $\phi$ is a quadruple 
$\calE_\bullet=(\calE_0,\calE_{-1},d_{0},d_{-1})$
where $\calE_0$ and $\calE_{-1}$ are (quasi-)coherent sheaves and $d_{-1}:\calE_0\to \calE_{-1}$ and $d_0:\calE_{-1}\to \calE_{0}$ are morphisms of (quasi-)coherent sheaves so that $$d_0\circ d_{-1}=d_{-1}\circ d_0=\phi. $$
We say a factorization $\calE_\bullet$ \textit{locally free} if $\calE_0,\calE_{-1}$ are locally free.
\end{definition}
A \textit{morphism} between two factorizations is a map of pairs of (quasi-)coherent sheaves so that the obvious square commutes. Similarly, the notion of a \textit{homotopy} between two morphisms is the obvious one \cite[\S 2]{BFK}. We denote the \textit{abelian categories} of \textit{coherent and quasi-coherent factorizations} of $(\calX,\phi)$ by 
$$\Fact_{\star}(\calX,\phi), \quad \star=\coh,\,\qcoh.$$ 
We say a factorization $\calE_\bullet$ is {\it contractible} if its identity map is homotopic to zero.

There is  an \textit{exterior product} on categories of factorizations  as \cite[Def.~3.22,~Def.~3.51]{BFK1}: 
\begin{equation}\label{equ on ten prod}
\boxtimes: \Fact_{\coh}(\mathcal{X},\phi)\otimes \Fact_{\coh}(\mathcal{Y},\phi')\to \Fact_{\coh}(\mathcal{X}\times \calY,\phi\boxplus \phi')
\end{equation}
$$(\eE_\bullet,\mathcal{F}_\bullet)\mapsto \eE_\bullet\boxtimes \mathcal{F}_\bullet:=\left((\eE_\bullet\boxtimes\mathcal{F}_\bullet)_0,(\eE_\bullet\boxtimes\mathcal{F}_\bullet)_{-1},
d^{\eE_\bullet\boxtimes \mathcal{F}_\bullet}_0,d^{\eE_\bullet\boxtimes \mathcal{F}_\bullet}_{-1}\right), $$
where 
$$(\eE_\bullet\boxtimes\mathcal{F}_\bullet)_0:=\eE_0\boxtimes \mathcal{F}_0\oplus \eE_{-1}\boxtimes \mathcal{F}_{-1}, $$
$$(\eE_\bullet\boxtimes\mathcal{F}_\bullet)_{-1}:=\eE_{-1}\boxtimes \mathcal{F}_0\oplus \eE_{0}\boxtimes \mathcal{F}_{-1}, $$
\begin{equation*} 
d^{\eE_\bullet\boxtimes \mathcal{F}_\bullet}_0:=  \begin{pmatrix}
d_0^{\eE_\bullet}\boxtimes 1 & 1\boxtimes d_0^{\calF_\bullet}   \\
-1\boxtimes d_{-1}^{\calF_\bullet} & d_{-1}^{\eE_\bullet}\boxtimes 1 \end{pmatrix},
\end{equation*}
\begin{equation*} 
d^{\eE_\bullet\boxtimes \mathcal{F}_\bullet}_{-1}:=  \begin{pmatrix}
d_{-1}^{\eE_\bullet}\boxtimes 1 & -1\boxtimes d_0^{\calF_\bullet}   \\
1\boxtimes d_{-1}^{\calF_\bullet} & d_{0}^{\eE_\bullet}\boxtimes 1 \end{pmatrix},
\end{equation*}
and the sum function is given by 
$$\phi\boxplus \phi': \mathcal{X}\times \calY\to \C, \quad (x,y)\mapsto \phi(x)+\phi'(y).$$

\subsection{Derived categories of factorizations}\label{sect on dev fact}
With the obvious notions of suspension auto-equivalence $[1]$, mapping cone of a morphism, and homotopy between two morphisms, $\Fact_{\star}(\calX,\phi)$ 
($\star=\coh, \qcoh$)
has a \textit{homotopy category} $\mathcal{K}(\Fact_{\star}(\calX,\phi))$, which are triangulated categories.

Starting with a bounded  complex of factorizations, its totalization in the obvious sense gives a single factorization. Let $\mathrm{Acycl}$ be the smallest thick subcategory of 
$\mathcal{K}(\Fact_{\star}(\calX,\phi))$ containing totalizations of all bounded exact complexes from $\Fact_{\star}(\calX,\phi)$. Objects of $\mathrm{Acycl}$ are called \emph{acyclic}.
The Verdier quotient is denoted by 
$$\mathrm{MF}_{\star}(\calX,\phi), $$
called the \textit{matrix factorization category} of $(\calX,\phi)$. 
As we mainly work in the coherent setting, we often use the shorthand 
$$\mathrm{MF}_{}(\calX,\phi):=\mathrm{MF}_{\coh}(\calX,\phi).$$

\begin{definition}\label{defi of supp}
Let $\mathcal{Z}\hookrightarrow \mathcal{X}$ be a closed substack and 
$\mathcal{U}=\mathcal{X}\setminus \mathcal{Z}$ be the complement with an open immersion $j: \mathcal{U}\to \mathcal{X}$.
    We say that a factorization $\calE_\bullet$ is \textit{supported on} $\calZ$ if $j^*\calE_\bullet$ is acyclic on $\calU$. The full subcategory of factorizations supported on $\calZ$ is denoted by $\mathrm{MF} (\calX,\phi)_{\calZ}$.

    We say that $\calE_\bullet$ is \textit{locally contractable} off $\calZ$ if there is a smooth atlas $p:Y\to \calX\setminus\calZ$ so that $p^*\calE_\bullet$ is \textit{contractible} (i.e.~the identity map is homotopic to the zero map). 
    \end{definition}

 Let $\mathcal{X}$ be an algebraic stack with two regular functions $\phi_1$ and $\phi_2$, $\calE_\bullet$ be a factorization of $\phi_1$ and $\calF_\bullet$ a factorization of $\phi_2$. Then there is a tensor $\calE_\bullet\otimes\calF_\bullet$ which is a factorization of $\phi_1+\phi_2$. Here the $d_i$'s on the tensor are defined similar to the differentials on a tensor complex with the same sign convention as \eqref{equ on ten prod}.  
 Notice in particular that the tensor of $\calE_\bullet$ with a morphism between $\calF_\bullet$ and $\calF'_\bullet$ is well-defined; and a homotopy between two morphisms induces a homotopy between the two morphisms on the tensor. 
 The following lemma is standard. 

\begin{lemma}\label{lem:tensor_acyc}
  Let $\mathcal{X}$ be an algebraic stack with two regular functions $\phi_1$ and $\phi_2$ on it,  
  $\calE_\bullet$, $\calF_\bullet$ be a factorization of $\phi_1$ and $\phi_2$ respectively. 
  
  \begin{enumerate}
      \item If $\calF_\bullet$ is contractible, then $\calE_\bullet\otimes\calF_\bullet$ is contractible. 
      \item In particular, if $\calF_\bullet$ is locally contractible off a closed substack $\calZ$, then so is $\calE_\bullet\otimes\calF_\bullet$.
      \item   
  If $\calE_\bullet$ is locally free and $\calF_\bullet$ is acyclic, then $\calE_\bullet\otimes\calF_\bullet$ is acyclic.
  \end{enumerate}
\end{lemma}

\begin{proof}

For (1), assume $h_\bullet^{\calF_\bullet}: \calF_\bullet \to \calF_\bullet[1]
    $ is a homotopy between $\id_{\calF_\bullet}$ and the zero endomorphism on $\calF_\bullet$. Then, it is straightforward to check that 
\begin{equation*} 
h^{\calE_\bullet\otimes \mathcal{F}_\bullet}_0:=  \begin{pmatrix}
 d_{-1}^{\calE_\bullet}\otimes 1 & -1\otimes h_{-1}^{\calF_\bullet} \\
1\otimes h_{0}^{\calF_\bullet} & -d_{0}^{\calE_\bullet}\otimes 1 \end{pmatrix},
\end{equation*}
\begin{equation*} 
h^{\eE_\bullet\otimes \mathcal{F}_\bullet}_{-1}:=  \begin{pmatrix}
 -d_{0}^{\calE_\bullet}\otimes 1 & 1\otimes h_{-1}^{\calF_\bullet} \\
-1\otimes h_{0}^{\calF_\bullet} & d_{-1}^{\calE_\bullet}\otimes 1 \end{pmatrix}
\end{equation*}
provides a homotopy between $\id_{\calE_\bullet\otimes\calF_\bullet}$ and the zero endomorphism of $\calE_\bullet\otimes\calF_\bullet$.
The statement (2) follows directly from (1).

The statement about acyclicity follows from the fact that the tensor product with flat coherent sheaves preserves exactness. 
\end{proof}

\subsection{Critical $K$-theories}



\begin{definition}\label{def of kgp}
Let $\mathcal{X}$ be an algebraic stack with a regular function $\phi:\mathcal{X}\to \bbC$. Its \textit{critical} $K$-\textit{theory} 
is the Grothendieck group of the coherent matrix factorization category:
$$K_0(\calX,\phi):=K_0(\mathrm{MF} (\mathcal{X},\phi)).  $$
Similarly, for a closed substack $\calZ$, we define 
\[K_0(\calX,\calZ,\phi):=K_0(\mathrm{MF} (\calX,\phi)_\calZ).\]
\end{definition}

\begin{prop}\label{prop:ab_vs_D}
Let $Z(\phi)\subseteq \calX$ be the zero locus of the function $\phi$, with inclusion $i:Z(\phi)\to \calX$. 
There is a natural functor from the derived category of coherent sheaves on $Z(\phi)$:
$$\Upsilon:\mathrm{D}^b_{\coh}(Z(\phi))\to \mathrm{MF} (\calX,\phi), \quad \eE\mapsto (i_*\eE,0,0,0), $$ 
which induce a surjective map of Grothendieck groups
\begin{equation}\label{equ on kgps}K_0(Z(\phi)) \twoheadrightarrow  K^{}_0(\calX,\phi). \end{equation}
Moreover, if $\calX$ is a quotient stack of a smooth quasi-projective variety by a reductive group, and $\phi$ is flat, then the kernel of the composition is generated by $[\calE]$ with $\calE$ a vector bundle on $Z(\phi)$.
\end{prop}
\begin{proof}
The surjectivity follows from \cite[Lem.~2.24]{BFK}. 

Under the additional assumptions, the functor $\Upsilon$ descends to a functor 
\[\mathrm{D}^b_{\coh}(Z(\phi))/\Perf(Z(\phi))\to \mathrm{MF} (\calX,\phi)\]
where $\Perf(Z(\phi))$ is the subcategory of perfect complexes \cite{Orl}, \cite[Lem.~3.63]{BFK1}, and the induced functor is an equivalence of categories \cite{Orl}, \cite[Rmk.~3.65]{BFK1}, \cite[Thm.~3.14]{PV3}. 
Therefore, the kernel of the map \eqref{equ on kgps} is generated by vector bundles on $Z(\phi)$ \cite[Exer.~3.1.6]{Sch}. 
\end{proof}

\begin{remark}\label{rmk on Kfact0}
The category $\mathrm{MF} (\mathcal{X},0)$ can be described as the derived category of 2-periodic complexes of coherent sheaves $\calE_\bullet$, which have two cohomology sheaves $H^{i}(\calE_\bullet)$ with $i=-1,0$.

Let $\calZ$ be a closed substack of a Noetherian stack $\mathcal{X}$, we have the following notions on $\calE_\bullet$: 
\begin{enumerate}

\item {\it cohomologically supported} on $\calZ$: $H^{i}(\calE_\bullet)$ is supported on $\calZ$, 

\item {\it cohomologically acyclic off $\calZ$}: $H^{i}(\calE_\bullet|_{\calX\setminus \calZ})=0$,

\item {\it locally contractible off $\calZ$}: see Definition \ref{defi of supp}. 

\end{enumerate}
We have $(1)\Leftrightarrow (2)$ 
and $(3) \Rightarrow (2)$ (e.g.~\cite[Cor.~2.10 (ii)]{OS}).

As argued in \cite[Lem.~2.12]{OS}, if $\calE_\bullet$ satisfies $(1)$, there exists a well-defined class
$$[\calE_\bullet]:=[H^{0}(\calE_\bullet)]-[H^{-1}(\calE_\bullet)]\in K_0(\calX,\calZ):=K_0(\mathrm{D}^b_{\coh}(\calX)_\calZ). $$
By d\'evissage, we have $K_0(\calZ)\cong K_0(\calX,\calZ)$.

\end{remark}

The exterior product \eqref{equ on ten prod} induces an \textit{exterior tensor product} map on critical $K$-theories: 
\begin{equation}\label{equ on ten prod on k}
\boxtimes: K_0(\mathcal{X},\phi)\otimes K_0(\mathcal{Y},\phi')\to K_0(\mathcal{X}\times \calY,\phi\boxplus \phi'),
\end{equation}
$$([\eE_\bullet],[\mathcal{F}_\bullet])\mapsto [\eE_\bullet\boxtimes \mathcal{F}_\bullet]. $$
It is straightforward to check the map is well-defined by using the fact that projections $\calX\times \calY\to \calX$, $\calX\times \calY\to \calY$ are flat. (Any $\C$-stack is flat over $\C$.)

As a  consequence of Lemma~\ref{lem:tensor_acyc}, we have the following.
\begin{prop}\label{prop:tensor_map}
Let $\mathcal{X}$ be a Noetherian algebraic stack with two regular functions $\phi_1$ and $\phi_2$ on it. Let 
$\calE_\bullet$ be a locally free factorization of $\phi_1$. Then there is a well-defined functor
\[\calE_\bullet\otimes(-):\mathrm{MF} (\calX,\phi_2)\to \mathrm{MF} (\calX,\phi_1+\phi_2).\]
If furthermore $\phi_2=-\phi_1$, and that $\calE_\bullet$ is locally contractible off a closed substack $\calZ$, then the above functor  induces a well-defined map 
\[\calE_\bullet\otimes(-):K_0(\calX,\phi_2)\to K_0(\calX,\calZ)\cong K_0(\calZ).\]
\end{prop}
\begin{proof}
Replacing the exterior tensor product in \eqref{equ on ten prod} by tensor product, we have a functor: 
$$\calE_\bullet\otimes(-): \Fact_{\coh}(\calX,\phi_2)\to \Fact_{\coh}(\calX,\phi_2+\phi_1), $$ 
which preserves short exact sequences by the flatness of $\calE_i$, and descends to the derived category by  Lemma~\ref{lem:tensor_acyc}. 

If furthermore $\calE_\bullet$ is locally contractible off $\calZ$, then Lemma~\ref{lem:tensor_acyc} yields that $\calE_\bullet\otimes\calF_\bullet$ is locally contractible off $\calZ$. 
The map between Grothendieck groups is guaranteed by \cite[Lem.~2.12]{OS} (see Remark~\ref{rmk on Kfact0} above).
\end{proof}

\section{Specialization to the normal cone in critical $K$-theory}\label{sect on sp m}

We construct a specialization map in critical $K$-theories. 
This will be used to construct critical pullbacks in \S \ref{sect on cri pb}. 

\subsection{Some categorical preparations}\label{sect on spe3}

For a full subcategory $\calA$ of a triangulated category $\calB$, we denote by   $\langle \calA \rangle$ the smallest triangulated subcategory of $\calB$
 which contains $\calA $ and is closed under taking direct summands. 
\begin{lemma}\label{lem on pb by bdl}
Let $\pi: E\to M$ be a finite rank vector bundle on a QCA stack $M$. Then we have 
$$\mathrm{D}^b_{\coh}(E)=\langle \pi^*\mathrm{D}^b_{\coh}(M) \rangle. $$
\end{lemma}
\begin{proof}
In this proof, we use the shorthand $\Coh(-)$ to denote the pre-triangulated dg-category of bounded derived category of coherent sheaves, i.e. a dg-enhancement of $\mathrm{D}^b_{\coh}(-)$, and 
$\IndCoh(-)$ the ind-coherent sheaves. We refer to~\cite{Gai, DG} for 
the foundation and basic properties of ind-coherent sheaves for stacks. Then 
$$ \pi^*\Coh(M)\hookrightarrow \Coh(E)=(\IndCoh(E))^{\mathrm{cp}}, $$
where the superscript denotes compact generators of ind-coherent sheaves.  

Since $\pi$ is scheme-theoretic and quasi-compact, there is a well-defined functor (e.g.\,\cite[\S 3.2.9]{DG}):
$$\pi_*: \IndCoh(E)\to \IndCoh(M). $$
For $A\in \IndCoh(E)$, if $\Hom(\pi^*\Coh(M),A)=0$, by adjunction, we have $$\Hom(\Coh(M),\pi_*A)=0, $$
which implies that $\pi_*A=0\in \IndCoh(M)$.

We claim that 
\begin{equation}\label{equ on a=0}A=0\in \IndCoh(E). \end{equation} 
Then
$$\Ind(\pi^*\Coh(M))=\IndCoh(E)=\Ind(\Coh(E))$$
where the last one is the Ind-completion of $\Coh(E)$ and the last equality holds by \cite[Thm.~0.4.5]{DG} as $E$ is a QCA stack by assumption. Then $\Coh(E)=\langle\pi^*\Coh(M)\rangle$ follows. 

To prove \eqref{equ on a=0}, we note that \cite[pp.~192, (3.14)]{DG}: 
$$\IndCoh(E)\xrightarrow{\simeq}  \varprojlim\IndCoh(E|_{U_i}). $$
Here the limit is after smooth morphisms $U_i \to M$
from an affine $U_i$. 
Then it is enough to show \eqref{equ on a=0} holds an affine cover of $M$. 
Let $U_i$ be an affine scheme with smooth map $g_i: U_i\to M$ such that 
$$\xymatrix{
E|_{U_i}\cong U_i\times \mathbb{A}^n \ar[r]^{\quad \quad\quad g_i} \ar[d]_{\pi_i}  \ar@{}[dr]|{\Box}  &  E \ar[d]^{\pi} \\
U_i \ar[r]^{g_i} & M. 
}
$$
Base change implies that 
$$ 0=g_i^!\pi_*A=\pi_{i*}g_i^!A. $$
Then we are left to show $\pi_{i*}: \IndCoh(E|_{U_i})\to \IndCoh(U_i)$ is conservative. 

We have an equivalence (e.g.\,\cite[\S 3.2.3]{DG}): 
$$\boxtimes: \IndCoh(U_i)\otimes \IndCoh(\mathbb{A}^n)\xrightarrow{\cong} \IndCoh(U_i\times  \mathbb{A}^n),   $$
and $$\IndCoh(U_i)=\Ind(\Coh(U_i)), \quad \IndCoh(\mathbb{A}^n)=\QCoh(\mathbb{A}^n)$$ which is compactly generated by $\Coh(U_i)$ 
and $\oO_{\mathbb{A}^n}$ respectively. So $\IndCoh(E|_{U_i})$ is compactly generated by $\pi_i^*\Coh(U_i)$. This implies
that $\pi_{i*}$ is conservative, because $\pi_{i*}A'=0$ gives that $$\Hom(\pi_i^*\Coh(U_i),A')=\Hom(\Coh(U_i),\pi_{i*}A')=0, $$  
and the above compact generation of $\IndCoh(E|_{U_i})$ implies that $A'=0$.  
\end{proof} 
\begin{lemma}\label{lem on cpg on M}
Let $f\colon M\to B$ be a quasi-projective morphism between classical stacks and $B$ is smooth and QCA. Then 
$\mathrm{D}_{\qcoh}(M)$ is compactly generated by $\Perf(M)$. 
\end{lemma}
\begin{proof}
In this proof, we use shorthands $\Coh(-)$ (resp.~$\QCoh(-)$) to denote dg-enhancement of $\mathrm{D}^b_{\coh}(-)$ (resp.~$\mathrm{D}_{\qcoh}(-)$). 

Since $B$ is QCA, by \cite[Thm.~0.4.5]{DG}, we have 
$$\IndCoh(B)=\Ind(\Coh(B)). $$
Since $B$ is smooth, we have 
$$\IndCoh(B)=\QCoh(B), \quad \Coh(B)=\Perf(B), $$
where the first equality follows from \cite[Thm.~4.2.6]{AG} and the fact that the singular support $\Sing(B)$ of $B$ is the 
classical truncation of the $(-1)$-shifted cotangent $T^*[-1]B$, which is isomorphic to $B$, the second follows from the smoothness of $B$. 
Therefore 
\begin{align}\label{Qcoh:Perf}
\QCoh(B)=\Ind(\Perf(B))
\end{align}
is compactly generated by $\Perf(B)$. 

Let $\mathcal{L}$ be a relative ample line bundle for $f: M\to B$ and consider subcategories 
\begin{equation}\label{sub of ln and perfb}\{\mathcal{L}^{\otimes n}\otimes  f^*\Perf(B)\}_{n\in \Z}\hookrightarrow \Perf(M)=\QCoh(M)^{\mathrm{cp}}. \end{equation}
Here $\QCoh(M)^{\mathrm{cp}}$ is the subcategory of compact objects 
on $\QCoh(M)$, which coincides with perfect complexes $\Perf(M)$ 
by~\cite[Cor.\,1.4.3]{DG}. 

Take $A\in \QCoh(M)$ such that $\Hom(\mathcal{L}^{\otimes n}\otimes  f^*\Perf(B),A)=0$ for any $n$. 
Then by adjunction,   
$$\Hom( \Perf(B),f_*(A\otimes\mathcal{L}^{\otimes -n}) )=0, $$
which implies that $f_*(A\otimes\mathcal{L}^{\otimes -n})=0$, since $\Perf(B)$ compactly generates $\QCoh(B)$ as argued in (\ref{Qcoh:Perf}).
This further implies that $A=0$ as $f$ is quasi-projective.  
This implies that $\mathcal{L}^{\otimes n} \otimes f^{*}\mathrm{Perf}(B)$
for $n\in \mathbb{Z}$ compactly generates $\QCoh(M)$, hence $\mathrm{Perf}(M)$ compactly generates $\QCoh(M)$. 
\end{proof}
\begin{lemma}\label{lem on cpg on ZM}
Let $\pi\colon Z\to M$ be an affine morphism between derived stacks, $Z$ is QCA, and $\mathrm{D}_{\qcoh}(M)$ is compactly generated by $\Perf(M)$.
Then $$\Perf(Z)=\langle \pi^*\Perf(M)\rangle. $$
\end{lemma}
\begin{proof}
Given $\eE\in \mathrm{D}_{\qcoh}(Z)$ with $$0=\Hom(\Perf(M),\pi_*\eE)=\Hom(\pi^*\Perf(M),\eE), $$
we have $\pi_*\eE=0$ as $\mathrm{D}_{\qcoh}(M)$ is compactly generated by $\Perf(M)$.

Since $\pi$ is affine, so $\pi_*$ is conservative on $\mathrm{D}_{\qcoh}(-)$, and hence $\eE=0$. Therefore we know $\pi^*\Perf(M)$ compactly generates
$\mathrm{D}_{\qcoh}(Z)$. By \cite[Cor.\,1.4.3]{DG}, $\Perf(Z)$ is the subcategory of compact objects of $\mathrm{D}_{\qcoh}(Z)$, therefore 
$\Perf(Z)=\langle \pi^*\Perf(M)\rangle$ by taking the subcategory of 
compact objects in $D_{\qcoh}(Z)$~\cite[\S 0.6.7]{DG}.
\end{proof}
\subsection{Potential function on deformation space}\label{sect on pt func}
Let $$f: M\to B$$ be a Deligne-Mumford morphism 
between algebraic stacks, by \cite[\S 6.1]{Kre} which generalizes \cite[\S 5]{Fu}, one has a \textit{deformation space} $M_f^{\circ}$ flat over $\mathbb{P}^1$ such that we have the following Cartesian diagrams
\begin{equation}\label{diag on defspace}
\xymatrix{
\fC_f  \ar@{^{(}->}[r]^{ }  \ar[d]_{ }  \ar@{}[dr]|{\Box}  &  M_f^{\circ} \ar[d]_{}^{\mathrm{flat}} \ar@{}[dr]|{\Box} & B\times (\mathbb{P}^1\setminus \{0\})   \ar@{_{(}->}[l]^{ } \ar[d] \\
\{0\}   \ar@{^{(}->}[r]^{ } &  \mathbb{P}^1 & \mathbb{P}^1\setminus \{0\}, 
\ar@{_{(}->}[l]^{ }
}
\end{equation}
where $\fC_f$ is the intrinsic normal cone of $f$ \cite{BF}.
 
Define an \textit{open substack} $\mathring{M}_f$ of the deformation space $M_f^{\circ}$ by the Cartesian diagram: 
\begin{equation}\label{diag on open defo} 
\xymatrix{
\mathring{M}_f \ar@{^{(}->}[r]^{ } \ar[d]_{ }  \ar@{}[dr]|{\Box}  &  M_f^{\circ} \ar[d]_{}^{\mathrm{flat}}   \\
\mathbb{P}^1\setminus \{\infty\}   \ar@{^{(}->}[r]^{\,\,\, \mathrm{open}} &  \mathbb{P}^1.
}
\end{equation}
Restricting diagram \eqref{diag on defspace} to the open substack $\mathring{M}_f$ gives 
 \begin{equation}\label{diag on defspace2}
\xymatrix{
\fC_f \ar@{^{(}->}[r]^{ }\ar[d]_{ }  \ar@{}[dr]|{\Box}  &  \mathring{M}_f \ar[d]_{t}^{\mathrm{flat}} \ar@{}[dr]|{\Box} & B\times \C^*  \ar@{_{(}->}[l]^{ } \ar[d] \\
\{0\}  \ar@{^{(}->}[r]^{ } & \mathbb{A}^1_{}:=\mathbb{P}^1\setminus \{\infty\} & \C^*. \ar@{_{(}->}[l]^{ }
}
\end{equation}
Assume $f: M\to B$ is the classical truncation of $\emph{\textbf{f}}$ in Setting \ref{setting of lag}, and 
\begin{equation}\label{reso of obs the}\mathbb{E}_f:=\mathbb{L}_{\emph{\textbf{f}}}\,|_{M}\cong (V\xrightarrow{d}  E \xrightarrow{d^\vee}  V^\vee)  \end{equation}
is a symmetric resolution by finite rank vector bundles $V, E$ on $M$. There is a 
closed immersion 
\begin{equation}\label{equ on cone emb to cone stack}\fC_f\hookrightarrow \fC_{\mathbb{E}_f} \end{equation} 
of the intrinsic normal cone $\fC_f$ of $f$ into the abelian cone stack \cite[Prop.~2.4]{BF}: 
\begin{align*}
    \fC_{\mathbb{E}_f}=\left[\Spec_M \mathrm{Sym}(Q)/V \right]
\end{align*}
of $\mathbb{E}_f$, where $Q=\mathrm{Cok}(d)$. 
It is proved in~\cite[Prop.~1.7]{Park1} that there is a well-defined function 
$$q_{\fC_{\mathbb{E}_f}}: \fC_{\mathbb{E}_f}\to \C $$
given by the descent of the function 
\begin{align*}
    \Spec \mathrm{Sym}(Q) \hookrightarrow E \stackrel{q_E}{\to} \mathbb{C}. 
\end{align*}
We denote its restriction to $\fC_f$ via \eqref{equ on cone emb to cone stack} to be 
\begin{equation}\label{fun on intr nc}q_{\fC_f}: \fC_f\to \C. \end{equation}
We notice that by the construction of $\mathring{M}_f$, there is a natural map 
\begin{align}\label{nat:map} \mathring{M}_f\to B 
\end{align}
making the following diagram commutative
\begin{align}\notag\xymatrix{
\mathring{M}_f\ar[dr]&B\times \bbC^*\ar[d]\ar@{_{(}->}[l]^{ }\\
&B,
}\end{align}
where the vertical map is the projection and horizontal map is an open immersion. 
Denote the pullback of function $\phi$ to $\mathring{M}_f$ by $\phi|_{\mathring{M}_f}$.

\begin{lemma}\label{prop on fun on def}
There is a regular function $\bar{\phi}\colon \mathring{M}_f \to \C$
such that  
\begin{itemize}
\item
$t^2\bar{\phi}=\phi|_{\mathring{M}_f}$, 
where $t$ is the coordinate function of $\mathbb{A}^1_{}$ 
\item
$\bar{\phi}|_{\fC_f}=q_{\fC_f}$.
\end{itemize}
\end{lemma}
\begin{proof}
 
Let $\emph{\textbf{D}}_{\emph{\textbf{f}}}$ be the derived deformation space of $\emph{\textbf{f}}\colon\emph{\textbf{M}} \to B$ with canonical map $\emph{\textbf{M}}\times \mathbb{A}^1\to \emph{\textbf{D}}_{\emph{\textbf{f}}}$ (e.g.~\cite[\S 1.2]{Park2}).
By \cite[Prop.~5.1.1]{Park2}, we can deform the $(-2)$-shifted symplectic form $\Omega_{\emph{\textbf{f}}}$ on $\emph{\textbf{f}}\colon\emph{\textbf{M}} \to B$
to a $(-2)$-shifted symplectic form $\Omega_{ }$ on $\emph{\textbf{M}}\times \mathbb{A}^1\to D_{\emph{\textbf{f}}}$. 
In particular, when restricting to a general fiber $t\neq 0$, $\Omega$ becomes $t^{-2}\cdot \Omega_{\emph{\textbf{f}}}$. 

We can take the underlying function \cite[Def.~1.1.1]{Park2} of $\Omega_{\emph{\textbf{f}}}$ and $\Omega$ (i.e.~their image in the periodic cyclic homology), denoted by $[\Omega_{\emph{\textbf{f}}}]$ and $[\Omega]$ respectively. By Setting \ref{setting of lag}, we have 
$$[\Omega_{\emph{\textbf{f}}}]=\phi: B\to \C. $$ 
We define $\bar{\phi}$ to be the restriction of $[\Omega]$ to the classical truncation of $\emph{\textbf{D}}_{\emph{\textbf{f}}}$, which defines a function 
$\bar{\phi}:\mathring{M}_f\to \C$.
Its restriction to general fibers then satisfies 
$$\bar{\phi}|_{B\times \{t\}}=t^{-2}\cdot \phi.$$ 
As for the restriction of $\bar{\phi}$ to the special fiber, by tracing the map $\mathrm{Sp}$ in the bottom of \cite[pp.~50]{Park2}, and using \cite[Prop.~3.2.1\,(2)]{Park2}, 
one can see this is exactly the quadratic function $q_{\fC_f}$ mentioned in \cite[Rmk.~1.8]{Park1}, i.e.~\eqref{fun on intr nc}.

To sum up, we obtain a commutative diagram:  
 \begin{equation*}
\xymatrix{
\fC_f  \ar@{^{(}->}[r]^{ } \ar[dr]_{q_{\fC_f}}     &  \mathring{M}_f \ar[d]_{}^{\bar{\phi} }   & B\times \C^*    \ar[dl]^{t^{-2}\phi}  \ar@{_{(}->}[l]^{ }  \\
&\mathbb{A}^1 &  }
\end{equation*}
By the flatness of $t\colon \mathring{M}_f\to \mathbb{A}^1_{\infty}$, this implies that $t^2\bar{\phi}=\phi|_{\mathring{M}_f}$. 
\end{proof}
\begin{remark}
The above proof uses symplectic deformation of $(-2)$-shifted symplectic structure to normal bundle/cone \cite[\S 5.1]{Park2}. One can alternatively give a 
direct construction of the function by gluing local ones on $(-1)$-shifted Lagrangian neighbourhood charts of Joyce-Safronov \cite{JS}.
\end{remark}

\subsection{Specialization map for zero loci}\label{sect on spe1}

Continued with the previous section, denote derived zero loci $Z^{\mathrm{der}}(\bar\phi)$, $Z^{\mathrm{der}}(q_{\fC_f})$ by the following homotopy pullbacks  
\begin{equation}
\xymatrix{
Z^{\mathrm{der}}(\bar\phi) \ar@{^{(}->}[r]^{ } \ar[d]_{ } \ar@{}[dr]|{\Box}   &  \mathring{M}_f  \ar[d]_{}^{\bar\phi}   \\
 \{0\}   \ar@{^{(}->}[r]^{ } &  \mathbb{A}^1,
}
\quad \quad
\xymatrix{
Z^{\mathrm{der}}(q_{\fC_f}) \ar@{^{(}->}[r]^{ } \ar[d]_{ } \ar@{}[dr]|{\Box}   &  \fC_f  \ar[d]_{}^{q_{\fC_f}=\bar\phi|_{\fC_f}}   \\
 \{0\}  \ar@{^{(}->}[r]^{ } &  \mathbb{A}^1.
}
\nonumber \end{equation}
We have inclusions of closed and open substacks 
\begin{equation}\label{equ on zqc}Z^{\mathrm{der}}(q_{\fC_f})
\stackrel{i}{\hookrightarrow}  Z^{\mathrm{der}}(\bar\phi) \stackrel{j}{\hookleftarrow} Z(\phi)\times \C^*, \end{equation} 
where $i$ is a regular embedding, defined by the zero locus of $t: Z^{\mathrm{der}}(\bar\phi) \to \mathbb{A}^1$. 

By the localization sequence (e.g.~\cite[Thm.~3.9]{Kh}), the Grothendieck groups of coherent sheaves on them fit into an exact sequence 
$$K_0(Z^{\mathrm{der}}(q_{\fC_f})) \xrightarrow{i_*} K_0(Z^{\mathrm{der}}(\bar\phi)) \xrightarrow{\,\,j^*} K_0(Z(\phi)\times \C^*)\to 0. $$ 
Since $i$ is a regular closed immersion, there is a Gysin pullback 
$$i^{*} \colon K_0(Z^{\mathrm{der}}(\bar\phi)) \to K_0(Z^{\mathrm{der}}(q_{\fC_f})), $$
which satisfies that $i^*i_*=0$ since the normal bundle of $i$ is a trivial line bundle\footnote{We refer to \cite{AP} for intersection operations in $K$-theory.}. 
Therefore there is a unique vertical map
making the following diagram commutative
\begin{equation}\label{diag on uniq sp map}
\xymatrix{
K_0(Z^{\mathrm{der}}(\bar\phi))   \ar[rd]_{i^*}\ar[r]^{j^*}&K_0(Z(\phi)\times \C^*)  \ar[d]^{\exists !}\\
&K_0(Z^{\mathrm{der}}(q_{\fC_f})). }   \end{equation}
We define a \textit{specialization map}: 
\begin{equation}\label{sp map in k of zero}
sp: K_0(Z(\phi)) \to K_0(Z(\phi)\times \C^*) \to K_0(Z^{\mathrm{der}}(q_{\fC_f})),    \end{equation} 
as the composition of smooth pullback of the projection $$Z(\phi)\times \C^*\to Z(\phi)$$ and the vertical map in \eqref{diag on uniq sp map}.

Given a symmetric resolution as in Setting \ref{setting of lag}: 
\begin{equation}\mathbb{E}_f:=\mathbb{L}_{\emph{\textbf{f}}}\,|_{M}\cong (V\xrightarrow{d}  E \xrightarrow{d^\vee}  V^\vee), \nonumber \end{equation} 
we define closed substack $C_f$ of $E$ by the base change:  
\begin{equation}\label{base cg of normal cones}
\xymatrix{
  C_f \ar@{^{(}->}[r]^{ }  \ar[d]_{p}  \ar@{}[dr]|{\Box}  &  E \ar[d]^{p} \\
 \fC_f  \ar@{^{(}->}[r]^{ }  & [E/V].  }
 \end{equation}
By \cite[Prop.~1.7]{Park1}, we know 
$$q_E|_{C_f}=q_{\fC_f}\circ p: C_f \to \mathbb{A}^1, $$
where $q_E$ is the quadratic function of the quadratic bundle $E$. 

Notations as in Setting \ref{setting of lag}, we define
\begin{equation}\label{sp map1}K_0(Z(\phi_X))\xrightarrow{\pi_X^*}  K_0(Z(\phi)) \xrightarrow{sp} K_0(Z^{\mathrm{der}}(q_{\fC_f})) 
\xrightarrow{p^*} K_0(Z^{\mathrm{der}}(q_{E}|_{C_f})).  \end{equation}  
\begin{lemma}\label{lem on perf goes to pr}
The image of a perfect complex on $Z(\phi_X)$ under the composition \eqref{sp map1} can be represented by a perfect complex on $Z^{\mathrm{der}}(q_{E}|_{C_f})$. 
\end{lemma}
\begin{proof}
The maps $\pi_X^*$ and $p^*$ clearly send perfect complexes to perfect complexes. 
It is enough to show the statement for the specialization map $sp$.  

By Lemma \ref{prop on fun on def}, we have a closed immersion and a projection
$$Z^{\mathrm{der}}(\bar\phi)\hookrightarrow Z^{\mathrm{der}}(\phi|_{\mathring{M}_f}) \to Z^{\mathrm{der}}(\phi)=Z^{}(\phi), $$ 
where the second arrow is induced by the natural map (\ref{nat:map}) and the last equality is by the flatness of $\phi: B\to \C$. By pulling back along these maps, a perfect complex $\eE$ on $Z^{}(\phi)$ 
goes to a perfect complex $\widetilde{\eE}$ on $Z^{\mathrm{der}}(\bar\phi)$. It is easy to check that the restriction of $\widetilde{\eE}$ to the open locus $Z(\phi)\times \C^*$ 
coincides with the pullback of $\eE$ via $Z(\phi)\times \C^*\to Z(\phi)$ (up to multiplication by the invertible function $t^{-2}$). Since $sp(\mathcal{E})$ is represented by 
$i^{\ast}\widetilde{\eE}$ where $i$ is the closed immersion in \eqref{equ on zqc}, 
it follows that $sp(\mathcal{E})$ is represented by a perfect complex. 
\end{proof}

\subsection{Relative singularity category and matrix factorization cateogry}\label{sect on spe2}

Consider homotopy pullback diagrams 
\begin{equation}\label{diag on iotace}
\xymatrix{
Z^{\mathrm{der}}(q_{E}|_{C_f}) \ar@{^{(}->}[r]^{ }\ar[d]^{\iota_{C}} \ar@{}[dr]|{\Box}   &  Z(q_E) \ar[r]^{ } \ar[d]^{\iota_{E}} \ar@{}[dr]|{\Box}  &  \{0\} \ar[d]_{}^{  }   \\
C_f   \ar@{^{(}->}[r]^{ }   & E   \ar[r]^{q_E} &  \mathbb{A}^1, }
\end{equation}
where we note that $q_E$ is flat so its derived zero locus coincides with the classical zero locus $Z(q_E)$. 

We have functors of taking quotient and proper pushforward: 
\begin{equation}\label{sp map2}\mathrm{D}^b_{\coh}(Z^{\mathrm{der}}(q_{E}|_{C_f}))\xrightarrow{ }  \frac{\mathrm{D}^b_{\coh}(Z^{\mathrm{der}}(q_{E}|_{C_f}))}{\langle 
\iota_{C}^*\mathrm{D}^b_{\coh}(C_f)) \rangle} \xrightarrow{ }  \frac{\mathrm{D}^b_{\coh}(Z (q_{E}))}{\langle \iota_E^*\mathrm{D}^b_{\coh}(E)) \rangle}.
\end{equation}
We show perfect complexes goes to zero under the functor \eqref{sp map2}.

\begin{prop}\label{prop on vanishing of pef}
Notations as in diagram \eqref{diag on iotace}, then any perfect complex on $Z^{\mathrm{der}}(q_{E}|_{C_f})$ lies in 
$\langle \iota_{C}^*\mathrm{D}^b_{\coh}(C_f)) \rangle$
\end{prop}
\begin{proof}
Consider the composition of $\iota_C$ in \eqref{diag on iotace} with the natural projection $p \colon C_f\to M$: 
\begin{equation}\label{equ pizm}\pi_{Z\to M}: Z^{\mathrm{der}}(q_{E}|_{C_f})\to M. \end{equation}
Applying Lemma \ref{lem on cpg on M} to $f: M\to B$, we know 
$\mathrm{D}_{\qcoh}(M)$ is compactly generated by $\Perf(M)$. Now
applying Lemma \ref{lem on cpg on ZM} to \eqref{equ pizm}, we conclude that 
$$\Perf(Z^{\mathrm{der}}(q_{E}|_{C_f}))= \langle \pi_{Z\to M}^*\Perf(M)\rangle=\langle \iota_C^{*}p^{*}\mathrm{Perf}(M))\rangle, $$ 
which is clearly a subcategory of 
$\langle \iota_{C}^*\mathrm{D}^b_{\coh}(C_f)) \rangle$. 
\end{proof}

Now we go from the relative singularity category to the matrix factorization category.
Note that our space $M$ (hence $E$) could be stacky and singular, it may not have resolution property. 
\begin{prop}\label{prop on compare sing and mf}
Notations as in diagram \eqref{diag on iotace}.
The functor
$$\iota_{E*}: \mathrm{D}^b_{\coh}(Z (q_{E})) \to \mathrm{MF}_{}(E,q_E),  $$
$$\eE\mapsto (\iota_{E*}\eE,0,0,0)$$
descends to a functor 
\begin{equation}\label{sp map3} \iota_{E*}: \frac{\mathrm{D}^b_{\coh}(Z (q_{E}))}{\langle \iota_E^*\mathrm{D}^b_{\coh}(E) \rangle}\to \mathrm{MF}_{}(E,q_E). \end{equation}
\end{prop}
\begin{proof}
We follow the proof of \cite[\S 2.7, Theorem]{EP}. Let $F^\bullet\in\mathrm{D}^b_{\coh}(E)$, then 
$$\iota_E^*F^\bullet=(F^\bullet\xrightarrow{\cdot q_E} F^\bullet).$$
Assume that $F^\bullet=\pi^*G^\bullet$ for some $G^\bullet\in \mathrm{D}^b_{\coh}(M)$. Since the function $q_E$ locally depends only on fibers, and 
$\pi^*G^\bullet$ is pulled back from the base, 
hence $\pi^*G^i$ is $q_E$-flat, i.e. $\pi^*G^i\xrightarrow{\cdot q_E}\pi^*G^i$ is injective for any $i$. 
There is a short exact sequence of matrix factorizations:
$$(\pi^*G^i,\pi^*G^i,q_E,1)\to (\pi^*G^i,\pi^*G^i,1,q_E) \to \iota_{E*}\iota_E^*F^i=(\pi^*G^i\otimes \oO_{Z(q_E)},0,0,0), $$
where the first two terms are zero in $\mathrm{MF}_{}(E,q_E)$ (e.g.~\cite[Lem.~2.24]{BFK}), so is the last term. 
Therefore 
$$\iota_{E*}\iota_E^*F^\bullet=\iota_{E*}(F^\bullet\xrightarrow{\cdot q_E} F^\bullet)=\iota_{E*}\left(\pi^*G^\bullet\otimes \oO_{Z(q_E)}\right)$$
is zero in the derived matrix factorization category. The vanishing obviously holds also for any direct summand of $\iota_{E*}\iota_E^*F^\bullet$. 
By using Lemma \ref{lem on pb by bdl}, we are done.  
\end{proof}

\subsection{Definition of specialization map for critical $K$-theories}\label{sect on spe}

Combining \eqref{sp map1}, \eqref{sp map2}, \eqref{sp map3}, we obtain: 
\begin{equation}\label{equ on spe map on k0} K_0(Z(\phi_X))\xrightarrow{\eqref{sp map1}} K_0(Z^{\mathrm{der}}(q_{E}|_{C_f}))\xrightarrow{\eqref{sp map2}}K_0\left(\frac{\mathrm{D}^b_{\coh}(Z (q_{E}))}{\langle \iota_E^*\mathrm{D}^b_{\coh}(E)) \rangle}\right) \xrightarrow{\eqref{sp map3}} K_0(E,q_E). \end{equation}
By Proposition \ref{prop:ab_vs_D}, there is an exact sequence
$$K_0(\Perf(Z(\phi_X)))\to K_0(Z(\phi_X))\to K_0(X,\phi_X)\to 0$$
By Lemma \ref{lem on perf goes to pr}, Proposition \ref{prop on vanishing of pef}, we know that the $K$-theory class of a perfect complex on
$Z(\phi_X)$ goes to zero under the composition of \eqref{sp map1}, \eqref{sp map2}. Therefore \eqref{equ on spe map on k0} descends to 
a map from $K_0(X,\phi_X)$. 
\begin{definition}\label{defi of specialization map}
Notations as above, we define the \textit{specialization map}: 
\begin{equation}\label{equ on spe map on k}
\sigma_{\pi_X\circ f}: K_0(X,\phi_X)\xrightarrow{} K_0(E,q_E) \end{equation}
as the unique map such that the following diagram commutes
$$
\xymatrix{
K_0(Z(\phi_X))  \ar@{^{}->>}[r]^{ }  \ar[rd]_{\eqref{equ on spe map on k0}}    &  K_0(X,\phi_X) \ar[d]^{\sigma_{\pi_X\circ f}} \\
&  K_0(E,q_E).
}
$$
\end{definition} 
\begin{remark}\label{rmk on sp of cri ex}
In an arbitrary 1-parameter family, the  $K$-theory of coherent sheaves has a \textit{specialization map}. In contrast,  when the family has a potential function,  \textit{critical} $K$-\textit{theory} does not necessarily have a specialization map which is compatible with the canonical map. 

Let $X=\C^2$, $Z=\Delta$ be the diagonal of $X$ and $U:=X\setminus Z$ be the complement, which fit into 
Cartesian diagrams 
$$
\xymatrix{
Z \ar@{^{(}->}[r]^{ }\ar[d]_{ }  &  X\ar[d]^{x-y} & U  \ar@{_{(}->}[l]^{ } \ar[d] \\
\{0\}  \ar@{^{(}->}[r]^{i} & \C_{} & \C^*. \ar@{_{(}->}[l]^{ }
}
$$
Take a regular function $\phi: X=\C^2\to \C$, $(x,y)\mapsto x\cdot y$. A specialization map is a map 
\begin{equation}\label{equ on sp ex}K_0(U,\phi|_{U})\to K_0(Z,\phi|_{Z}), \end{equation}
obtained by taking an extension to $K_0(X,\phi)$, followed by the Gysin pullback $i^!$ to the zero fiber. Note that critical 
locus of $\phi$ is supported on $Z$, hence $K_0(U,\phi|_{U})=0$. If \eqref{equ on sp ex} exists, Gysin pullback 
$$i^!: K_0(X,\phi)\to K_0(Z,\phi|_{Z})$$
must be zero. However, $[(\oO_{\C^2},\oO_{\C^2},x,y)]\in K_0(X,\phi)$ goes to the generator of $K_0(Z,\phi|_{Z})\cong \Z_2$.
\end{remark}

\section{Spinor morphisms}\label{sec:spin}

In this section, we construct in a greater generality a $K$-theoretic analogue to the Koszul/fundamental factorizations in the literature \cite{CFGKS,FK,PV2},
which we call spinor morphisms (Definition \ref{defi of spin map gener}).
With the specialization map defined in the previous section, they will be used to construct critical pullbacks in \S \ref{sect on cri pb}.

\subsection{Koszul factorization}\label{sect on kos}

Let $\mathcal{X}$ be an algebraic stack with a regular function $\phi:\mathcal{X}\to \bbC$. 
Let $\Lambda$ be a finite rank vector bundle on $\calX$ and $E:=\Lambda\oplus\Lambda^\vee$ be the direct sum with natural non-degenerate quadratic form $Q$ defined by 
pairing $\Lambda$ and $\Lambda^\vee$. 
Let $s_0\in H^0(\calX,\Lambda)$ and $s_1\in H^0(\calX,\Lambda^\vee)$ such that 
$$ \langle s_0,s_1\rangle=\phi. $$
Written using $Q$ and $s=(s_0,s_1)$, the above is equivalent to $Q(s,s)=2\phi$. 
\begin{definition}\label{defi of kos fact}
The \textit{Koszul factorization} $\Kos(\Lambda, s)$ of the above data is 
\begin{equation}\label{equ on kos fa}\xymatrix{
\bigoplus_k\bigwedge^{2k+1}\Lambda^\vee
\ar@/^/[r]^{\,\,\, s_1\wedge+\iota_{s_0}}&\bigoplus_k\bigwedge^{2k}\Lambda^\vee. \ar@/^/[l]^{\,\,\, s_1\wedge+\iota_{s_0}}
}\end{equation}
It is easy to check the composition of two differentials above is the multiplication by $\phi$ \cite[Def.~3.19]{BFK1}. Therefore it is an object in the abelian category 
$\Fact_{\coh}(\mathcal{X},\phi)$ of matrix factorizations.
\end{definition}

\begin{prop} \label{prop:Spin}
We have the following: 
\begin{enumerate}
    \item In the setup above, for any section $s$  of $E$ so that $Q(s,s)=2\phi$, the underlying pair of coherent sheaves of the factorization $\mathrm{Kos}(\Lambda, s)$ is a pair of flat coherent sheaves. 
    \item The factorization $\mathrm{Kos}(\Lambda, s)$ is locally contractible  off $Z(s)$. 
\end{enumerate} 
\end{prop}
\begin{proof}
(1) By definition, this is a pair of locally free coherent sheaves. 

(2)  
There is a smooth atlas on which $\mathrm{Kos}(\Lambda, s)$ is a tensor product of flat factorizations with at least a tensor-factor contractible (e.g.~\cite[proof of Prop.~2.3.3]{CFGKS}). Hence, the conclusion follows from Lemma~\ref{lem:tensor_acyc}. 
\end{proof}

\subsection{Spinor morphism:~special case}
Let $(E, Q)$ be a non-degenerate quadratic bundle on an algebraic stack $\calX$, 
with a choice of \textit{orientation},~i.e.~a reduction of the structure group of $(E, Q)$ from $O(E, Q)$ to $SO(E, Q)$.
Assume there is a maximal isotropic subbundle $\Lambda$ of $E$, so that there is a short exact sequence 
\begin{equation}\label{equ on ses of E}0\to \Lambda\to E\to \Lambda^\vee\to 0.\end{equation}
We define the sign $(-1)^{|\Lambda|}=\pm 1$ of $\Lambda$ 
according to the choice of orientation on $(E,Q)$ \cite[Def.~2.2]{OT}. 

In this section, we use a Jouanolou-type trick as follows.
Consider the total space of vector bundle $\mathcal{H}om(\Lambda^\vee,E)$ on $\calX$. We define stack $\calY'_{\Lambda,E}$ by 
the following Cartesian diagram of stacks over $\calX$: 
$$\xymatrix{ 
\calY'_{\Lambda,E} \ar[r]^{ } \ar[d]_{ }  \ar@{}[dr]|{\Box}  &  \mathcal{H}om(\Lambda^\vee,E) \ar[d]^{} \\
\{\id_{\Lambda^\vee}\} \ar[r]^{ } & \mathcal{H}om(\Lambda^\vee,\Lambda^\vee), 
}
$$
where $\id_{\Lambda^\vee}: \Lambda^\vee\to \Lambda^\vee$ is the identity map, and the right vertical map is via the projection in \eqref{equ on ses of E}. 
Clearly, $\calY'_{\Lambda,E}$ is a torsor over the vector bundle $\mathcal{H}om(\Lambda^\vee,\Lambda)$. By construction,  after base change to $\calY'_{\Lambda,E}$, the sequence \eqref{equ on ses of E} splits which gives a canonical isomorphism 
$$E|_{\calY'_{\Lambda,E}}\cong \Lambda|_{\calY'_{\Lambda,E}}\oplus\Lambda^\vee|_{\calY'_{\Lambda,E}}. $$ 
Let $\calY_{\Lambda,E}\subset\calY'_{\Lambda,E}$ be the locus where quadratic forms are preserved under the isomorphism above, 
where the quadratic form on $\Lambda\oplus\Lambda^\vee$ is given by the pairing $\langle\bullet,\bullet\rangle$ between $\Lambda$ and  $\Lambda^\vee$. 
Now we construct a subbundle of $\mathcal{H}om(\Lambda^\vee,\Lambda)$, over which $\calY_{\Lambda,E}$ is  a torsor. 
This is a local question. We choose a local basis of $\Lambda$ and dual basis of $\Lambda^\vee$. Consider the subspace of $\mathcal{H}om(\Lambda^\vee,\Lambda)$
where the map  
$\varphi: \Lambda^\vee\to \Lambda$ preserves the quadratic form on $\Lambda^\vee\oplus\Lambda$. Under the choice of basis, this condition on $\varphi$ becomes 
$$
\begin{pmatrix}
\id_{ } & \varphi  \\
0 & \id_{ }
\end{pmatrix} 
\cdot
\begin{pmatrix}
0 & \id \\
\id & 0
\end{pmatrix}
\cdot 
\begin{pmatrix}
\id_{ } & 0  \\
\varphi^{t} & \id_{ }
\end{pmatrix} 
=
\begin{pmatrix}
0 & \id \\
\id & 0
\end{pmatrix},
$$
Here $\id$ is the identity matrix with the same rank as $\rk\Lambda$ and 
$\varphi^t$ denotes the transpose of the matrix $\varphi$. 
The condition above  reduces to 
$$\varphi+\varphi^t=0, $$
which defines a subvector bundle of $\mathcal{H}om(\Lambda^\vee,\Lambda)$, since on each fiber over $\calX$ the subspace consists of skew-symmetric matrices. (More intrinsically, it is the unipotent radical of the parabolic subgroup of $SO(E,Q)$ preserving $\Lambda$). 
It is straightforward to check $\calY_{\Lambda,E}$ is a torsor over this subvector bundle. 

Denote the projection of the affine bundle $\calY_{\Lambda,E}$ to the base $\calX$ by 
\begin{equation}\label{map of YtoX}\pi_{\calY/\calX}: \calY_{\Lambda,E}\to \calX. \end{equation}
By pulling $E$ back to $\calY_{\Lambda,E}$, we have a canonical isomorphism of quadratic bundles: 
$$\pi_{\calY/\calX}^*E\cong \pi_{\calY/\calX}^*\Lambda \oplus\pi_{\calY/\calX}^*\Lambda^\vee. $$ 
Let $s$ be a section of $E$ such that $2\phi:=-Q(s, s)$.
Then construction in \S \ref{sect on kos} gives a Koszul factorization $\Kos\left(\pi_{\calY/\calX}^*\Lambda,\pi_{\calY/\calX}^*s\right)$
of $-\pi_{\calY/\calX}^{\ast}\phi$. 

\begin{definition}
Let $(E, Q)$ be an oriented non-degenerate quadratic bundle on an algebraic stack $\calX$ such that \eqref{equ on ses of E} exists. 
Let  $s$ be a section of $E$ and $2\phi:=-Q(s,s): \calX\to \C$. Denote $Z(s)\subseteq\calX$ to be the zero locus of $s$, and 
$\tilde Z(s):=\pi_{\calY/\calX}^{-1}(Z(s))$ to be the preimage.

The \textit{spinor morphism} $\mathrm{Spin}(E,\Lambda,s)$ is the following composition:

\begin{align}\label{equ on spin mo}
K_0(\calX,\phi)\xrightarrow{\pi_{\calY/\calX}^*} K_0(\calY_{\Lambda,E},\pi_{\calY/\calX}^*\phi)  
 \xrightarrow{\otimes (-1)^{|\Lambda|} \Kos\left(\pi_{\calY/\calX}^*\Lambda,\pi_{\calY/\calX}^*s\right)}  K_0(\calY_{\Lambda,E},\tilde Z(s)) & \cong  K_0(\tilde Z(s))  
    \xrightarrow{\cong}    K_0(Z(s)) 
  \\ \nonumber 
  & \xrightarrow{\otimes\sqrt{\det\Lambda}}K_0(Z(s),\Z[1/2]), 
\end{align}
where the first map is the flat pullback by \eqref{map of YtoX}, the second map is given by Proposition~\ref{prop:tensor_map} while conditions therein are guaranteed by Proposition~\ref{prop:Spin}, $(-1)^{\lvert \Lambda \rvert}=\pm 1$ is the sign of $\Lambda$ \cite[Def.~2.2]{OT}. 
The first isomorphism is by d\'evissage, 
the second is
by the homotopy property of $K$-theory, e.g.~\cite[Thm.~3.17]{Kh},
the square root $\sqrt{\det\Lambda}$ exists in $K_0(\calX,\Z[1/2])$ by \cite[Lem.~5.1]{OT}. 
\end{definition}
\begin{prop}\label{prop on ycover}
Notations as above, if there is an isomorphism $E\cong \Lambda\oplus \Lambda^\vee$ of quadratic vector bundles, then the composition
$$K_0(\calX,\phi)\xrightarrow{\otimes (-1)^{|\Lambda|} \Kos(\Lambda,s)} K_0(\calX,Z(s))\cong K_0(Z(s)) \xrightarrow{\otimes\sqrt{\det\Lambda}} K_0(Z(s),\Z[1/2])$$
coincides with the spinor morphism \eqref{equ on spin mo}. Here the isomorphism is by d\'evissage. 
\end{prop}
\begin{proof}
The isomorphism $E\cong \Lambda\oplus \Lambda^\vee$ defines a section $i$ of $\pi_{\calY/\calX}: \calY_{\Lambda, E}\to \calX$,  
which restricts to a section $i$ of $\tilde Z(s)\to Z(s)$. And we have 
$$i^*\Kos\left(\pi_{\calY/\calX}^*\Lambda,\pi_{\calY/\calX}^*s\right)=\Kos(\Lambda,s). $$
Then the claim follows directly from the following commutative diagram
$$
\xymatrix{
K_0(\calX,\phi) \ar[rrr]^{\otimes \Kos(\Lambda,s)}     & & &  K_0(\calX,Z(s)) \ar[r]^{\,\, \cong}  &  K_0(Z(s))  
 \\
K_0(\calY_{\Lambda,E},\pi_{\calY/\calX}^*\phi)  
\ar[rrr]^{\otimes\Kos\left(\pi_{\calY/\calX}^*\Lambda,\pi_{\calY/\calX}^*s\right)} \ar[u]_{i^*} &  & & K_0(\calY_{\Lambda,E},\tilde Z(s)) \ar[r]^{\quad \cong} \ar[u]_{i^*}^{ }   &  K_0(\tilde Z(s)), \ar[u]_{i^*}^{ } }
$$
and the fact that $i^*\pi_{\calY/\calX}^*=\id$.
\end{proof}

\subsection{Spinor morphism:~general case}\label{subsec:general_spin}
Let $(E,Q)$ be an oriented non-degenerate quadratic bundle of even rank on $\calX$, 
$s$ be a section of $E$ with zero locus $Z(s)\subseteq\calX$, and $2\phi:=-Q(s,s): \calX\to \C$.

Let $\rho:\tilde\calX\to \calX$ be the isotropic Grassmannian of $(E,Q)$, so that $\rho^*E$ 
has a tautological maximal isotropic subbundle $\Lambda_\rho$ \cite[\S 6]{EG} (see also \cite[\S 3.1]{OT}). 
\begin{definition}\label{defi of spin map gener}
Notations as above, the \textit{spinor morphism} is the following composition:  
\begin{equation}\label{equ on spin gene}\Spin(E,s):K_0(\calX,\phi)\xrightarrow{\rho^*}K_0(\tilde\calX,\rho^*\phi)\xrightarrow{\Spin(\rho^*E,\Lambda_\rho,\rho^*s)}
K_0(Z(\rho^*s),\Z[1/2])\xrightarrow{\rho_*}K_0(Z(s),\Z[1/2]),
\end{equation}
where $\rho_*$ is the Deligne-Mumford proper pushforward. 
\end{definition}

\subsection{Comparison with $K$-theoretic localized square root Euler class}
We follow the notations in \S \ref{subsec:general_spin}. 
Let $Z(\phi)$ be the zero locus of $\phi: \calX\to \C$. 
By abuse of notations, we denote by 
 $(E,Q,s)$ its base-change to $Z(\phi)$. Similarly, we write base-change of  the isotropic Grassmannian to $Z(\phi)$ by $\rho$, and similar for the base-change of $\Lambda_\rho$.
Note that after base-change to $Z(\phi)$, the section $s$ is \textit{isotropic}. 
Using the data above, Oh and Thomas constructed a localized $K$-theoretic half Euler class \cite[Def.~5.7]{OT}:
\begin{equation}\label{equ on hec}\sqrt{e}(E,s): K_0(Z(\phi))\xrightarrow{\rho^*} K_0(Z(\rho^*\phi)) \xrightarrow{\sqrt{e}(\rho^*E,\rho^*s,\Lambda_\rho)} K_0(Z(\rho^*s),\Z[1/2])\xrightarrow{\rho_*} K_0(Z(s),\Z[1/2]). \end{equation}
Here the map $\sqrt{e}(\rho^*E,\rho^*s,\Lambda_\rho)$ is defined in \cite[Eqn.~(99)]{OT}.
It is easy to check that this class keeps the same if one multiplies $s$ by a constant.

We compare this map with the spinor morphism \eqref{equ on spin gene}. 
\begin{prop}\label{prop on cpr spin and old}
Notations as above, there is a commutative diagram 
$$
\xymatrix{
K_0(Z(\phi))  \ar@{^{}->>}[r]^{ }  \ar[rd]_{\sqrt{e}(E,s)}    &  K_0(\calX,\phi) \ar[d]^{\Spin(E,s)} \\
&  K_0(Z(s),\Z[1/2]),
}
$$
where the horizontal map is the proper pushforward along the inclusion $Z(\phi)\to \calX$, which is surjective by Proposition \ref{prop:ab_vs_D}. 
\end{prop}
\begin{proof}
By comparing \eqref{equ on spin gene} and \eqref{equ on hec}, it is enough to show the commutativity of diagram:
\begin{equation}\label{diag on k00}
\xymatrix{ 
K_0(Z(\rho^*\phi))  \ar[rrr]^{ \sqrt{e}(\rho^*E,\rho^*s,\Lambda_\rho) \quad \,\,\,\, }  \ar[d]^{ } & &  &   K_0(Z(\rho^*s),\Z[1/2])  \ar@{=}[d] \\ 
K_0(\tilde\calX,\rho^*\phi)   \ar[rrr]^{\Spin(\rho^*E,\Lambda_\rho,\rho^*s) \quad \quad}  &  &  & K_0(Z(\rho^*s),\Z[1/2]).
}
\end{equation}
Without loss of generality, we may assume $E$ has a maximal isotropic subbundle $\Lambda$, and hence replace $\tilde\calX$ by  $\calX$ and
drop $\rho$ in  diagram above. 
By an abuse of notation, we denote by $\calY_{\Lambda,E}$ \eqref{map of YtoX} its restriction to $Z(\phi)$. Given the commutativity of the following diagram
$$
\xymatrix{ 
K_0(Z(\phi))  \ar[rrr]^{ \sqrt{e}(E,s,\Lambda) \quad \,\,\,\, }  \ar[d]^{\pi_{\calY/\calX}^*} & & &   K_0(Z(s),\Z[1/2])  \ar[d]^{\pi_{\calY/\calX}^*} \\ 
K_0(\calY_{\Lambda,E})  \ar[rrr]^{\sqrt{e}(\pi_{\calY/\calX}^*E,\pi_{\calY/\calX}^*s,\pi_{\calY/\calX}^*\Lambda) \quad \,\,\,\, }  & &  &   K_0(\tilde Z(s),\Z[1/2]),
}
$$
 to show the commutativity of diagram \eqref{diag on k00}, we may assume without loss of generality that there is an isomorphism $E\cong \Lambda\oplus \Lambda^\vee$ of quadratic bundles.

Let $n=\rk(\Lambda)$, $(-1)^{\lvert \Lambda \rvert}=\pm 1$ be the sign of $\Lambda$ \cite[Def.~2.2]{OT}, and
$$\fc_n(\Lambda^\vee,s): K_0(Z(\phi)) \xrightarrow{sp} K_0(C_{Z(s_1)/Z(\phi)})   \xrightarrow{0^!_{\Lambda^\vee |_{Z(s_1)},s_0|_{Z(s_1)}}} K_0(Z(s)),  $$
where $s=(s_0,s_1)\in \Gamma(Z(\phi),\Lambda\oplus \Lambda^\vee)$, and $0^!_{\Lambda^\vee |_{Z(s_1)},s_0|_{Z(s_1)}}$ is the cosection localization of Kiem and Li \cite{KL2}. 
We recall that \cite[Eqns.~(99)]{OT}: 
$$\sqrt{e}(E,s,\Lambda)=(-1)^{n+|\Lambda|}\fc_n(\Lambda^\vee,s)\cdot\sqrt{\det\Lambda^\vee}.$$
By Proposition \ref{prop on ycover}, and \cite[Eqns.~(98),~(99)]{OT}, the commutativity of \eqref{diag on k00} is equivalent to the commutativity of the following diagram
\begin{equation}\label{diag on compar two}
\xymatrix{ 
K_0(Z(\phi))  \ar[rrr]^{ (-1)^{n+|\Lambda|}\fc_n(\Lambda^\vee,s) \quad \quad \quad \,\,\, }  \ar[d]^{ } & &  & K_0(\calX,Z(s))\cong K_0(Z(s))   \ar[r]^{\quad \,\,\, \sqrt{\det\Lambda^\vee}} &  K_0(Z(s),\Z[1/2]) \ar@{=}[d]  \\ 
K_0(\calX,\phi)   \ar[rrr]^{\otimes (-1)^{|\Lambda|} \Kos(\Lambda,s)\quad \quad \quad \quad} &  &   &  K_0(\calX,Z(s))\cong K_0(Z(s)) \ar[r]^{\quad \,\,\, \sqrt{\det\Lambda}}  & K_0(Z(s),\Z[1/2]).
}
\end{equation}
The restriction of $\Kos(\Lambda,s)$ to $Z(\phi)$ is a 2-periodic complex, which is written as $\{s_1,s_0 \}$ following \cite[Def.~3.1]{OS}. 
By definition \cite[Def.~2.13]{OS}, the composition 
$$K_0(Z(\phi))\to K_0(\calX,\phi) \xrightarrow{\otimes \Kos(\Lambda,s)}  K_0(Z(s))$$ 
is the $K$-theoretic localized Chern character of $\{s_1,s_0 \}$, denoted by $h(\{s_1,s_0 \})$.

Let $p: \Lambda^\vee|_{Z(s_1)}\to Z(s_1)$ be the projection,  $t\in \Gamma(\Lambda^\vee|_{Z(s_1)}, p^*\Lambda^\vee|_{Z(s_1)})$  the tautological section, 
and $p^*(s_0|_{Z(s_1)}) \in \Gamma(\Lambda^\vee|_{Z(s_1)},p^*\Lambda|_{Z(s_1)})$  the pullback section of $s_0|_{Z(s_1)}$.
Define a function
$$w=\langle p^{\ast}\left(s_0|_{Z(s_1)}\right), t\rangle: \Lambda^\vee|_{Z(s_1)}\to \C. $$
Then \cite[Thm.~3.5]{OS} yields
$$  h(\{p^*(s_0|_{Z(s_1)}),t\})=0^!_{\Lambda^\vee|_{Z(s_1)},s_0|_{Z(s_1)}}: K_0(Z(w))\to K_0(Z(s)). $$ 
By  a deformation to the normal cone argument as in  \cite[Lem.~2.4\,(2)]{KO}, the composition 
$$K_0(Z(\phi)) \xrightarrow{sp} K_0(C_{Z(s_1)/Z(\phi)})  \xrightarrow{ } 
K_0\left(Z(w)\right) \xrightarrow{h(\{p^*(s_0|_{Z(s_1)}),t\})} K_0(Z(s))  $$
equals to $h(\{s_0,s_1 \})$, where $C_{Z(s_1)/Z(\phi)} \subseteq Z(w)$ follows from \cite[Lem.~2.4\,(1)]{KO}. 

Finally, we have 
\begin{align*}h(\{s_0,s_1 \})&=h(\{s_1,s_0 \}\otimes \det\Lambda[-n]) \\
&=(-1)^n\cdot h(\{s_1,s_0 \})\,\otimes \det\Lambda, 
\end{align*}
where the first equality is by \cite[Lem.~2.3]{KO} and the  second equality uses \cite[Def.~2.13]{OS}. 
Combing the above, diagram \eqref{diag on compar two} commutes. 
\end{proof}
\begin{remark}\label{rmk on biv on spin}
The spinor morphism \eqref{equ on spin gene} commutes with DM proper pushforwards, flat pullbacks and Gysin pullbacks and hence is a bivariant class.
This can be proven using Proposition \ref{prop on cpr spin and old}, the surjectivity of the horizontal map in \textit{loc}.\,\textit{cit.}~and the fact that 
$K$-theoretic localized square root Euler classes are bivariant classes.
\end{remark}

\begin{remark}
In the setting of \S \ref{subsec:general_spin}, if $E$ admits a maximal isotropic subbundle $\Lambda$, compatible with the orientation of $E$,
then the map   
\begin{equation}K_0(\calX,\phi)\xrightarrow{\mathrm{Spin}(E,\Lambda,s)} 
K_0(Z(s),\Z[1/2]), \nonumber\end{equation}
agrees with the spinor morphism $\Spin(E,s)$ \eqref{equ on spin gene}. 
This can be proven similarly as above, by reducing to  
the fact that analogous claim holds for $K$-theoretic localized square root Euler classes.

\end{remark}

\subsection{Tautological spinor morphism}\label{sec:fund_fact}

Let $(E,Q)$ be an oriented non-degenerate quadratic bundle of even rank on $\calX$.  

Let $\rho:\tilde\calX\to \calX$ be the isotropic Grassmannian of $(E,Q)$, so that $\rho^*E$ 
has a tautological maximal isotropic subbundle $\Lambda_\rho$ \cite[\S 6]{EG} (see also \cite[\S 3.1]{OT}). 
Let $\pi_{\rho}: \rho^*E\to \tilde\calX$ be the projection and $s_{\rho}^t$ be the tautological section of the bundle $\pi_{\rho}^*\rho^*E\to  \rho^*E$.
It satisfies 
$$\pi_{\rho}^*\rho^*Q(s_{\rho}^t,s_{\rho}^t)=\rho^*q_E:  \rho^*E\to \C, $$
where $q_E: E\to \C$ is the function defined by $q_E(e)=Q(e,e)$. 
By multiplying $s_{\rho}^t$ by $\sqrt{-2}$, we obtain a section which satisfies $\pi_{\rho}^*\rho^*Q(\sqrt{-2}\cdot s_{\rho}^t,\sqrt{-2}\cdot s_{\rho}^t)=-2\rho^*q_E$.

\begin{definition}
The \textit{tautological spinor morphism} on $(\calX,E,Q)$ is the  following composition:  
\begin{equation}\label{equ on spin fa}\Spin:K_0(E,q_E)\xrightarrow{\rho^*}K_0(\rho^*E,\rho^*q_E)\xrightarrow{\Spin(\pi_{\rho}^*\rho^*E,\pi_{\rho}^*\Lambda_\rho,\sqrt{-2}\cdot s_{\rho}^t)}K_0(\tilde\calX,\Z[1/2])\xrightarrow{\rho_*}K_0(\calX,\Z[1/2]),
\end{equation}
where $q_E: E\to \C$ is the quadratic function of $E$, and $\rho_*$ is the  proper pushforward. 
\end{definition}
When the rank of $E$ is odd, we define the tautological spinor morphism $\Spin$ to be zero.

\section{Critical pullbacks}\label{sect on cri pb}
Using the specialization map in \S \ref{sect on sp m} and 
the spinor morphism in \S \ref{sec:spin}, we define critical pullback maps, and prove their basic properties. 

\subsection{Definition}\label{sect on const of crit pb}

\begin{definition}\label{defi of crit pb}
In Setting \ref{setting of lag}, the \textit{critical pullback} is the following composition:
\begin{align}\label{equ on crit pb der}f_{\pi_X}^!: K_0(X,\phi_X)\xrightarrow{\sigma_{\pi_X\circ f}} K_0(E,q_E)\xrightarrow{\Spin}K_0(M,\Z[1/2]) 
\xrightarrow{\otimes\sqrt{\det(V)^\vee}}  K_0(M,\Z[1/2]), \end{align} 
where $\sigma_{\pi_X\circ f}$ is the specialization map \eqref{equ on spe map on k}, and $\Spin$ is given by \eqref{equ on spin fa}. 
\end{definition}
\begin{remark}\label{rmk on T-equiv}
When there is a torus $T$ acting on \eqref{equ on -2 symp} and map $\pi_X: B\to X$, which preserves the shifted symplectic form $\Omega_{\textbf{\emph{f}}}$, and  
the function $\phi_X: X\to \C$. Then critical pullbacks lift to maps between
$T$-equivariant $K$-groups.
\end{remark}

\subsection{Compatibility with square root virtual pullback}\label{subsec:zero-loci}
From this section on, we assume that the critical locus $\Crit(\phi)$ is a closed substack of the zero locus $Z(\phi)$:
\begin{equation}\label{equ on crit emb to zero}\Crit(\phi)\hookrightarrow Z(\phi). \end{equation}
In Example \ref{ex of papers}, this follows from the assumption $\Crit(\phi_X)\hookrightarrow Z(\phi_X)$ in Setting \ref{setting of lag}. 

We have the following Cartesian diagram of classical stacks
\begin{equation}\label{diag on MZPHI}
\xymatrix{
M \ar[r]^{f_Z} \ar[d]_{}  \ar@{}[dr]|{\Box}  &  Z(\phi) \ar[d]^{} \\
M \ar[r]^{f} & B,  }
\end{equation}
where the base change of $M$ is still isomorphic to $M$ 
because $f$ factors through the critical locus $\Crit(\phi)$ (ref.~Proposition~\ref{thm on equi of two}) and hence the zero locus by inclusion \eqref{equ on crit emb to zero}. 
We can endow  $f_Z$ with a symmetric obstruction theory by pulling back the one on $f$.

In Setting \ref{setting of lag}, we also have a Cartesian diagram
\begin{equation}\label{diag on ZZ}
\xymatrix{
Z(\phi) \ar[r]^{\pi_X} \ar[d]_{}  \ar@{}[dr]|{\Box}  & Z(\phi_X)  \ar[d]^{} \\
B \ar[r]^{\pi_X} & X.  }
\end{equation}
The Cartesian diagram \eqref{diag on MZPHI} induces a closed embedding \cite[Prop.~2.26]{Man}:
$$\fC_{f_Z}\hookrightarrow  \fC_f. $$ 
We define $C_{f_Z}$ and $C_f$ by the following Cartesian diagrams 
$$\xymatrix{
C_{f_Z} \ar@{^{(}->}[r]^{ } \ar[d]_{p}  \ar@{}[dr]|{\Box}  & C_f \ar@{^{(}->}[r]^{ }  \ar[d]_{p}  \ar@{}[dr]|{\Box}  &  E \ar[d]^{p} \\
\fC_{f_Z} \ar@{^{(}->}[r]^{ }  & \fC_f  \ar@{^{(}->}[r]^{ }  & [E/V].  }
$$
By Proposition \ref{thm on equi of two}, the function $q_{\fC_f}: \fC_f\to \C$ restricts to zero on $\fC_{f_Z}$, so there is a 
$K$-theoretic square root virtual pullback \cite[App.~B.2]{Park1}:
\begin{equation}\label{equ on k sqr pb}
 \sqrt{f_Z^!}: K^{}_0(Z(\phi)) \xrightarrow{sp_{f_Z}}  K^{}_0(\fC_{f_Z})  \xrightarrow{p^*}  K^{}_0(C_{f_Z})  \xrightarrow{\sqrt{e}\left(\pi_E^*E|_{C_{f_Z}},\tau_E|_{C_{f_Z}}\right)}   K_0(M,\Z[1/2])  \xrightarrow{\sqrt{\det(V)^\vee}}  K_0(M,\Z[1/2]), \end{equation}
where $sp_{f_Z}$ is the specialization map, $\pi_E: E\to M$ is the projection and $\tau_E$ is its tautological section. 

\begin{prop}\label{prop:compare_OT}
Notations as above, there is a commutative diagram 
$$
\xymatrix{
K_0(Z(\phi_X))    \ar@{->>}[r]^{ }   \ar[rd]_{\sqrt{f_Z^!}\circ\pi_X^*\,\,\,\, }    &  K_0(X,\phi_X) \ar[d]^{f_{\pi_X}^!} \\
&  K_0(M,\Z[1/2]),
}
$$
where $\pi_X^*: K_0(Z(\phi_X))\to K_0(Z(\phi))$ is the flat pullback by $\pi_X$ in diagram \eqref{diag on ZZ}.
\end{prop}
\begin{proof}
Recall that the specialization map $\sigma_{\pi_X\circ f}$ \eqref{equ on spe map on k} satisfies the following commutative diagram  
$$
{\footnotesize
\xymatrix{
K^{}_0(Z(\phi_X)) \ar[r]^{\pi_X^*} \ar[d]_{ } & K^{}_0(Z(\phi)) \ar[r]^{sp\quad \,\,}  & K_0(Z^{\mathrm{}}(q_{\fC_f}))   \ar[r]^{p^*\,\,}  & 
K_0(Z^{\mathrm{}}(q_{E}|_{C_f}))    \ar[r]^{ }  &   
K_0(Z(q_E))  \ar[d]_{ }     \\ 
K^{}_0(X,\phi_X)\ar[rrrr]^{\sigma_{\pi_X\circ f}  } &  &  &  & K_0(E,q_E).
}}
$$
Here we use the nil-invariance of $K$-theory $K_0(Z^{\mathrm{ }}(-))\cong K_0(Z^{\mathrm{der}}(-))$ for $(-)=q_{\fC_f}, q_{E}|_{C_f}$. 

We compare the above with the first two maps in \eqref{equ on k sqr pb}. 
By \cite[Thm.~2.31]{Man}, the Cartesian diagram \eqref{diag on MZPHI} induces Cartesian diagrams 
$$ 
\xymatrix{
 \fC_{f_Z} \ar@{^{(}->}[r]^{ } \ar@{^{(}->}[d]^{ } \ar@{}[dr]|{\Box}   &  \mathring{M}_{f_Z}   \ar[d]^{ } \ar@{}[dr]|{\Box}  &  Z(\phi)\times \C^* \ar@{_{(}->}[l]^{ } \ar[d]_{}^{  }   \\
\fC_f  \ar@{^{(}->}[r]^{ } \ar[d]^{ }  \ar@{}[dr]|{\Box} & \mathring{M}_f    \ar[d]^{ } \ar@{}[dr]|{\Box} &  B\times \C^* \ar[d]^{ } \ar@{_{(}->}[l]^{ } \\
\{0\}  \ar@{^{(}->}[r]^{ }  & \mathbb{A}^1    &   \C^*  \ar@{_{(}->}[l]^{ }   }
$$
As the derived structure we put on $f_Z$ is the pullback one from $f$, the potential function on $\mathring{M}_{f_Z}$
as constructed from Lemma \ref{prop on fun on def} is the pullback of that function on $\mathring{M}_{f}$, which is zero. 
Then we have Cartesian diagrams 
$$ 
\xymatrix{
 \fC_{f_Z} \ar@{^{(}->}[r]^{ } \ar[d]^{ } \ar@{}[dr]|{\Box}   &  \mathring{M}_{f_Z}   \ar[d]^{ }  &    \\
Z^{\mathrm{ }}(q_{\fC_f }) \ar[d]^{ } \ar@{^{(}->}[r]^{ } \ar@{}[dr]|{\Box}   & Z^{\mathrm{ }}(\bar\phi)   \ar[d]^{ } \ar[r] \ar@{}[dr]|{\Box}  &  \{0\}    \ar[d]_{}^{  }   \\
\fC_f  \ar@{^{(}->}[r]^{ } \ar[d]^{ }  \ar@{}[dr]|{\Box} & \mathring{M}_f    \ar[d]^{ }  \ar[r]^{\bar\phi}    &   \mathbb{A}^1    \\
\{0\}  \ar@{^{(}->}[r]^{i}  & \mathbb{A}^1.    &       }
$$
Base change implies the commutativity of 
$$
\xymatrix{
 K^{}_0(\mathring{M}_{f_Z}) \ar[r]^{i^*}  \ar[d]^{}   & K_0(\fC_{f_Z}) \ar[d]^{ } \\
  K^{}_0(Z^{\mathrm{ }}(\bar\phi)) \ar[r]^{i^*}  & K_0(Z^{\mathrm{ }}(q_{\fC_f})).  }
$$
Combining with fact that $Z^{\mathrm{ }}(\bar\phi)\setminus Z^{\mathrm{ }}(q_{\fC_f})=\mathring{M}_{f_Z}\setminus \fC_{f_Z}=Z(\phi)\times \C^*$, we know 
the diagram 
$$
\xymatrix{
 K^{}_0(Z(\phi)) \ar[r]^{sp_{f_Z} }  \ar@{=}[d]   & K_0(\fC_{f_Z}) \ar[d]^{ } \\
  K^{}_0(Z(\phi)) \ar[r]^{sp  \, \,\,}  & K_0(Z^{\mathrm{ }}(q_{\fC_f}))  }
$$
commutes. Therefore we obtain commutative diagrams: 
$$
{\footnotesize
\xymatrix{
   K^{}_0(Z(\phi)) \ar[r]^{sp_{f_Z}} \ar@{=}[d]  & K^{}_0(\fC_{f_Z})\ar[r]^{p^*} \ar[d]_{ } & K^{}_0(C_{f_Z}) \ar[d]_{ }  \ar[r]_{ }  & K^{}_0(Z(q_E))  \ar@{=}[d]  
   \ar[rr]^{\sqrt{e}\left(\pi_E^*E,\tau_E\right)\quad \quad}  &  &   K_0(M,\Z[1/2])  \ar@{=}[d]     \\ 
 K^{}_0(Z(\phi)) \ar[r]^{sp  \, \,\,} & K_0(Z^{\mathrm{ }}(q_{\fC_f}))  \ar[r]^{p^*\,\, } & K_0(Z^{\mathrm{ }}(q_{E}|_{C_f}))  \ar[r]^{ } & K^{}_0(Z(q_E)) \ar[r]^{ }  & K^{}_0(E,q_E)  \ar[r]^{\Spin \quad \,\,\,} &K_0(M,\Z[1/2]),
}}
$$
where the commutativity of the right square follows from Proposition \ref{prop on cpr spin and old}.

It is obvious that the following commutes
$$
\xymatrix{
K^{}_0(C_{f_Z}) \ar[r]^{ }    \ar[rrrd]_{\sqrt{e}\left(\pi_E^*E|_{C_{f_Z}},\tau_E|_{C_{f_Z}}\right) \quad \quad\quad\,\, } &K^{}_0(Z(q_E))  
\ar[rr]^{\sqrt{e}\left(\pi_E^*E,\tau_E\right)\quad} & & K_0(M,\Z[1/2])  \ar@{=}[d]    \\ 
 & & &  K_0(M,\Z[1/2]).
}
$$
Therefore we are done. 
 \end{proof}

\subsection{Properties of critical pullbacks} 
\subsubsection{Independence of resolutions}
\begin{prop}\label{prop:indep}
The critical pullback is independent of the choice of the symmetric resolution in Setting \ref{setting of lag}
\end{prop}
\begin{proof}
By Proposition \ref{prop:compare_OT}, 
we have a commutative diagram 
$$
\xymatrix{
K^{}_0(Z(\phi_X)) \ar[r]^{\pi_X^* }\ar@{->>}[d]^{c}   &  K^{}_0(Z(\phi))   \ar[r]^{ \sqrt{f_Z^!} \quad \,\,\, } &  K^{}_0(M,\Z[1/2])  \ar@{=}[d]  \\
K^{}_0(X,\phi_X) \ar[rr]^{f_{\pi_X}^! }  &   &   K^{}_0(M,\Z[1/2]),    }
$$
i.e.~the critical pullback satisfies that 
$$f_{\pi_X}^!= \sqrt{f_Z^!}  \circ \pi_X^*\circ c^{-1}. $$
Since $\sqrt{f_Z^!}$ is independent of the choice of resolution \cite[Lem.~1.14]{Park1}, so is $f_{\pi_X}^!$. 
\end{proof}

\subsubsection{Bivariance}

Given a pullback diagram of derived stacks (below $B_1,B_2$ are classical stacks and viewed as derived stacks via the inclusion functor):
\begin{equation}\label{diag on comm of pp}
\xymatrix{
\textbf{\emph{M}}_2 \ar[r]^{\textbf{\emph{f}}_2} \ar[d]_{\textbf{\emph{h}}}  \ar@{}[dr]|{\Box}  &  B_2 \ar[d]^{k} \\
\textbf{\emph{M}}_1 \ar[r]^{\textbf{\emph{f}}_1} & B_1, 
}
\end{equation}
and a Cartesian diagram of classical stacks: 
\begin{equation}\label{diag on comm of pp3}
\xymatrix{
B_2 \ar[r]^{\pi_{X_2}} \ar[d]_{k}  \ar@{}[dr]|{\Box}    &  X_2 \ar[d]^{l} \\
B_1 \ar[r]^{\pi_{X_1}} & X_1, 
}
\end{equation}
where $(\textbf{\emph{M}}_i,B_i,\textbf{\emph{f}}_i,X_i,\pi_{X_i})$ ($i=1,2$) satisfy conditions in Setting \ref{setting of lag}, such that 
\begin{itemize}
\item
the $(-2)$-shifted symplectic structure on $\textbf{\emph{f}}_2$ is the pullback $\Omega_{\textbf{\emph{f}}_2}=\textbf{\emph{h}}^*\Omega_{\textbf{\emph{f}}_1}$
from $\textbf{\emph{f}}_1$
\item
the regular functions $\phi_{X_i}: X_i\to \C$ satisfy that 
$\phi_{i}=\pi_{X_i}^*\phi_{X_i}$ ($i=1,2$) and $\phi_{X_2}=\phi_{X_1}\circ l$.
\end{itemize}
Then critical pullbacks
$$f_{\pi_{X_i}}^!: K_0(X_i,\phi_{X_i})\to K_0(M_i,\Z[1/2]), \quad i=1,2 $$ 
satisfy the following bivariance properties. 
\begin{prop}\label{pullback comm with base change}
We have the following:
\begin{itemize}
\item If $l$ is flat, then 
$$h^* \circ f_{\pi_{X_1}}^!=f_{\pi_{X_2}}^! \circ l^*.$$
\item If $l$ is proper, then 
$$h_* \circ f_{\pi_{X_2}}^!=f_{\pi_{X_1}}^! \circ l_*. $$
\item If $l$ is a regular embedding, then 
$$l^! \circ f_{\pi_{X_1}}^!=f_{\pi_{X_2}}^! \circ l^*,$$
where $l^!: K_0(M_1,\Z[1/2])\to K_0(M_2,\Z[1/2])$ is the Gysin pullback. 
\end{itemize}
\end{prop}
\begin{proof}
Similar to the proof of Proposition~\ref{prop:indep} above, by using 
Proposition~\ref{prop:compare_OT}, the claim is reduced to the bivariance of square root virtual pullbacks \cite[Prop.~1.15]{Park1}.
\end{proof}

\subsubsection{Functoriality }

Given a commutative diagram of derived stacks
\begin{equation}
   \label{eqn on comm diag func}
   \xymatrix{
\textbf{\emph{N}}' \ar[r]^{\textbf{\emph{i}}} \ar[d]_{\textbf{\emph{f}}}     &  \textbf{\emph{N}} \ar[d]^{\textbf{\emph{h}}} \\
\textbf{\emph{M}} \ar[r]^{\textbf{\emph{g}}} & \textbf{\emph{B}},  }
\end{equation}
where the classical truncation $i:=t_0(\textbf{\emph{i}})$ of $\textbf{\emph{i}}$ is an \textit{isomorphism}, and $B:=t_0(\textbf{\emph{B}})\cong \textbf{\emph{B}}$, we write the obstruction theories: 
$$\phi_f:\bbE_f\to \bfL_f, \quad \phi_g:\bbE_g\to \bfL_g, \quad \phi_{h}:\bbE_{h}\to \bfL_{h}, \quad  \phi'_{h}:\bbD:=\bbE_{h\circ i}\to \bfL_{h\circ i}\cong \bfL_{h}, $$ 
obtained by restricting cotangent complexes of $\textbf{\emph{f}}$, $\textbf{\emph{g}}$, $\textbf{\emph{h}}$, $\textbf{\emph{h}}\circ \textbf{\emph{i}}$ to their classical truncations $f,g,h,h\circ i\cong h$ (here we use $i$ is an isomorphism), and then taking 
$[-1,0]$-truncations (here $\bfL_{(-)}:=\tau^{\geqslant -1}\bbL_{(-)}$ denotes the truncated cotangent complex). 
The diagram \eqref{eqn on comm diag func} implies two maps $\alpha:\bbE_{h}\to \bbD$ and $\beta:f^*\bbE_g\to \bbD$ such that 
$\phi'_{h}\circ \alpha=\phi_{h}$.

We say the above obstruction theories are {\it compatible},~if we have a commutative diagram:
\begin{equation}
   \label{eqn:triang_obst}
\xymatrix{\bbD^\vee[2]\ar[r]^{\alpha^\vee}\ar[d]_{\beta^\vee}&\bbE_{h}\ar[d]_\alpha\ar[r]^\delta&\bbE_f \ar@{=}[d] \\
f^*\bbE_g\ar[r]^\beta\ar[d]_{f^*\phi_g}&\bbD\ar[r]^\gamma\ar[d]_{\phi'_{h}}&\bbE_f\ar[d]_{\phi'_f}\\
\tau^{\geqslant -1}f^*\bfL_g\ar[r]& \bfL_{h}\ar[r]&\bfL_f'.}
\end{equation}
Here horizontal lines are exact triangles (the bottom one defines $\bfL_f'$), and $\phi_f=r\circ \phi'_f$ 
(where $r: \bfL_f'\to \tau^{\geqslant -1}\bfL_f'\cong \bfL_f$ is the canonical map), and the orientation of $\bbE_{h}$ is compatible with the orientation of $\bbE_{g}$ (ref.~\cite[Def.~2.1]{Park1}).

\begin{theorem}\label{thm on funct}
Given a commutative diagram \eqref{eqn on comm diag func} of derived stacks
such that 
\begin{itemize}
\item $\textbf{{g}}$ and $\textbf{{h}}$ satisfy conditions in Setting \ref{setting of lag} for a common $\pi_X: B\to X$ and $\phi_X: X\to \C$,
\item $\textbf{{f}}$ is quasi-smooth and the classical truncation of $\textbf{{i}}$ is an isomorphism,
\item the obstruction theories are compatible, 
\end{itemize}
then we have   
\begin{equation}\label{equ on fun of crit pb}h_{\pi_X}^!=f^!\circ g_{\pi_X}^!: K_0(X,\phi_X)\to K_0(N,\Z[1/2]), \end{equation}
where $g_{\pi_X}^!,h_{\pi_X}^!$ are critical pullbacks \eqref{equ on crit pb der} and $f^!: K_0(M)\to K_0(N)$ is the virtual pullback \cite{Qu}. 
\end{theorem}
\begin{proof}
As in above, by using Proposition~\ref{prop:compare_OT}, the claim follows from the functoriality of 
square root virtual pullbacks \cite[Thm.~2.2]{Park1}.
\end{proof}
\begin{remark}\label{rmk:lag_cl}
It is straightforward to check that
if the induced map $\textbf{\emph{N}}'\to \textbf{\emph{N}} \times_{\textbf{\emph{B}}} \textbf{\emph{M}}$ has a $(-2)$-shifted Lagrangian structure (relative to $\textbf{\emph{B}}$),
where $\textbf{\emph{N}}\to \textbf{\emph{B}}$ is endowed with the inverse shifted symplectic structure, or more generally, if $\bbL_{\textbf{\emph{N}}'/\textbf{\emph{B}}}$ is maximal isotropic (quotient) complex of the symmetric complex $\bbL_{\textbf{\emph{N}} \times_{\textbf{\emph{B}}} \textbf{\emph{M}}/\textbf{\emph{B}}}$, 
then the compatibility condition in \eqref{eqn:triang_obst} holds.
\end{remark}
\begin{remark}\label{rmk2 on T-equiv}
As in Remark \ref{rmk on T-equiv}, in the presence of a torus action, the above listed properties hold in the $T$-equivariant setting. 
\end{remark}

\section{Category of Lagrangian correspondences and its linearization} 
\label{sec:linearization}
In this section, as an application of the theory of critical pullbacks, we define a linearization of the category of Lagrangian correspondences.
\subsection{Lagrangian correspondence} 
Consider derived critical loci $\textbf{Crit}(\boldsymbol{\phi}_i)$ for $\boldsymbol{\phi}_i: \textbf{\emph{B}}_i\to \C$ ($i=1,2$) (Definition \ref{defi of derived crit loci}) with $(-1)$-shifted symplectic structure $\Omega_i$. The fiber product  
$$\textbf{Crit}(\boldsymbol{\phi}_1)\times \textbf{Crit}(\boldsymbol{\phi}_2)$$
has a natural $(-1)$-shifted symplectic structure $\Omega_1 \boxplus \Omega_2$.  

\begin{definition}\label{def of lag corr}
An oriented (exact) \textit{Lagrangian correspondence} of $\textbf{Crit}(\boldsymbol{\phi}_i)$ ($i=1,2$) is an (exact) Lagrangian 
$$\textbf{\emph{f}}_{12}: \textbf{\emph{M}}_{12}\to \textbf{Crit}(-\boldsymbol{\phi}_1)\times \textbf{Crit}(\boldsymbol{\phi}_2), $$
where the target is endowed with $(-1)$-shifted symplectic structure $(-\Omega_1) \boxplus \Omega_2$, and 
$\textbf{\emph{f}}_{12}$ has an orientation in the sense of Definition \ref{def of rel or} (here the canonical orientation of the target is chosen as in Remark \ref{rmk on can ori}).  
\end{definition}
\begin{definition}\label{def of comp lag corr}
Given a Lagrangian correspondence of $\textbf{Crit}(\boldsymbol{\phi}_i)$ ($i=1,2$):
$$\textbf{\emph{f}}_{12}: \textbf{\emph{M}}_{12}\to \textbf{Crit}(-\boldsymbol{\phi}_1)\times \textbf{Crit}(\boldsymbol{\phi}_2), $$
and a Lagrangian correspondence of $\textbf{Crit}(\boldsymbol{\phi}_i)$ ($i=2,3$): 
$$\textbf{\emph{f}}_{23}: \textbf{\emph{M}}_{23}\to \textbf{Crit}(-\boldsymbol{\phi}_2)\times \textbf{Crit}(\boldsymbol{\phi}_3),$$
their \textit{composed Lagrangian correspondence} is the following Lagrangian~\cite[Cor.~2.13]{AB}: 
\begin{equation}\label{equ on M13}\textbf{\emph{M}}_{13}:=\textbf{\emph{M}}_{12}\times_{\textbf{Crit}(\boldsymbol{\phi}_2)} \textbf{\emph{M}}_{23}, \quad \textbf{\emph{f}}_{13}:=(\textbf{\emph{f}}_{12}\times_{\textbf{Crit}(\boldsymbol{\phi}_2)} \textbf{\emph{f}}_{23})_{13}: \textbf{\emph{M}}_{13}\to\textbf{Crit}(-\boldsymbol{\phi}_1)\times \textbf{Crit}(\boldsymbol{\phi}_3), 
\end{equation}
where 
$(-)_{13}$ denotes the projection 
$$\textbf{Crit}(-\boldsymbol{\phi}_1)\times \textbf{Crit}(\boldsymbol{\phi}_2)\times \textbf{Crit}(\boldsymbol{\phi}_3)
\to \textbf{Crit}(-\boldsymbol{\phi}_1)\times \textbf{Crit}(\boldsymbol{\phi}_3) $$
to the 1st and 3rd components, and we also use the canonical equivalence 
$$(\textbf{Crit}(\boldsymbol{\phi}_2),\Omega_2)\cong (\textbf{Crit}(-\boldsymbol{\phi}_2),-\Omega_2). $$
\end{definition}
\begin{remark}
If $\textbf{\emph{f}}_{12}$ and $\textbf{\emph{f}}_{23}$ are (oriented) exact Lagrangian correspondences, then $\textbf{\emph{f}}_{13}$ is also (oriented) exact. 
\end{remark}
\begin{proof}
Derived critical loci are exact symplectic, the composition of exact Lagrangian correspondences is still exact \cite[\S 2.2]{Park2}. The orientation follows from 
a direct calculation of cotangent complexes and the pairing on them, e.g.~\cite[Lem.~5.6]{AB}. 
\end{proof}

The above composition of Lagrangian correspondences is pictured in the following diagram: 
$$
{\footnotesize
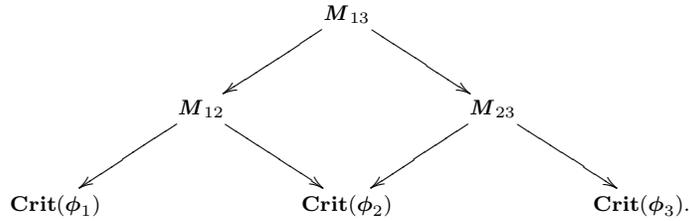
\xymatrix{
& & \textbf{\emph{M}}_{13}\ar[dr]^{ } \ar[dl]^{ }  & & \\
  &  \textbf{\emph{M}}_{12} \ar[dr]^{ } \ar[dl]^{ } &    & \textbf{\emph{M}}_{23} \ar[dr]^{ } \ar[dl]^{ }   &  \\
\textbf{Crit}(\boldsymbol{\phi}_1)   & & \textbf{Crit}(\boldsymbol{\phi}_2) & &  \textbf{Crit}(\boldsymbol{\phi}_3). }
}
$$

\subsection{Lagrangian category}\label{sect on lag cat}





We refer to \cite[\S 4]{AB} for the general construction of 2-category of Lagrangians in a shifted symplectic stack. 
Here we restrict to the following setting:  
\begin{itemize}
\item
Consider the 2-category $\textbf{Lag}_{\mathrm{crit}}^{(2)}$ of Lagrangians in derived critical loci:
objects are pairs $(X,\phi)$ of quotient stack $X$ of smooth quasi-projective scheme by linear algebraic group and flat regular function $\phi$,
1-morphisms are oriented exact Lagrangian correspondences of 
$\bCrit^{}({\phi}), \bCrit^{}({\phi'})$ (Definition~\ref{def of lag corr}), 
2-morphisms are exact Lagrangeomorphisms \cite[\S 3]{AB} which preserve orientations. The identity functor for any $(X,\phi)$ is given by the 
diagonal Lagrangian $\bCrit^{}({\phi})\to \bCrit^{}({\phi})\times \bCrit^{}({\phi})$.
\item
Let $\textbf{Lag}_{\mathrm{crit}}^{}$ denote the 1-truncation of the 2-category $\textbf{Lag}_{\mathrm{crit}}^{(2)}$, which has isomorphism classes of oriented exact Lagrangians as morphisms. 
\end{itemize}
Let $\textbf{Lag}_{\mathrm{crit}}^{\mathrm{DM}}$ be the non-unital subcategory 
of $\textbf{Lag}_{\mathrm{crit}}^{}$ in which objects $X$ are Deligne-Mumford, morphisms and their compositions 
are proper oriented exact Lagrangians satisfying \textit{``DM"} and \textit{``Resolution"} conditions in Setting \ref{setting of lag}. 

In the presence of a torus $T$-action, there is a category 
$$\textbf{Lag}_{\mathrm{crit},T}^{\mathrm{DM}} $$
of $T$-equivariantly proper\,\footnote{This means the $T$-fixed locus of the Lagrangian is proper.} oriented exact Lagrangians, 
which recovers $\textbf{Lag}_{\mathrm{crit}}^{\mathrm{DM}}$ when $T=\{\id\}$.

The category $\textbf{Lag}_{\mathrm{crit}}$ is a \textit{monoidal} category:  
$$\otimes: \textbf{Lag}_{\mathrm{crit}} \times \textbf{Lag}_{\mathrm{crit}} \to \textbf{Lag}_{\mathrm{crit}}, $$
where at the level of objects, the tensor structure is given by 
$$(X_1,\phi_{1}), (X_2,\phi_{2})\mapsto 
(X_1\times X_2,\phi_{1}\boxplus \phi_{2} ),$$
with sum function   
$$\phi_{1}\boxplus \phi_{2}: X_1\times X_2\to \C, \quad (x_1,x_2)\mapsto \phi_{1}(x_1)+\phi_{2}(x_2). $$
At the level of morphisms, the tensor structure $\otimes$ is given by the product of Lagrangians \cite[Prop.~2.6]{AB}, 
where one uses the fact that there is a canonical equivalence  
$$ \textbf{{Crit}}(\phi_{1}\boxplus \phi_{2})\cong \textbf{{Crit}}(\phi_{1})\times \textbf{{Crit}}(\phi_{2}). $$
The category $\textbf{Lag}_{\mathrm{crit}}$ has a \textit{dualizing} functor 
$$\textbf{Lag}_{\mathrm{crit}} \to \textbf{Lag}_{\mathrm{crit}}, $$
where at the level of objects, it is given by 
$$(X,\phi)\mapsto (X,-\phi), $$
and at the level of morphisms, it is given by the canonical bijection between (oriented, exact) Lagrangians in 
$\textbf{{Crit}}(\phi)$ and $\textbf{{Crit}}(-\phi)$ \cite[Prop.~2.6]{AB}.

The category $\textbf{Lag}_{\mathrm{crit}}^{\mathrm{DM}}$ is a monoidal non-unital subcategory of $\textbf{Lag}_{\mathrm{crit}}$ with induced 
dualizing functor. Similarly for the $T$-equivariant version $\textbf{Lag}_{\mathrm{crit},T}^{\mathrm{DM}}$.

\subsection{Critical pullbacks under compositions of Lagrangian correspondences}\label{sec on cpb and lagcor}
Continued with the setting in \S \ref{sect on lag cat}. Given morphisms $\textbf{\emph{f}}_{12}, \textbf{\emph{f}}_{23}$ in $\textbf{Lag}_{\mathrm{crit},T}^{\mathrm{DM}}$
and their composition $\textbf{\emph{f}}_{13}$ in the sense of Definition \ref{def of comp lag corr}. 
By abuse of notations, we denote their composition with the inclusion 
$$\textbf{Crit}(-\phi_i)\times \textbf{Crit}(\phi_j)\to X_i\times X_j $$
also by  
\begin{equation}\label{equ on fij symp}\textbf{\emph{f}}_{ij}: \textbf{\emph{M}}_{ij} \to X_i\times X_j. \end{equation} 
By Proposition \ref{thm on equi of two}, 
$\textbf{\emph{f}}_{ij}$ \eqref{equ on fij symp} has a relative $(-2)$-shifted symplectic structure
$\Omega_{\textbf{\emph{f}}_{ij}}$ with 
$$[\Omega_{\textbf{\emph{f}}_{ij}}]=(-\phi_i\boxplus \phi_j) |_{\textbf{\emph{M}}_{ij}}. $$ 
Clearly, the product $$\textbf{\emph{f}}_{12}\times \textbf{\emph{f}}_{23}: \textbf{\emph{M}}_{12}\times \textbf{\emph{M}}_{23} \to X_1\times X_2\times X_2\times X_3 $$
has a product relative $(-2)$-shifted symplectic structure $\Omega_{\textbf{\emph{f}}_{12}}\boxplus \Omega_{\textbf{\emph{f}}_{23}}$ with 
$$[\Omega_{\textbf{\emph{f}}_{12}}\boxplus \Omega_{\textbf{\emph{f}}_{23}}]=
(-\phi_1\boxplus \phi_2\boxplus -\phi_2\boxplus \phi_3) |_{\textbf{\emph{M}}_{12}\times \textbf{\emph{M}}_{23}}. $$
Consider the diagonal morphism
$${\Delta}^{123}_{1223}: X_1\times X_2 \times X_3\to 
X_1\times X_2\times X_2\times X_3, $$
$$(x_1,x_2,x_3)\mapsto (x_1,x_2,x_2,x_3), $$
and a homotopy pullback diagram of derived stacks 
\begin{equation}\label{equ on define m1223b}
\xymatrix{
\textbf{\emph{M}}_{12}\times_{X_2} \textbf{\emph{M}}_{23} \ar[r]^{ } \ar[d]_{\textbf{\emph{f}}_{12}\times_{X_2} \textbf{\emph{f}}_{23}}  \ar@{}[dr]|{\Box}  &  \textbf{\emph{M}}_{12}\times \textbf{\emph{M}}_{23} 
\ar[d]^{\textbf{\emph{f}}_{12}\times \textbf{\emph{f}}_{23}} \\
X_1\times  X_2\times X_3\ar[r]^{{\Delta}^{123}_{1223} \quad \,\,\, } \ar[d]_{\pi^{123}_{13}} & X_1\times X_2\times X_2\times X_3 \\
X_1\times X_3, &  }
\end{equation}
where $\pi^{123}_{13}$ is the natural projection. 
Under pullback, the symplectic structure on $\textbf{\emph{f}}_{12}\times \textbf{\emph{f}}_{23}$ defines a symplectic structure
on $\textbf{\emph{f}}_{12}\times_{X_2} \textbf{\emph{f}}_{23}$ with 
$$[\Omega_{\textbf{\emph{f}}_{12}\times_{X_2} \textbf{\emph{f}}_{23}}]=
(-\phi_1\boxplus  \phi_3) |_{\textbf{\emph{M}}_{12}\times_{X_2} \textbf{\emph{M}}_{23}}. $$
Note the Cartesian diagrams of classical stacks: 
$$
\xymatrix{
M_{12}\times_{\Crit(\phi_2)}  M_{23} \ar@{}[dr]|{\Box}  \ar[r]^{ } \ar[d]_{ }     &  M_{12}\times_{}  M_{23}   \ar[d]^{ } \\
\Crit(\phi_2) \ar[r]^{ }  \ar[d]_{ }  \ar@{}[dr]|{\Box}  & \Crit(\phi_2) \times \Crit(\phi_2) \ar[d]_{ } \\ 
X_2 \ar[r]^{ }    & X_2\times X_2,
}
$$
which implies an isomorphism of classical stacks: 
$$M_{12}\times_{\Crit(\phi_2)}  M_{23}\cong M_{12}\times_{X_2}  M_{23}. $$
Definition \ref{defi of crit pb} gives critical pullbacks:
$$(f_{12}\times_{X_2}f_{23})_{}^!: K_0(X_1\times  X_2\times X_3,-\phi_{1}\boxplus0 \boxplus \phi_{3})\to K_0(M_{12}\times_{X_2} M_{23}),  $$
$$(f_{13})_{}^!: K_0(X_1\times X_3,-\phi_{1} \boxplus \phi_{3})\to K_0(M_{12}\times_{\mathrm{Crit}(\phi_2)} M_{23}). $$ 
Here in view of Setting \ref{setting of lag}, we are in the situation where $B=X$ and $\pi_X$ is the identity map, 
so we simply denote the critical pullback $f^!_{\pi_X}$ (of Definition \ref{defi of crit pb}) as $f^!$.
\begin{theorem}\label{thm:functorial}
We have 
$$(f_{13})_{}^!=(f_{12}\times_{X_2}f_{23})_{}^!\circ \left(\pi^{{123}}_{{13}}\right)^*, $$
where $\left(\pi^{{123}}_{{13}}\right)^*$ is the smooth pullback.
\end{theorem}
\begin{proof}
By abuse of notations, the composition of  
$$ \textbf{\emph{f}}_{12}\times_{\bCrit(\phi_2) } \textbf{\emph{f}}_{23}: \textbf{\emph{M}}_{12}\times_{\bCrit(\phi_2)} \textbf{\emph{M}}_{23}\to 
\textbf{Crit}(-{\phi}_1)\times \textbf{Crit}({\phi}_2)\times \textbf{Crit}({\phi}_3)$$
and the inclusion 
$$\textbf{Crit}(-{\phi}_1)\times \textbf{Crit}({\phi}_2)\times \textbf{Crit}({\phi}_3)\to X_1\times X_2 \times X_3 $$
is still denoted as $\textbf{\emph{f}}_{12}\times_{\bCrit(\phi_2) } \textbf{\emph{f}}_{23}$. 
We have the following diagram
\begin{equation}\label{diag:s exi}
\xymatrix{
\textbf{\emph{M}}_{12}\times_{\bCrit(\phi_2)} \textbf{\emph{M}}_{23} \times_{X_1\times  X_3} X_1\times  X_2\times X_3 \ar[r]^{\quad \quad \quad \quad \quad \,\,\bar{\pi}^{123}_{13} }     \ar[d]_{\textbf{\emph{g}}}     
   &   \ar@{.>}@/_1.5pc/[l]_(.5){\textbf{\emph{s}}}  \textbf{\emph{M}}_{12}\times_{\bCrit(\phi_2)} \textbf{\emph{M}}_{23} \ar[ld]^{\quad\quad\,\,\, 
   \textbf{\emph{f}}_{12}\times_{\bCrit(\phi_2) } \textbf{\emph{f}}_{23}  }  
\ar[d]^{\textbf{\emph{f}}_{13}} \\
 X_1\times  X_2\times X_3    \ar[r]_{\quad \pi^{123}_{13}}  & X_1\times  X_3,     }
\end{equation}
where the square is a homotopy pullback diagram (which defines $\bar{\pi}^{123}_{13}$, $\textbf{\emph{g}}$), and the lower right triangle commutes 
by \eqref{equ on M13}.
This implies the map $\bar{\pi}^{123}_{13}$ has a section 
$\textbf{\emph{s}}$ which satisfies 
\begin{equation}\label{equ on sect s}\bar{\pi}^{123}_{13}\circ \textbf{\emph{s}}=\id. \end{equation}
The maps $\textbf{\emph{s}}$, $\textbf{\emph{g}}$ fit into commutative diagrams 
\begin{equation} \label{diag:s_g}
\xymatrix{
\textbf{\emph{M}}_{12}\times_{\bCrit(\phi_2)} \textbf{\emph{M}}_{23} \ar[r]^{\emph{\textbf{i}} } \ar[d]_{\textbf{\emph{s}}} \ar[rd]^{\quad \textbf{\emph{f}}_{12}\times_{\bCrit(\phi_2) } \textbf{\emph{f}}_{23}}   &  
\textbf{\emph{M}}_{12}\times_{X_2} \textbf{\emph{M}}_{23}
\ar[d]^{\textbf{\emph{f}}_{12}\times_{X_2} \textbf{\emph{f}}_{23}} \\
\textbf{\emph{M}}_{12}\times_{\bCrit(\phi_2)} \textbf{\emph{M}}_{23} \times_{X_1\times  X_3} X_1\times  X_2\times X_3\ar[r]^{\quad \quad\quad\quad \quad\quad \textbf{\emph{g}} }  & X_1\times X_2 \times X_3. }
\end{equation}
As $X_2$ is smooth, hence $\textbf{\emph{s}}$ is quasi-smooth. As $\textbf{\emph{g}}$ is the pullback of $\textbf{\emph{f}}_{13}$ and $\textbf{\emph{f}}_{12}\times_{X_2} \textbf{\emph{f}}_{23}$ is the pullback of $\textbf{\emph{f}}_{12}\times \textbf{\emph{f}}_{23}$ \eqref{equ on define m1223b}, 
the image of their $(-2)$-symplectic forms in the periodic cyclic homology are given by the pullback of the same function $-\phi_1\boxplus \phi_3$.
In below, we show how to verify remaining conditions in Theorem \ref{thm on funct} for diagram \eqref{diag:s_g}, 
which then gives
\begin{equation}\label{equ on f123}\left(f_{12}\times_{X_2} f_{23}\right)_{}^!=s^!\circ g_{}^!. \end{equation}
First of all, we claim that there is a commutative diagram
\begin{equation} \label{diag:lag_cor}
\xymatrix{
(\textbf{\emph{M}}_{12}\times_{\textbf{Crit}(\phi_2)} \textbf{\emph{M}}_{23})\ar[d]_{}\ar[r]&(\textbf{\emph{M}}_{12}\times \textbf{\emph{M}}_{23})\times_{X_2^2}\textbf{Crit}(\phi_2) \ar[d]\\
(\textbf{\emph{M}}_{12}\times_{\textbf{Crit}(\phi_2)} \textbf{\emph{M}}_{23})\times_{ X_1\times X_3 }{X_1\times X_2\times X_2\times X_3 }\ar[r]_{\quad \quad\quad\quad \quad\quad}& X_1\times X_2\times X_2\times X_3 
  }
\end{equation}
such that the  $\textbf{\emph{M}}_{12}\times_{\textbf{Crit}(\phi_2)} \textbf{\emph{M}}_{23}$  
is a Lagrangian correspondence between the upper-right and lower-left corners relative to $X_1\times X_2\times X_2\times X_3$. 

This follows from  Lagrangian triple intersection theorem \cite[Thm.~3.1]{Ben}, by a similar argument as that of \cite[Lem.~2.4.3]{Park2}. 
Indeed, we simplify notations and write $X_{13}= X_1\times X_3 $, similarly for $X_{1223}$, $X_{22}$. We write $C_i:=\textbf{Crit}(\phi_i)$ for $i=1,2,3$,  $C_{ij}=C_i\times C_j$, and similarly $C_{1223}$, $C_{22}X_{13}$, $M_{12}M_{23}$, etc, stand for the product. 
The Lagrangian fibration $C_{13}\to X_{13}$ induces a Lagrangian $C_{13}\to C_{13}X_{13}$ relative to the base $X_{13}$. By base change, 
\begin{equation}\label{equ on lag c13}C_{13}X_{22}\to C_{13}X_{1223} \end{equation}is a Lagrangian relative to the base $X_{1223}$. 

By Definition \ref{def of lag corr} and base change, we know 
$M_{12}M_{23}X_{1223}\to C_{1223}X_{1223}$ is a Lagrangian relative to the base $X_{1223}$. Then we have composition of Lagrangian correspondences relative to $X_{1223}$: 
$$
{\footnotesize
\xymatrix{
& & C_{13}X_{22}\times_{C_{13}X_{1223}}M_{12}M_{23}X_{1223} \ar[dr]^{ } \ar[dl]^{ }  & & \\
  &  C_{13}X_{22}\ar[dr]^{ } \ar[dl]^{ } &    & M_{12}M_{23}X_{1223} \ar[dr]^{ } \ar[dl]^{ }   &  \\
X_{1223}   & & C_{13}X_{1223} & &  C_{22}X_{1223}, }
}
$$
i.e.~$L_1:=C_{13}X_{22}\times_{C_{13}X_{1223}}M_{12}M_{23}X_{1223}\to M:=C_{22}X_{1223}$ is a Lagrangian relative to $X=X_{1223}$. 

We have two more Lagrangians on $M$ relative to $X$, $L_2:=C_2X_{1223}\to M$ given by diagonal map $C_{2}\to C_{22}$, and $L_3:=C_{22}X_{13}\to M$ 
constructed as \eqref{equ on lag c13}. 
A direct calculation shows that 
\begin{equation}
\label{eqn:Lag_triple}L_1\times_ML_2\times_ML_3\cong C_2\times_{C_{22}}M_{12}M_{23}, 
\end{equation}
\begin{equation}
    \label{eqn:symp_triple}(L_1\times_ML_2)\times_X(L_2\times_ML_3)\times_X(L_3\times_ML_1)\cong C_2\times_{X_{22}}(M_{12}\times_{C_2}M_{23})\times_{X_{13}}M_{12}M_{23}.
\end{equation}
Then Lagrangian triple intersection theorem  \cite[Thm.~3.1]{Ben} yields that the \eqref{eqn:Lag_triple} is Lagrangian in the \eqref{eqn:symp_triple}. This implies diagram \eqref{diag:lag_cor} and the claimed Lagrangian property. 

We write $M_{13}:=C_2\times_{C_{22}}M_{12}M_{23}$.
By base-change of diagram \eqref{diag:lag_cor} along the map ${\Delta}^{123}_{1223} :X_{123}\to X_{1223}$, we obtain a Lagrangian correspondence
\[\xymatrix{
M_{13}\times_{X_{13}}X_{123}&\ar[l]\ar[r]M_{13}\times_{X_{22}}X_2&M_{12}\times_{X_2}M_{23}\times_{X_{22}}C_2
}\] 
relative to $X_{123}$. That implies $\bbL_{M_{13}/X_{1223}}\cong \bbL_{M_{13}\times_{X_{22}}X_2/X_{123}}$ is a maximal isotropic quotient complex of  
\begin{equation}\label{equ on lag 100}\bbL_{C_2\times_{X_{22}}(M_{13})\times_{X_{13}}M_{12}M_{23}/X_{1223}}\cong \bbL_{C_2\times_{X_{22}}(M_{13})\times_{X_{13}}M_{12}\times_{X_2}M_{23}/X_{123}}, \end{equation}
where we omit pullbacks of cotangent complexes in the above two identifications, also for the remaining argument. 

Notice also that, we have commutative diagram:
\[\xymatrix{
\bbL_{C_2/X_{22}}    \ar[r]^{\cong \quad\quad\quad\quad\quad}   & \bbL_{M_{12}\times_{X_2}M_{23}\times_{X_{22}}C_2/M_{12}\times_{X_2}M_{23}} \ar[d]  \\
  \bbL_{X_2/X_{22}} \ar[r]^{\cong  \quad \quad}   \ar[u] &\bbL_{M_{13}\times_{X_{22}}X_2/M_{13}, } 
}\]
which makes
$\bbL_{M_{13}\times_{X_{22}}X_2/M_{13}}$ maximal isotropic quotient complex of the symmetric complex 
\begin{equation}\label{equ on lag 101}\bbL_{C_2\times_{X_{22}}M_{12}M_{23}\times_{X_{22}}X_2/M_{12}M_{23}\times_{X_{22}}X_2}\cong \bbL_{C_2/X_{22}}. \end{equation}
Here the $(-2)$-shifted symplectic structure on $C_2/X_{22}$ follows from 
the composition of $(-1)$-shifted Lagrangian $C_2\to C_{22}$ and $(-1)$-shifted Lagrangian fibration $C_{22}\to X_{22}$. 

We have a direct sum decomposition of symmetric complex: 
$$\bbL_{C_2\times_{X_{22}}(M_{13})\times_{X_{13}}M_{12}M_{23}/X_{1223}}\cong \bbL_{C_2/{X_{22}}}\oplus\bbL_{M_{13}\times_{X_{13}}M_{12}M_{23}\times_{X_{22}}X_2/X_{123}}. $$
Since $\bbL_{X_2/X_{22}}$ is maximal isotropic in $\bbL_{C_2/X_{22}}$, it is isotropic in the above direct sum. 
The isotropic reduction of $\bbL_{C_2\times_{X_{22}}(M_{13})\times_{X_{13}}M_{12}M_{23}/X_{1223}}$ by $\bbL_{X_2/X_{22}}$ is  $\bbL_{M_{13}\times_{X_{13}}M_{12}M_{23}\times_{X_{22}}X_2/X_{123}}$ (e.g.~\cite[App.~C]{Park2}). Moreover, 
under the above isotropic reduction, the maximal isotropic quotient complex $\bbL_{M_{13}/X_{1223}}$ defines a maximal isotropic quotient complex $\bbL_{M_{13}/X_{123}}$ in $\bbL_{M_{13}\times_{X_{13}}M_{12}M_{23}\times_{X_{22}}X_2/X_{123}}$. 

More precisely, with shorthand $E:=\bbL_{M_{13}\times_{X_{13}}M_{12}M_{23}\times_{X_{22}}X_2/X_{123}}$, we have a commutative diagram 
\[\xymatrix{
& \bbL_{M_{13}/X_{1223}}^\vee[2] \ar[d]^{ }  & \\
\bbL_{X_2/X_{22}}\oplus E^\vee[2]   \ar[r]^{ }   & \bbL_{C_2/X_{22}}^\vee[2]  \oplus E^\vee[2]  \ar[rr]^{ } \ar[d]^{\cong}  & & \bbL_{X_2/X_{22}}^\vee[2]    \\
  \bbL_{X_2/X_{22}} \ar[r]^{ } \ar[u]^{ }  \ar[rd]^{ }        & \bbL_{C_2/X_{22}}  \oplus E  \ar[rr]^{ }  \ar[d]^{ }      &    &   \bbL_{X_2/X_{22}}^\vee[2] \oplus E  \ar[u] \\
  &  \bbL_{M_{13}/X_{1223}}  \ar[rd]^{ }    & & E \ar[u]  \ar[dl] \\ 
  & &  \bbL_{M_{13}/X_{123}} & 
}\]
Here two horizontal lines are dual fiber sequences (up to a degree shift), the middle vertical fiber sequence follows from the Lagrangian in \eqref{equ on lag 100}, the right vertical fiber sequence 
realizes symmetric complex $E$ as the isotropic reduction.
The row $\bbL_{X_2/X_{22}}\to \bbL_{M_{13}/X_{1223}} \to \bbL_{M_{13}/X_{123}}$ is a fiber sequence, and 
it is straightforward to check $\bbL_{M_{13}/X_{123}}$ is a maximal isotropic quotient complex of $E$.  

Therefore, by Remark~\ref{rmk:lag_cl}, diagram \eqref{eqn:triang_obst} holds for \eqref{diag:s_g}. It is straightforward to check the compatibility of orientations. 
So,  Theorem \ref{thm on funct}  applies to \eqref{diag:s_g} and yields \eqref{equ on f123}.  

Applying smooth pullback $\left(\pi^{{123}}_{{13}}\right)^*$ to \eqref{equ on f123}, we obtain
\begin{align*}
\left(f_{12}\times_{X_2} f_{23}\right)_{}^!\circ \left(\pi^{{123}}_{{13}}\right)^*&= s^!\circ g_{}^!\circ \left(\pi^{{123}}_{{13}}\right)^* \\
&= s^!\circ \left(\bar{\pi}^{123}_{13}\right)^* \circ (f_{13})_{}^! \\
&=  (f_{13})_{}^!,
\end{align*}
where the second equality follows from the first property of Proposition \ref{pullback comm with base change} applied on diagram \eqref{diag:s exi}, the last equality is by considering the classical truncation of \eqref{equ on sect s}.
\end{proof}

\begin{corollary}\label{cor on funct of linearization}
Notations as in Theorem \ref{thm:functorial}. 
    Let $[\Delta_{X_2}]\in K^T_0(  X_2 \times X_2 ,\phi_{2}\boxplus (-\phi_{2}) )_{\loc} $ be the diagonal class \eqref{equ on diag class}, 
    and $i_{\Delta}: M_{12}\times_{X_2} M_{23}\to M_{12}\times_{} M_{23}$ be the inclusion. 
    Then
    \[i_{\Delta*}\circ (f_{13})_{}^!=(f_{12}\times f_{23})_{}^!\circ (-\,\boxtimes\, [\Delta_{X_2}]):K^T_0(X_1\times X_3,-\phi_{1}\boxplus \phi_{3})_{\loc}\to  K_0^T(M_{12}\times_{} M_{23})_{\loc},\]
    where $i_{\Delta*}$ is the proper pushforward.  
\end{corollary}
\begin{proof}
We have the following  diagram 
\begin{equation}
    \label{eqn:proof_lemma_func}
\xymatrix{
K^T_0(X_1\times X_3,(-\phi_{1})\boxplus \phi_{3})_{\loc}    \ar[d]_{\pi^{{123}*}_{{13}} }   \ar[rr]^{\quad \quad (f_{13})_{}^! }  &   & K_0^T(M_{12}\times_{\mathrm{Crit}(\phi_2)} M_{23})_{\loc}   \ar[d]^{\cong}  \\
K^T_0(X_1\times  X_2 \times X_3,(-\phi_{1})\boxplus 0 \boxplus \phi_{3})_{\loc}  \ar[rr]^{\quad \quad\quad \quad (f_{12}\times_{X_2}f_{23})_{}^!}  \ar[d]_{\Delta^{{123}}_{{1223}*}} & & K_0^T(M_{12}\times_{X_2} M_{23})_{\loc}  \ar[d]^{i_{\Delta*}} \\
K^T_0(X_1\times  X_2 \times X_2 \times X_3,(-\phi_{1})\boxplus \phi_{2}\boxplus (-\phi_{2}) \boxplus \phi_{3})_{\loc}  \ar[rr]^{\quad \quad\quad \quad \quad\quad \quad\quad\quad (f_{12}\times f_{23})_{}^!\quad\,\, } & & K_0^T(M_{12}\times_{} M_{23})_{\loc}. }
\end{equation} 
The upper square commutes by Theorem \ref{thm:functorial}, and the lower square commutes by base change property of critical pullbacks 
(Proposition~\ref{pullback comm with base change}).  
The outer square then commutes which is the statement.
\end{proof}

\subsection{Linearization of Lagrangian category}\label{sect on linearization}
Let $T$ be a torus, and $K_0^{T}(\mathrm{pt})_{\loc}$ be the field of fractions of $K_0^{T}(\mathrm{pt})$. Denote  
$\mathrm{Vect}_{K_0^{T}(\mathrm{pt})_{\loc}}$ to be the category of vector spaces over $K_0^{T}(\mathrm{pt})_{\loc}$.
\begin{definition}\label{def of linear func}
A \textit{linearization functor} is a Lax monoidal functor 
$$\mathscr{L}: \textbf{Lag}_{\mathrm{crit},T}^{\mathrm{DM}}\to \mathrm{Vect}_{K_0^{T}(\mathrm{pt})_{\loc}}.$$
Here the \textit{Lax structure} requires a map 
\[\mathscr{L}(X_1,\phi_{1})\otimes_{K_0^{T}(\mathrm{pt})_{\loc}} \mathscr{L}(X_2,\phi_{2})\to 
\mathscr{L}(X_1\times X_2,\phi_{1}\boxplus \phi_{2} )\] 
for objects $(X_1,\phi_{1})$ and $(X_2,\phi_{2})$. 
\textit{Functoriality} requires for each morphism 
$$\left(\textbf{\emph{M}}_{12}\xrightarrow{\textbf{\emph{f}}_{12}} \bCrit^{}(-\phi_1)\times \bCrit^{}(\phi_2)\right), $$ 
a linear map $\mathscr{L}(X_2,\phi_{2}) \to \mathscr{L}(X_1,\phi_{1})$ 
compatible with compositions.
\end{definition}
\begin{remark} 
    Recall \cite[Cor.~7.23, Prop.~7.25]{AB} that
    there is  a functor from the \textit{oriented cobordism} 2-category to the 2-category of Lagrangians in $(-1)$-shifted symplectic stacks:
    \[\calZ:\mathrm{Cob}^{or}\to \textbf{Lag}_{-1}^{(2)}. \]
 Here objects of $\mathrm{Cob}^{or}_{}$ are closed oriented 3-manifolds and 1-morphisms are 4-manifolds with boundaries and compatible orientations,  
 $\textbf{Lag}_{-1}^{(2)} $ is similar to $\textbf{Lag}_{\mathrm{crit}}^{(2)}$ but objects are allowed to be any $(-1)$-shifted symplectic stacks (instead of derived critical loci).
 
    Similarly, one reformulates \cite[Thm.~2.9]{Cal} as a functor of 2-categories
     \[\calZ^\bbC:\mathrm{Cob}^{ or_{\bbC}}_{}\to \textbf{Lag}_{-1}^{(2)}.\]
     Here  objects of $\mathrm{Cob}^{ or_\bbC}_{} $  are compact Calabi-Yau 3-folds, and 1-morphisms are 4-dimensional log Calabi-Yau pairs.
    There are examples of interesting subcategories of $\mathrm{Cob}^{or}_{}$ or $\mathrm{Cob}^{ or_{\bbC}}_{}$  to which the restrictions of the functors land in $\textbf{Lag}^{\mathrm{DM}}_{\mathrm{crit}}$. 
    
    Composing the 1-truncation of above with a linearization functor (Definition \ref{def of linear func}) then induces Lax monoidal functors from the 1-truncation of $\mathrm{Cob}^{or}_{}$ or $\mathrm{Cob}^{ or_\bbC}_{}$ to the category of vector spaces, which is a \textit{topological field theory} in the usual sense with target a linear category.
\end{remark}
We formulate a linearization of the Lagrangian category in terms of critical $K$-theory.  
\begin{theorem}\label{thm:functor}
Under Setting \ref{set of perf pair}, the assignment 
$$(X,\phi) \mapsto K^T_0(X,\phi)_{\loc}:=K^T_0(X,\phi)\otimes_{K_0^{T}(\mathrm{pt})_{}}K_0^{T}(\mathrm{pt})_{\loc} $$
defines a Lax monoidal linearization functor $$K_{\mathrm{crit}}: \textbf{\emph{Lag}}_{\mathrm{crit},T}^{\mathrm{DM}}\to \mathrm{Vect}_{K_0^{T}(\mathrm{pt})_{\loc}}.$$
\end{theorem}
The rest of this section is devoted to its proof. 
Firstly, the \textit{Lax monoidal structure} is given by the \textit{exterior tensor product} \eqref{equ on ten prod on k}:
\begin{equation}\label{equ on ten prod on k2}K^T_0(X_1,\phi_{1})_{\loc}\otimes K^T_0(X_2,\phi_{2})_{\loc}\to K^T_0(X_1\times X_2,\phi_{1}\boxplus \phi_{2})_{\loc}.\end{equation} 
For a morphism in the category $\textbf{Lag}_{\mathrm{crit},T}^{\mathrm{DM}}$: 
$$\left(\textbf{\emph{M}}_{12}\xrightarrow{\textbf{\emph{f}}_{12}} \bCrit^{}(-\phi_1)\times \bCrit^{}(\phi_2)\right), $$ 
we want to define a morphism between critical $K$-groups $K^T_0(X_i,\phi_{i})_{\loc}$ ($i=1,2$).

Consider the following linear maps:
\begin{equation}K^T_0(X_1\times X_2,-\phi_{1}\boxplus \phi_{2})_{\loc} \xrightarrow{ (f_{12})_{}^!} K^T_0(M_{12})_{\loc} 
\xrightarrow{p_{M_{12}*}} K^T_0(\pt)_{\loc}, \nonumber \end{equation}
where $(f_{12})_{}^!$ is the critical pullback \eqref{equ on crit pb der} where we drop the subindex $\pi_X$ in the notation as $\pi_X$ is the identity in this case, $p_{M_{12}*}$ is the (equivariant) pushforward to a point.  

Composing with the exterior tensor product \eqref{equ on ten prod on k2}, we obtain a map 
\begin{equation}\label{equ on K0X12} K^T_0(X_1,-\phi_{1})_{\loc}\otimes K^T_0(X_2,\phi_{2})_{\loc}\to K_0^{T}(\mathrm{pt})_{\loc}. \end{equation}
To turn this map into a map between $K^T_0(X_i,\phi_{i})_{\loc}$, we need a perfect pairing on $K$-groups. 

\subsection{Pairing on critical $K$-theories} \label{subsec:pairing_crit_K}
\begin{lemma}\label{lem on pairing}
Let $X$ be a quotient stack of a smooth quasi-projective scheme  by a linear algebraic group and $\phi$ be a flat regular function\,\footnote{We assume the critical locus $\Crit(\phi)\hookrightarrow Z(\phi)$  embeds into 
 the zero locus $Z(\phi)$.} on $X$. 
Assume there is a torus $T$-action on $X$ which preserves $\phi$  
such that the critical locus $\Crit(\phi)$ is $T$-equivariantly proper.
Then we have a natural pairing
\begin{equation}\label{equ trace pairing}
K^T_0(X,\phi)_{\loc}\otimes K^T_0(X,-\phi)_{\loc}\to K^T_0(\Crit(\phi))_{\loc} \to K_0^{T}(\mathrm{pt})_{\loc},
\end{equation}
where the first map is induced by \eqref{equ on ten prod on k} and the last map is the equivariant pushforward to a point.
\end{lemma}
\begin{proof}
By \cite[Cor.~3.18]{PV3}, we know any element $\eE$ in $\mathrm{MF} (X,\phi)$ is supported 
on the  singular locus $\mathrm{Sing}(Z(\phi))=Z(\phi)\cap \Crit(\phi)=\Crit(\phi)$ of $Z(\phi)$,
where the second equation follows from the assumption on the function $\phi$. 
Therefore we have 
$$K^T_0(X,\phi)_{\loc}=K^T_0(X,\Crit(\phi),\phi)_{\loc}. $$
By composing \eqref{equ on ten prod on k} with the pullback along diagonal (and taking care of support), we have  
$$K^T_0(X,\Crit(\phi),\phi)_{\loc}\otimes K^T_0(X,\Crit(\phi),-\phi)_{\loc}\to K^T_0(X,\Crit(\phi),0)_{\loc}\cong K^T_0(\Crit(\phi))_{\loc}, $$
where the isomorphism is by  d\'evissage.
This defines the required map.
\end{proof}
\begin{setting}\label{set of perf pair}
Let $(X,\phi)$ be as in Lemma \ref{lem on pairing} and assume further that $X$ is a DM stack.

Consider the following conditions on $K^T_0(X,\phi)_{\loc}$: 
\begin{itemize}
\item $K^T_0(X,\phi)_{\loc}$ and $K^T_0(X\times X,-\phi\boxplus \phi)_{\loc}$ are finite dimensional over $K:=K_0^{T}(\mathrm{pt})_{\loc}$,  
$$\dim_{K}K^T_0(X\times X,-\phi\boxplus \phi)_{\loc}=\left(\dim_{K}K^T_0(X,\phi)_{\loc}\right)^2.$$
\item Under the canonical isomorphism 
\begin{equation}\label{equ on cano iso on minus}K^T_0(X,-\phi)_{\loc}\cong K^T_0(X,\phi)_{\loc}, \end{equation}
the composition map \eqref{equ trace pairing} is a perfect pairing,~i.e.~the induced map 
\begin{equation}\label{equ on iso of dual}K^T_0(X,\phi)_{\loc}\xrightarrow{\cong} K^T_0(X,\phi)_{\loc}^\vee \end{equation} 
is an isomorphism.
\end{itemize}
\end{setting}
We give an equivalent description of the hypotheses in above.  
Under the pullback along diagonal
$$X \xrightarrow{\Delta}  X\times X, $$
the sum function $\phi \boxplus (-\phi)$ goes to zero, therefore there is a proper pushforward
\begin{equation}\label{equ on diag}\Delta_*: K^T_0(X,0)\to  K^T_0(X\times X,\phi\boxplus (-\phi))_{ }. \end{equation}
The \textit{diagonal class} is the following image 
\begin{equation}\label{equ on diag class}[\Delta]:=\Delta_*[\oO_X]. \end{equation}
\begin{prop}\label{prop on equiv diag}
The following conditions are equivalent: 
\begin{itemize}

\item Conditions in Setting \ref{set of perf pair}. 

\item The diagonal class $[\Delta]$ lies in the image of the exterior product \eqref{equ on ten prod on k}:
\begin{equation}\label{equ on kunn prod}K^T_0(X,\phi)_{\loc}\otimes K^T_0(X,-\phi)_{\loc}\to K^T_0(X\times X,\phi\boxplus (-\phi))_{\loc}. \end{equation}
\end{itemize}
Moreover the above conditions imply that for any pair $(Y,\phi_Y)$ of quotient stack $Y$ and regular function $\phi_Y$ with $T$-action, the exterior product  
\begin{equation}\label{equ on kunn prod2}K^T_0(X,\phi)_{\loc}\otimes K^T_0(Y,\phi_Y)_{\loc}\xrightarrow{\cong} K^T_0(X\times Y,\phi\boxplus \phi_Y)_{\loc} \end{equation}
is an isomorphism. 
\end{prop}
\begin{proof}
We follow the proof of \cite[Thm.~5.6.1]{CG}.
\end{proof}

\begin{example}
[Dimensional reduction]
\label{ex on nak}
Let $E$ be a vector bundle on a smooth scheme $U$ and $s\in \Gamma(U,E^\vee)$ be a cosection of $E$, whose zero locus $Z(s)$ is a smooth scheme.
There is an induced regular function $\phi=\langle s,-\rangle: E\to \C$ and  a Cartesian square 
$$
\xymatrix{
E|_{Z(s)} \ar[r]^{\,\,\iota} \ar[d]_{\pi}     &  E \ar[d]^{ } \\
Z(s) \ar[r]^{} & U. }
$$
By dimensional reduction \cite[Thm.~2.4 \& Cor.~3.13]{Toda:localsurfaceZ/2}, we have isomorphisms 
$$\iota_*\pi^*: K_0(Z(s))\xrightarrow{\cong}  K_0(E,\phi), \quad (\iota\times\iota)_*(\pi\times\pi)^*: K_0(Z(s)\times Z(s))\xrightarrow{\cong}  K_0(E\times E,\phi\boxplus -\phi). $$
\textbf{Claim}: The diagonal class \eqref{equ on diag class} of $K_0(E\times E,\phi\,\boxplus -\phi)$  coincides with the diagonal class of $K_0(Z(s)\times Z(s))$ under the 
above isomorphism. 

When there is a torus action, the above also extends to the torus equivariant setting. 
Note that \textit{Nakajima quiver varieties} can be realized as $Z(s)$ above and have decomposition of diagonals \cite[\S 7.3]{Nak}. Therefore, 
 corresponding quivers with potentials $(E,\phi)$ also have decomposition of diagonals.  
\end{example}
\begin{proof}
In fact, we have a commutative diagram 
$$
\xymatrix{
K_0(Z(s)) \ar[r]^{\pi^*}_{\cong\,\,\,\,} \ar[d]^{\Delta_*}     &  K_0\left(E|_{Z(s)}\right) \ar[d]^{\Delta_*}\ar[r]^{\iota_*}_{ } & K_0(E) \ar[d]^{\Delta_*}  \\
K_0(Z(s)\times Z(s)) \ar[r]^{(\pi\times\pi)^*\,\,}_{\cong\,\,\,\,\,\,} & K_0\left(E|_{Z(s)}\times E|_{Z(s)}\right)  \ar[r]^{(\iota\times\iota)_*}_{\cong\,\,\,} &  K_0(E\times E,\phi\boxplus -\phi).  }
$$
The difference of two diagonals is 
$$\Delta_*\left(\iota_*\oO_{E|_{Z(s)}}-\oO_{E}\right), $$
whose support is away from the critical locus $\Crit(\phi\boxplus -\phi)\cong Z(s)\times Z(s)$. Hence the claim holds.
\end{proof}

\begin{example}
[Hilbert schemes of points on $\C^3$]
\label{ex on hilb3}
Let $T=\{t_1 t_2 t_3=1\} \subset (\mathbb{C}^{\ast})^3$ and consider the non-commutative Hilbert scheme of $n$-points on $\C^3$:  
\begin{equation}\mathrm{NHilb}^n(\C^3)=\left(\Hom(\mathbb{C}^n,\mathbb{C}^n)^{\times 3}\times \mathbb{C}^n\right)/\!\!/\GL(n),  \nonumber \end{equation}
with potential function
$$\phi(b_1,b_2,b_3,v)=\tr (b_1[b_2,b_3]), $$  
whose critical locus satisfies 
\begin{equation} \Crit(\phi)\cong \Hilb^n(\C^3). \nonumber \end{equation}
Then the equivalent conditions in Proposition \ref{prop on equiv diag} hold for $(\mathrm{NHilb}^n(\C^3),\phi)$.
\end{example}
\begin{proof}
By~Remark~\ref{rmk:general}, we have an isomorphism 
\begin{align}\label{isom:KHilb}
    K_0^{T}(\mathrm{MF}^{\rm{gr}}(\mathrm{NHilb}^n(\C^3), \phi)) \stackrel{\cong}{\to} 
    K_0^{T}(\mathrm{NHilb}^n(\C^3), \phi),
\end{align}
where the grading on LHS is with respect to a weight two 
action on the last factor of $\Hom(\mathbb{C}^n, \mathbb{C}^n)^{\times 3}$. 
For each $\mu \in \mathbb{R}$, 
we have a semiorthogonal decomposition (SOD) by~\cite[Thm.~3.2]{PT}:
\begin{align}\label{sod:PT}
    \mathrm{MF}_T^{\rm{gr}}(\mathrm{NHilb}^n(\C^3), \phi)=\left\langle\, \boxtimes_{i=1}^k \mathbb{T}_T(d_i)_{v_i} \right\rangle,
\end{align}
where the RHS runs over all partitions $n=d_1+\cdots+d_k$ such that 
\begin{equation}\label{equ on hnf}\mu\leq v_1/d_1<\cdots<v_k/d_k <\mu+1. \end{equation} 
The category $\mathbb{T}_T(d)_v$
is a subcategory in $\D_{T}^b(\mathcal{C}(d))$ called \textit{quasi-BPS category} in~\cite{PT}, 
where $\mathcal{C}(d)$ is the derived commuting stack. 
The category $\boxtimes_{i=1}^k \mathbb{T}_T(d_i)_{v_i}$ is the subcategory 
of $\D_T^b(\times_{i=1}^k \mathcal{C}(d_i))$ split generated by $\boxtimes_{i=1}^k A_i$ 
for $A_i \in \mathbb{T}_T(d_i)_{v_i}$. 
It is shown in the proof of~\cite[Cor.~4.14]{PT} that we have the isomorphism 
\begin{align}\label{kunneth:T}
    \bigotimes_{i=1}^k K_0^T(\mathbb{T}(d_i)_{v_i})_{\loc} \stackrel{\cong}{\to} 
    K_0^T(\boxtimes_{i=1}^k \mathbb{T}(d_i)_{v_i})_{\loc}. 
\end{align}
It follows that we have the isomorphism 
\begin{align*}
    K_0^T(\mathrm{MF}^{\rm{gr}}(\mathrm{NHilb}^n(\C^3), \phi))_{\loc}\cong
    \bigoplus \bigotimes_{i=1}^k K_0^T(\mathbb{T}(d_i)_{v_i})_{\loc}.
\end{align*}
Here the direct sum runs over all $n=d_1+\cdots+d_k$ and $v_i \in \mathbb{Z}$ as in (\ref{sod:PT}), \eqref{equ on hnf}.

We also have a semiorthogonal decomposition:
\begin{align*}
    \mathrm{MF}^{\rm{gr}}_T\left(\mathrm{NHilb}^n(\mathbb{C}^3) \times \mathrm{NHilb}^n(\mathbb{C}^3), \phi \boxplus -\phi\right)=\left\langle \left(\boxtimes_{i=1}^k \mathbb{T}_T(d_i)_{v_i}\right)\boxtimes \left(\boxtimes_{i=1}^{k'}\mathbb{T}_T(d_i')_{v_i'}\right)  \right\rangle,
\end{align*}
where the right hand side is after all partitions $n=d_1+\cdots+d_k=d_1'+\cdots+d_{k'}'$
such that $$0\leq v_1/d_1<\cdots<v_k/d_k<1, \quad 1\leq v_1'/d_1'<\cdots<v_{k'}'/d_{k'}' <2.$$ 
The above SOD can be proved to exist either using an argument 
similar to~\cite[Thm.~3.2]{PT}, or proving a similar SOD with zero potential case from the fact that $A\boxtimes B$ for 
$A, B \in \D_T^b(\mathrm{NHilb}^n(\C^3))$ split generate $\D_T^b(\mathrm{NHilb}^n(\mathbb{C}^3)\times \mathrm{NHilb}^n(\mathbb{C}^3))$, together with~\cite[Eqn.~(3.62)]{PT}, and then 
applying~\cite[Prop.~2.5]{PT}. 
Combing with~\cite[Cor.~3.13]{Toda:localsurfaceZ/2}, we can also decompose $K_0^T(\mathrm{NHilb}^n(\C^3)^{\times 2}, \phi \boxplus -\phi)$ into the direct sum of products of $K_0^T(\mathbb{T}(d_i)_{v_i})$ and $K_0^T(\mathbb{T}(d_i')_{v_i'})$. It follows that we have an isomorphism 
\begin{equation*}
    K_0^{T}(\mathrm{NHilb}^n(\C^3), \phi)_{\loc} \otimes  K_0^{T}(\mathrm{NHilb}^n(\C^3), -\phi)_{\loc}
    \stackrel{\cong}{\to}  K_0^{T}(\mathrm{NHilb}^n(\C^3)^{\times 2}, \phi\boxplus -\phi)_{\loc}. \qedhere
\end{equation*}
\end{proof}

\begin{remark}\label{rmk:general}
In (\ref{isom:KHilb}), we need a slightly generalized version of~\cite[Cor.~3.13]{Toda:localsurfaceZ/2} as follows. 

Let $Y$ be a smooth $\mathbb{C}$-scheme with an action of 
a reductive group $G$, $V \to Y$ be a $G$-equivariant 
vector bundle with a $G$-invariant section $s$. 
Consider $G$-invariant function: 
$$w \colon V^{\vee} \to \mathbb{C}, \quad (y, v)\mapsto \langle s(y), v\rangle,$$
where $y\in Y$ and $v \in V^{\vee}|_{y}$. 
Let $Z \subset \mathrm{Crit}(w)$ be a 
$(G\times \mathbb{C}^{\ast})$-invariant closed subset of the critical locus, where 
$\mathbb{C}^{\ast}$ acts on the fibers of $V^{\vee} \to Y$
by weight two. 

\textbf{Claim}: the forgetful functor induces an isomorphism 
\begin{align}\label{isom:Z2}
\mathrm{forg} \colon K_0(\mathrm{MF}^{\rm{gr}}([(V^{\vee}\setminus Z)/G] , w))
\stackrel{\cong}{\to} K_0([(V^{\vee}\setminus Z)/G], w).
\end{align}
Note that the isomorphism (\ref{isom:KHilb}) is a special case of (\ref{isom:Z2}). 
And the isomorphism (\ref{isom:Z2}) is proved 
in~\cite[Cor.3.13]{Toda:localsurfaceZ/2} when $Z=\emptyset$ and $G$ is trivial. 
The same argument of~\cite[Cor.~3.13]{Toda:localsurfaceZ/2} works verbatim 
when $Z=\emptyset$ and $G$ is non-trivial, so 
we explain the case when $Z\neq \emptyset$. 

Let $\mathcal{V}^\vee=[V^\vee/G]$. 
We have functors 
\begin{align*}
    \mathrm{MF}^{\rm{gr}}(\mathcal{V}^{\vee}, w)
    \stackrel{p^{\ast}}{\to} \mathrm{MF}^{\rm{gr}}(\mathcal{V}^{\vee}\times \mathbb{C}^{\ast}, p^{\ast}w)
    \stackrel{i^{\ast}}{\to} \mathrm{MF}^{\rm{gr}}(\mathcal{V}^{\vee}, w),
\end{align*}
where the first one is given by the pull-back via 
$p \colon \mathcal{V}^{\vee} \times \mathbb{C}^{\ast} \to \mathcal{V}^{\vee}$
and the second one is the restriction by 
$1 \in \mathbb{C}^{\ast}$. 
By the same argument of~\cite[Lem.~3.8]{Toda:localsurfaceZ/2}, we have the equivalence 
\begin{align*}
\mathrm{MF}(\mathcal{V}^{\vee}, w) \stackrel{\cong}{\to}
\mathrm{MF}^{\rm{gr}}(\mathcal{V}^{\vee}\times \mathbb{C}^{\ast}, p^{\ast}w)
\end{align*}
and the composition of the following induced map 
\begin{align}\label{induced:K}
K_0(\mathrm{MF}^{\rm{gr}}(\mathcal{V}^{\vee}, w))\stackrel{p^{\ast}}{\to} K_0(\mathcal{V}^{\vee}, w) 
\stackrel{i^{\ast}}{\to} K_0(\mathrm{MF}^{\rm{gr}}(\mathcal{V}^{\vee}, w))
    \end{align}
is the identity. 
The first map is identified with (\ref{isom:Z2})
for $Z=\emptyset$, 
which is an isomorphism by~\cite[Cor.~3.13]{Toda:localsurfaceZ/2}, hence 
the second map is also an isomorphism. 

The maps in \eqref{induced:K} fit into a commutative diagram 
\begin{align}\label{diag on kkk}
\xymatrix{
K_0\left(\mathrm{MF}^{\rm{gr}}(\mathcal{V}^{\vee}, w)_{[Z/G]}\right) \ar[r]^{ }  \ar[d]^{} &  
K_0\left(\mathrm{MF}(\mathcal{V}^{\vee}, w)_{[Z/G]}\right) \ar[r]^{ } \ar[d]^{} &
K_0\left(\mathrm{MF}^{\rm{gr}}(\mathcal{V}^{\vee}, w)_{[Z/G]}\right) \ar[d]^{} \\
K_0\left(\mathrm{MF}^{\rm{gr}}(\mathcal{V}^{\vee}, w)_{ }\right) \ar[r]^{p^{\ast}}_{\cong \,\,\, } \ar[d]^{} &  
K_0\left(\mathrm{MF}(\mathcal{V}^{\vee}, w)\right) \ar[r]^{i^{\ast}}_{\cong \,\,\, } \ar[d]^{}  &
K_0\left(\mathrm{MF}^{\rm{gr}}(\mathcal{V}^{\vee}, w)_{ }\right) \ar[d]^{} \\
K_0\left(\mathrm{MF}^{\rm{gr}}([(V^{\vee}\setminus Z)/G], w)_{ }\right) \ar[r]^{ }   &  
K_0\left(\mathrm{MF}([(V^{\vee}\setminus Z)/G], w)\right) \ar[r]^{ }  &
K_0\left(\mathrm{MF}^{\rm{gr}}([(V^{\vee}\setminus Z)/G], w)_{ }\right),  
}
 \end{align} 
where vertical maps are exact sequences by \cite[Eqn.~(2.7)]{Toda:localsurfaceZ/2}. 
Let $K_0'$ denote the image of the natural map 
    $K_0\left((-)_{[Z/G]}\right) \to K_0(-)$. The above diagram induces injective maps
\begin{align}\label{induced:K2}
K_0'\left(\mathrm{MF}^{\rm{gr}}(\mathcal{V}^{\vee}, w)_{[Z/G]}\right)
\stackrel{p^{\ast}}{\to} 
K_0'\left(\mathrm{MF}(\mathcal{V}^{\vee}, w)_{[Z/G]}\right)
\stackrel{i^{\ast}}{\to}
K_0'\left(\mathrm{MF}^{\rm{gr}}(\mathcal{V}^{\vee}, w)_{[Z/G]}\right).
    \end{align}
Since $p^{\ast}$ in \eqref{induced:K} is an isomorphism and $i^{\ast}$ is its 
    inverse, it follows that the maps in (\ref{induced:K2}) are also surjective, hence are isomorphisms. 
Therefore by taking quotients of vertical arrows  in \eqref{diag on kkk}, and using \cite[Eqn.~(2.8)]{Toda:localsurfaceZ/2}, we obtain (\ref{isom:Z2})
for $Z\neq \emptyset$.

\end{remark}

\subsection{Proof of Theorem~\ref{thm:functor}}

Using the pairing introduced above, we can return back to the end of \S \ref{sect on linearization} and finish the definition of linearization functor.
\begin{definition}
Assume the decomposition of diagonal \eqref{equ on kunn prod}. 
Under the perfect pairing  \eqref{equ on kunn prod2}, 
the map \eqref{equ on K0X12}: 
$$K^T_0(X_1\times X_2,-\phi_{1}\boxplus \phi_{2})_{\loc}\to K^T_0(\mathrm{pt})_{\loc}$$
induces a map 
\begin{equation}\label{equ on K0X12dual}f^!_{1\to 2}: K^T_0(X_2,\phi_{2})_{\loc}\to K^T_0(X_1,\phi_{1})_{\loc}. \end{equation}
Similarly we define $f^!_{2\to 3}$, $f^!_{1\to 3}$. 
\end{definition} 
The above map associates morphisms in $\textbf{{Lag}}_{\mathrm{crit},T}^{\mathrm{DM}}$ to morphisms in $\mathrm{Vect}_{K_0^{T}(\mathrm{pt})_{\loc}}$. We prove the functoriality as follows. 
\begin{corollary}
The map \eqref{equ on K0X12dual} is functorial,~i.e.
\begin{equation}\label{equ on K0X12dual compse}
f^!_{1\to 3}= f^!_{1\to 2}\circ f^!_{2\to 3}: K^T_0(X_3,\phi_{3})_{\loc}\to K^T_0(X_1,\phi_{1})_{\loc}. \end{equation}
\end{corollary}
\begin{proof}
This follows from Corollary~\ref{cor on funct of linearization}. In fact, 
under isomorphisms \eqref{equ on iso of dual} and \eqref{equ on kunn prod2}, the upper-right composition in diagram \eqref{eqn:proof_lemma_func} composed with 
equivariant pushforward to a point: 
$$ K^T_0(X_1\times X_3,-\phi_{1}\boxplus \phi_{3})_{\loc} \xrightarrow{(f_{13})_{}^!} K_0^T(M_{12}\times_{\mathrm{Crit}(\phi_2)} M_{23})_{\loc} \to K_0^{T}(\mathrm{pt})_{\loc}$$
 becomes 
$$f^!_{1\to 3}: K^T_0(X_3,\phi_{3})_{\loc}\to K^T_0(X_1,\phi_{1})_{\loc}. $$
While the maps
$$K^T_0(X_1\times  X_2 \times X_2 \times X_3,-\phi_{1}\boxplus \phi_{2}\boxplus (-\phi_{2}) \boxplus \phi_{3})_{\loc} \xrightarrow{(f_{12}\times f_{23})_{}^!}
 K_0^T(M_{12}\times_{} M_{23})_{\loc} \to K^T_0(\mathrm{pt})_{\loc}, $$
$$ 
K^T_0(X_1\times X_3,-\phi_{1}\boxplus \phi_{3})_{\loc}\xrightarrow{(f_{12}\times f_{23})_{}^!\circ \Delta^{{123}}_{{1223}*}\circ \pi^{{123}*}_{{13}}} K_0^T(M_{12}\times_{} M_{23})_{\loc} \to K^T_0(\mathrm{pt})_{\loc}$$ 
become
\begin{equation}\label{equ on wrt as du1} K^T_0(X_2,\phi_{2})_{\loc}^\vee \otimes K^T_0(X_2,\phi_{2})_{\loc}\otimes 
K^T_0(X_3,\phi_{3})_{\loc}\to K^T_0(X_1,\phi_{1})_{\loc}, 
\end{equation}
\begin{equation}\label{equ on wrt as du2} K^T_0(X_3,\phi_{3})_{\loc}\to K^T_0(X_1,\phi_{1})_{\loc} \end{equation}
respectively. 
Plugging $[\Delta_{X_2}]\boxtimes (-)$ into \eqref{equ on wrt as du1} gives the map \eqref{equ on wrt as du2}, which by definition 
is the composition $f^!_{1\to 2}\circ f^!_{2\to 3}$.
\end{proof}
This finishes the proof of Theorem~\ref{thm:functor}. 

\section{Examples and applications}\label{sec:appl}
\subsection{Application I: Quantum critical $K$-theory} 
In this section, we define $K$-theoretic \textit{Gromov-Witten/quasimap type invariants} for $(W/\!\!/ G,\phi_{W/\!\!/ G})$ and prove a gluing formula for these invariants. 

\subsubsection{$K$-theoretic invariants}

Notations as in Example \ref{ex of papers}\,(1), let $W/\!\!/ G$ be a smooth GIT quotient of a complex vector space $W$ by a linear algebraic group $G$, and 
$$\phi_{W/\!\!/ G}: W/\!\!/ G\to \C$$ be a flat regular function, descent from a regular function $\phi_W: W\to \C$. 

Let $\widetilde{T}$ be a torus acting on $W/\!\!/ G$ such that $\phi_{W/\!\!/ G}$ is equivariant with respect to a character 
$$\chi:\widetilde{T}\to \C^* $$ and  $T:=\ker(\chi)$ be the subtorus, such that 
the torus fixed locus $\mathrm{Crit}(\phi_{W/\!\!/ G})^T$ is proper.

Let $\overline{\mathcal{M}}_{g,n}$ be the moduli stack of genus $g$, $n$-pointed stable curves (with $2g-2+n>0$) and consider the product
$$X:=(W/\!\!/ G)^n\times \overline{\mathcal{M}}_{g,n}.$$ 
Define the sum function    
$$\phi_X^{\boxplus n}:=\phi_{W/\!\!/ G}^{\boxplus n}\boxplus 0: (W/\!\!/ G)^n\times \overline{\mathcal{M}}_{g,n} \to \C, \quad (x_1,x_2,\cdots,x_n,C)\mapsto \sum_{i=1}^n \phi_{W/\!\!/ G}(x_i),$$
which does not depend on the factor $\overline{\mathcal{M}}_{g,n}$.

Proposition \ref{prop on ex sat set} ensures conditions in Definition \ref{defi of crit pb}. Combining with  Remark \ref{rmk on T-equiv}, we get a group homomorphism 
$$f_{\pi_X}^!:K_0^T\left(X, \phi_X^{\boxplus n}\right) \to K_0^T\left(\mathrm{QM}_{g,n}^{R_{\chi}=\omega_{\mathrm{log}}}(\Crit(\phi_{W/\!\!/ G}),\beta)
\right)_{}. $$
In below we also write $\mathrm{QM}_{g,n}^{R_{\chi}=\omega_{\mathrm{log}}}(\Crit(\phi_{W/\!\!/ G}),\beta)$ as $\mathrm{QM}_{g,n,\beta}$ for short. 
Applying \eqref{equ on k sqr pb} to above setting gives a $K$-theoretic analogue of \cite[Thm.~1.1]{CZ} (see \cite[Rmk.~4.14]{CZ}): 
\[\sqrt{f_Z^!}\circ\pi_X^*:K_0^T\left(Z\left(\phi_X^{\boxplus n}\right)\right)   \to K_0^T\left(\mathrm{QM}_{g,n,\beta}\right)_{ }.\]
By Proposition \ref{prop:compare_OT}, the following  diagram commutes
$$
\xymatrix{
K_0^T\left(Z\left(\phi_X^{\boxplus n}\right)\right)    \ar@{->>}[d]^{ }   \ar[rr]^{\sqrt{f_Z^!}\circ\pi_X^*      }   & & K_0^T\left(\QM_{g,n,\beta}\right)_{ }  \ar@{=}[d] \\
K_0^T\left(X,\phi_X^{\boxplus n}\right) \ar[rr]^{f_{\pi_X}^!  }  & &  K_0^T\left(\QM_{g,n,\beta}\right)_{ }.
}
$$
Taking exterior product in the domain and equivariant pushforward to a point in the target, we obtain the following
commutative diagram 
\begin{equation}
    \label{eqn:GW_diag}\xymatrix{
K_0(\overline{\mathcal{M}}_{g,n})\otimes K_0^T\left(Z\left(\phi_{W/\!\!/ G}^{\boxplus n}\right)\right)_{ }    \ar@{->>}[d]^{ }   \ar[r]^{ }    & K_0^T\left(\pt\right)_{\loc} \ar@{=}[d] \\
K_0(\overline{\mathcal{M}}_{g,n})\otimes K_0^T\left((W/\!\!/ G)^n,\phi_{W/\!\!/ G}^{\boxplus n}\right)_{ } 
\ar[r]^{\quad \quad \quad\quad\quad \quad  \Phi^{\Crit,K}_{g,n,\beta}}   &  K_0^T\left(\pt\right)_{\loc},
}
\end{equation}
where the upper row is the $K$-theoretical analogue of the map constructed in \cite[Def.~5.5]{CZ}. 
\begin{definition}
When $2g-2+n>0$, we define the $K$-theoretic \textit{quasimap invariants} of $(W/\!\!/ G,\phi_{W/\!\!/ G})$ to be the map in the bottom row:
 \begin{equation}\label{equ on phign}\Phi^{\Crit,K}_{g,n,\beta}: K_0(\overline{\mathcal{M}}_{g,n})\otimes K_0^T\left((W/\!\!/ G)^n,\phi_{W/\!\!/ G}^{\boxplus n}\right)_{ } \to K_0^T\left(\pt\right)_{\loc}. \end{equation}
 \end{definition}
\begin{remark}
For any decomposition $n=n_1+n_2$, by using tensor product, 
we can also evaluate $\Phi^{\Crit,K}_{g,n,\beta}$ on classes from
$K_0(\overline{\mathcal{M}}_{g,n})\otimes K_0^T\left((W/\!\!/ G)^{n_1},\phi_{W/\!\!/ G}^{\boxplus n_1}\right)\otimes K_0^T\left((W/\!\!/ G)^{n_2},\phi_{W/\!\!/ G}^{\boxplus n_2}\right)$. 
\end{remark}
When $(g,n)=(0,2)$, which is not in the stable range, there is no stablization map. 
We instead consider the map 
 $$\mathrm{QM}_{0,2}^{R_{\chi}=\omega_{\mathrm{log}}}(\Crit(\phi_{W/\!\!/ G}),\beta)\to \fBun_{H_R,0,2}^{R_{\chi}=\omega_{\mathrm{log}}}\times_{[\pt/G]^2}(W/\!\!/ G)^2\to (W/\!\!/ G)^2.$$ 
 This still satisfies Setting \ref{setting of lag} and we have
 \begin{equation}\label{equ on phi02}\Phi^{\Crit,K}_{0,2,\beta}: K_0^T\left((W/\!\!/ G)^2,\phi_{W/\!\!/ G}^{\boxplus 2}\right)_{ } \to K_0^T\left(\pt\right)_{\loc}. \end{equation}
 In what follows, we show that $\Phi^{\Crit,K}_{g,n,\beta}$ satisfies a gluing formula. 

\subsubsection{An automorphism}
There is an identification between $K_0^T(W/\!\!/ G,\phi_{W/\!\!/ G})$ and $K_0^T(W/\!\!/ G,-\phi_{W/\!\!/ G})$ (e.g.,~\cite[\S 5]{FK} and \cite[Rmk.~4.18]{CZ}), different from the one in
\eqref{equ on cano iso on minus}, which we now recall. 

Consider a $(G\times T)$-equivariant \textit{automorphism}: 
\begin{equation}\label{ord 2 autom}\sigma:W\to W, \quad \mathrm{such}\,\, \mathrm{that}\,\, \sigma^{*}\phi_W=-\phi_W. \end{equation}
It descents to an automorphism
\begin{equation}\label{ord 2 autom2}\sigma: W/\!\!/ G\to W/\!\!/ G, \quad \mathrm{such}\,\, \mathrm{that}\,\,\sigma^{*}\phi_{W/\!\!/ G}=-\phi_{W/\!\!/ G}.\end{equation}

\begin{example}
Recall \cite[Setting~2.1, Def.~2.5]{CZ}, starting from a nontrivial group homomorphism 
$$R_\chi: \bbC^*\stackrel{R}{\to} \widetilde{T} \stackrel{\chi}{\to} \bbC^*, $$
we can take $\sigma\in R_\chi^{-1}(-1)$ as a preimage of $-1$, which acts on $W$ through map $R$ and $\widetilde{T}$-action. 
Then $\sigma$ commutes with the action of $G\times T$ and satisfies $\sigma^{*}\phi_W=-\phi_W$ (ref.~\cite[Rmk.~4.18]{CZ}). 
\end{example}
The automorphism \eqref{ord 2 autom2} induces an   
automorphism on the moduli stack (\cite[Def.~4.20]{CZ}):
\begin{equation}\label{ord 2 autom3}\sigma: \QM_{g,n,\beta}\xrightarrow{\cong} \QM_{g,n,\beta}, \end{equation}
which commutes with \eqref{ord 2 autom2} under the evaluation map. In particular, for any 
$g_1,g_2,n_1,n_2,\beta_0,\ldots,\beta_k,\beta_{\infty}$, there is a commutative diagram 
\begin{equation}\label{eqn:bdy_divisor_curve2}
{\footnotesize{
\xymatrix{
\QM_{g_1,n_1+1,\beta_0}\times \prod_{i=1}^k\QM_{0,2,\beta_i}\times  \QM_{g_2,n_2+1,\beta_\infty} \ar[d]\ar[rrr]_{\cong}^{\id\times \prod_{i=1}^k\sigma^{\frac{1-(-1)^{i}}{2}}\times \sigma^{\frac{1+(-1)^{k}}{2}}} & & &
\QM_{g_1,n_1+1,\beta_0}\times \prod_{i=1}^k\QM_{0,2,\beta_i} \times \QM_{g_2,n_2+1,\beta_\infty} \ar[d] \\
(W/\!\!/ G)^{n_1+1} \times ((W/\!\!/ G)^2)^{k}\times (W/\!\!/ G)^{n_2+1} \ar[rrr]_{\cong}^{\id\times \prod_{i=1}^k\sigma^{\frac{1-(-1)^{i}}{2}}\times \sigma^{\frac{1+(-1)^{k}}{2}} } & &  &
 (W/\!\!/ G)^{n_1+1}\times  ((W/\!\!/ G)^2)^{k} \times (W/\!\!/ G)^{n_2+1},
} }}
\end{equation}
where $\sigma^0=\id$ is the identity map and $\sigma^{\frac{1+(-1)^{k}}{2}}$ acts on the factor $\QM_{g_2,n_2+1,\beta_\infty}$ and $(W/\!\!/ G)^{n_2+1}$. 

It is easy to check that the pullback of function (below we denote $\phi=\phi_{W/\!\!/ G}$ for simplicity)
$$\phi_{}^{\boxplus (n_1+1)}\boxplus (\phi \boxplus \phi)^{\boxplus k} \boxplus \phi^{\boxplus (n_2+1)}: (W/\!\!/ G)^{n_1+1} \times ((W/\!\!/ G)^2)^{k}\times (W/\!\!/ G)^{n_2+1}\to \C $$
via the isomorphism $\left(\id\times \prod_{i=1}^k\sigma^{\frac{1-(-1)^{i}}{2}}\times \sigma^{\frac{1+(-1)^{k}}{2}}\right)$ becomes: 
$$\phi^{\boxplus (n_1+1)}\boxplus \left(\boxplus_{i=1}^k(-1)^i(\phi \boxplus \phi)\right) \boxplus ((-1)^{k+1}\phi)^{\boxplus (n_2+1)}: (W/\!\!/ G)^{n_1+1} \times ((W/\!\!/ G)^2)^{k} \times (W/\!\!/ G)^{n_2+1}\to \C. $$
Therefore, we obtain: 
\begin{prop}\label{lem of crit k via sigma}
The above map induces an isomorphism of critical $K$-groups: 
\begin{align*}
&\quad \, \, K_0^T\left((W/\!\!/ G)^{n_1+1} \times ((W/\!\!/ G)^2)^{k}\times (W/\!\!/ G)^{n_2+1},\phi^{\boxplus (n_1+1)}\boxplus (\phi \boxplus \phi)^{\boxplus k} \boxplus \phi^{\boxplus (n_2+1)}\right) \\ \nonumber 
&\cong K_0^T\left((W/\!\!/ G)^{n_1+1} \times ((W/\!\!/ G)^2)^{k}\times (W/\!\!/ G)^{n_2+1},\phi^{\boxplus (n_1+1)}\boxplus \left(\boxplus_{i=1}^k(-1)^i(\phi \boxplus \phi)\right) \boxplus ((-1)^{k+1}\phi)^{\boxplus (n_2+1)}\right).
\end{align*}
\end{prop}
\begin{remark}\label{rmk on repack fun}
By repacking the functions, we have 
\begin{align*}
&\quad \, \, K_0^T\left((W/\!\!/ G)^{n_1+1} \times ((W/\!\!/ G)^2)^{k}\times (W/\!\!/ G)^{n_2+1},\phi^{\boxplus (n_1+1)}\boxplus \left(\boxplus_{i=1}^k(-1)^i(\phi \boxplus \phi)\right) \boxplus ((-1)^{k+1}\phi)^{\boxplus (n_2+1)}\right) \\ \nonumber 
&\cong K_0^T\left((W/\!\!/ G)^{n_1} \times ((W/\!\!/ G)^2)^{k+1}\times (W/\!\!/ G)^{n_2},\phi^{\boxplus \,n_1}\boxplus \left(\boxplus_{i=1}^{k+1}(-1)^i(-\phi \boxplus \phi)\right) \boxplus ((-1)^{k+1}\phi)^{\boxplus \,n_2}\right).
\end{align*}
\end{remark}

\subsubsection{Gluing formula}
For a decomposition $n=n_1+n_2$, $g=g_1+g_2$ (with $2g_i-2+n_i+1>0$ for $i=1,2$), we have the \textit{gluing morphism}
\begin{equation}\label{equ on glue mor}\iota: \overline{\calM}_{g_1,n_1+1}  \times \overline{\calM}_{g_2,n_2+1} \to \overline{\calM}_{g,n}. \end{equation}
The image of $\iota$ is a divisor. 
The preimage of $\Image(\iota)$ along the map $\QM_{g,n,\beta}\to \overline{\calM}_{g,n}$ is a reducible divisor. 
Each decomposition of curve class gives a component on this divisor. 

More precisely, let $\bbX(G)$ denote the character group of $G$. For any $\beta\in \Hom_{\mathbb{Z}}(\bbX(G),\mathbb{Z})$, and a decomposition of curve class 
$$\beta=\beta_0+\beta_1+\cdots+\beta_k+\beta_\infty, $$ 
we consider Cartesian diagram
\begin{equation}\label{eqn:bdy_divisor_curve}
{\footnotesize{
\xymatrix{
\QM_{g_1,n_1+1,\beta_0}\times_{\Delta}\cdots\times_{\Delta}\QM_{g_2,n_2+1,\beta_\infty}\ar[d]\ar[rr] \ar@{}[dr]|{\Box} 
& &\QM_{g_1,n_1+1,\beta_0}\times \QM_{0,2,\beta_1}\times \cdots\times \QM_{0,2,\beta_k}\times \QM_{g_2,n_2+1,\beta_\infty}\ar[d] \\
(W/\!\!/ G)^{n_1}\times (W/\!\!/ G)^{k+1}\times (W/\!\!/ G)^{n_2}\ar[rr]^{\id^{n_1}\times \Delta^{k+1} \times \id^{n_2} \quad \quad \quad \quad } &&
(W/\!\!/ G)^{n_1}\times(W/\!\!/ G) \times ((W/\!\!/ G)^2)^{k}\times(W/\!\!/ G) \times (W/\!\!/ G)^{n_2},
} }}
\end{equation}
where $\Delta: W/\!\!/ G\to W/\!\!/ G\times W/\!\!/ G$ is the \textit{diagonal map},
the upper-left corner being a substack of $\QM_{g,n,\beta}$, which under the map $\QM_{g,n,\beta}\to \overline{\calM}_{g,n}$ lands in the image of $\iota$,   

Proper pushforward along the diagonal map 
$$\Delta_*: K_0^T\left(W/\!\!/ G\right)\to K_0^T\left(W/\!\!/ G\times W/\!\!/ G, -\phi_{W/\!\!/ G}\boxplus \phi_{W/\!\!/ G}\right)$$
gives the diagonal class
$$\oO_{\Delta}:=\Delta_*\oO_{W/\!\!/ G}. $$
For any $k\in\bbN$, consider the product 
\begin{equation}\label{equ on higher diag}\oO_{\Delta^{k+1}}:=( \oO_{\Delta})^{\otimes(k+1)}\in 
K_0^T\left(W/\!\!/ G\times W/\!\!/ G, -\phi_{W/\!\!/ G}\boxplus \phi_{W/\!\!/ G}\right)^{\otimes (k+1)}. \end{equation}
We define a map 
\begin{align}\label{map phi tens}
&\Phi^{\Crit,K}_{g_1,n_1+1,\beta_0}\otimes \Phi^{\Crit,K}_{0,2,\beta_1}\otimes \cdots \otimes\Phi^{\Crit,K}_{0,2,\beta_k}\otimes \Phi^{\Crit,K}_{g_2,n_2+1,\beta_{\infty}}: 
\\  \nonumber
&\quad \quad  K_0(\overline{\mathcal{M}}_{g_1,n_1+1}\times \overline{\mathcal{M}}_{g_2,n_2+1})\otimes  
K_0^T\left((W/\!\!/ G)^n\times (W/\!\!/ G)^{2k+2},\phi_{W/\!\!/ G}^{\boxplus n}\boxplus \phi_{W/\!\!/ G}^{\boxplus (2k+2)}\right) 
\to K_0^T\left(\pt\right)_{\loc},
\end{align} 
as the critical pullback applied to the right column of \eqref{eqn:bdy_divisor_curve} followed by equivariant pushforward to a point. 
Using Proposition \ref{lem of crit k via sigma} and Remark \ref{rmk on repack fun}, we can evaluate \eqref{map phi tens} on  
\begin{equation} 
\alpha\boxtimes \gamma \boxtimes \oO_{\Delta^{k+1}},
\nonumber \end{equation}
where $$\alpha\in K_0(\overline{\mathcal{M}}_{g_1,n_1+1}\times \overline{\mathcal{M}}_{g_2,n_2+1}), \,\,\, \gamma\in K_0^T\left((W/\!\!/ G)^{n_1},\phi_{W/\!\!/ G}^{\boxplus n_1}\right)\otimes K_0^T\left((W/\!\!/ G)^{n_2},(-1)^{k+1}\phi_{W/\!\!/ G}^{\boxplus n_2}\right).$$
By the isomorphism $\sigma^{\frac{1+(-1)^{k}}{2}}$, $\gamma$ is also a class in 
$K_0^T\left((W/\!\!/ G)^{n_1},\phi_{W/\!\!/ G}^{\boxplus n_1}\right)\otimes K_0^T\left((W/\!\!/ G)^{n_2},\phi_{W/\!\!/ G}^{\boxplus n_2}\right)$.
\begin{theorem}
[Gluing formula in quantum critical $K$-theory]
\label{thm on glue in GLSM}
For any classes 
$$\alpha\in K_0(\overline{\mathcal{M}}_{g_1,n_1+1}\times \overline{\mathcal{M}}_{g_2,n_2+1}),
\,\,\, \gamma\in K_0^T\left((W/\!\!/ G)^{n_1},\phi_{W/\!\!/ G}^{\boxplus n_1}\right)\otimes K_0^T\left((W/\!\!/ G)^{n_2},(-1)^{k+1}\phi_{W/\!\!/ G}^{\boxplus n_2}\right), $$
we have
\begin{align} 
\label{form on glue in GLSM}
&\quad \,\, \Phi^{\Crit,K}_{g,n,\beta}(\iota_*\alpha\boxtimes\gamma)\\  \nonumber
&=\sum_{k=0}^{\infty}(-1)^k\sum_{\beta_0+\beta_1+\cdots+\beta_k+\beta_\infty=\beta}\Phi^{\Crit,K}_{g_1,n_1+1,\beta_0}\otimes \Phi^{\Crit,K}_{0,2,\beta_1}\otimes \cdots \otimes\Phi^{\Crit,K}_{0,2,\beta_k}\otimes \Phi^{\Crit,K}_{g_2,n_2+1,\beta_{\infty}}(\alpha\boxtimes \gamma\boxtimes\oO_{\Delta^{k+1}}).
\end{align}
 \end{theorem}
\begin{proof}
Taking into account the inclusion-exclusion principle \cite[Lem.~3]{Lee} and the functoriality (Theorem \ref{thm on funct}), the proof of \cite[Thm.~5.7]{CZ} 
implies the claim, with the role of square root virtual pullbacks in \cite[Prop.~4.22]{CZ} replaced by critical pullbacks. 
We leave details to interested readers.
\end{proof}

\subsubsection{Gluing formula under decomposition of diagonal}
Now we assume the potential function
$$\phi_{W/\!\!/ G}: W/\!\!/ G\to \C, $$
satisfies the decomposition of diagonal \eqref{equ on kunn prod}.  
The diagonal class decomposes according to a basis of the critical $K$-group:
$$\oO_{\Delta}=\sum_{i,j}g_{ij}\,\eta_{i}\otimes \eta_{j}.$$
Taking product, for any $k\in\bbN$
$$\oO_{\Delta^{k+1}}=\sum_{i_0,\cdots,i_{2k+1}}\bigotimes_{m=0}^{k}g_{i_{2m},i_{2m+1}}\eta_{i_{2m}}\otimes \eta_{i_{2m+1}}.$$
We write $$\gamma=\gamma_1\otimes\cdots \otimes \gamma_n\in K_0^T\left(W/\!\!/ G,\phi_{W/\!\!/ G}\right)^{\otimes n}_{\loc}\cong  K_0^T\left((W/\!\!/ G)^n,\phi_{W/\!\!/ G}^{\boxplus n}\right)_{\loc}. $$
Let $$\alpha=\alpha_1\otimes\alpha_2\in K_0(\overline{\mathcal{M}}_{g_1,n_1+1})\otimes K_0(\overline{\mathcal{M}}_{g_2,n_2+1}).$$
The right-hand side of Eqn.~\eqref{form on glue in GLSM} then takes the more familiar form:
\begin{align}
\label{equ on pair n}
&\quad \,\, \Phi^{\Crit,K}_{g_1,n_1+1,\beta_0}\otimes \Phi^{\Crit,K}_{0,2,\beta_1}\otimes \cdots \otimes\Phi^{\Crit,K}_{0,2,\beta_k}\otimes \Phi^{\Crit,K}_{g_2,n_2+1,\beta_{\infty}}(\alpha\boxtimes\gamma\boxtimes\oO_{\Delta^{k+1}})= \sum_{i_0,\cdots,i_{2k+1}}\prod_{m=0}^{k}g_{i_{2m},i_{2m+1}} \\ \nonumber 
&\cdot \Phi^{\Crit,K}_{g_1,n_1+1,\beta_0}(\alpha_1\boxtimes \gamma_1\otimes\cdots \otimes\gamma_{n_1}\otimes \eta_{i_0})\cdot \prod_{m=1}^{k}\Phi^{\Crit,K}_{0,2,\beta_m}(\eta_{i_{2m-1}}\otimes\eta_{i_{2m}}) \cdot \Phi^{\Crit,K}_{g_2,n_2+1,\beta_{\infty}}(\alpha_2\boxtimes\gamma_{n_1+1}\otimes\cdots \otimes\gamma_{n}\otimes\eta_{i_{2k+1}}), 
\end{align}
written as a sum over tensor products of \eqref{equ on phign} and \eqref{equ on phi02}.



\subsection{Application II: Donaldson-Thomas theory of local Calabi-Yau 4-folds} 
Let $(Y,D)$ be a 4-dimensional (local) log Calabi-Yau pair \cite[Eqn.~(1.3)]{CZZ},~i.e.
\begin{equation}\label{equ on log cy}Y=\Tot(\omega_{U}(S)), \quad D=\Tot(\omega_S),  \end{equation}
where $U$ is a smooth 3-fold and $S$ is a smooth divisor of $U$. Note that $D$ is an anti-canonical divisor of $Y$.
We assume in below that for all numerical $K$-theory classes $R$ on $D$ considered, there is an equivalence 
of $(-1)$-shifted symplectic derived schemes: 
\begin{equation}\label{equ on iso of hilbd}\textbf{Hilb}^{R}(D)\cong \textbf{Crit}\left(\phi_R\right), \end{equation}
where $\phi_R$ is a regular function on a smooth scheme. Let $T$ be a torus acting on the pair $(Y,D)$ and domain of $\phi_R$ (preserving the function 
$\phi_R$), so that the action is compatible under \eqref{equ on iso of hilbd}.

Given a simple degeneration $X\to \mathbb{A}^1$ of local Calabi-Yau 4-folds \cite[Defs.~2.6,~2.10]{CZZ}, such that the central fiber 
is the union  
$$X_0=Y_-\cup_D Y_+, $$
where $(Y_-,D)$, $(Y_+,D)$ are two log Calabi-Yau pairs as \eqref{equ on log cy} 
and $D$ is connected (for simplicity), and given numerically equivalent $K$-theory classes $P_t \in K_{c,\leqslant 1}^\mathrm{{num}} (X/\mathbb{A}^1)$ for all $t\in \bbA^1$ \cite[Def.~2.16,~Eqn.~(2.11)]{CZZ}, 
there is a \textit{family virtual class}\footnote{The $\Phi^{P_t}_{X/\bbA^1}$ here is the $K$-theoretic analogue of $\Phi^{P_t}_{X/\bbA^1}([\calC^{P_t}])$ in \cite[Def.~5.13]{CZZ}.}:\begin{equation*}\Phi^{P_t}_{X/\bbA^1}\in K_0^T\left(\Hilb^{P_t}(X/\bbA^1)\right)_{ }.  \end{equation*}
Its Gysin pullback under inclusion $0\hookrightarrow \bbA^1$ \cite[diagram~(5.21)]{CZZ} gives 
\begin{equation}\label{equ on i0}i_0^!\Phi^{P_t}_{X/\bbA^1}\in K_0^T\left(\Hilb^{P_t}(X/\bbA^1)_0\right)_{ }. \end{equation}
Here $\Hilb^{P_t}(X/\bbA^1)_0$ is the zero fiber of $\Hilb^{P_t}(X/\bbA^1)\to \bbA^1$, which admits gluing maps: 
\begin{equation}\label{equ on glue map}
{
\footnotesize{
\xymatrix{
\Hilb^{Q_0}(Y_-,D)\times_{\Hilb^{R_0}(D)} \Hilb^{Q_1}(\Delta)^{\sim}\times_{\Hilb^{R_1}(D)}\cdots \times_{\Hilb^{R_{k-1}}(D)}\Hilb^{Q_k}(\Delta)^{\sim}\times_{\Hilb^{R_{k}}(D)} \Hilb^{Q_{\infty}}(Y_+,D) \to 
\Hilb^{P_t}(X/\bbA^1)_0, } } }
\end{equation}
where $\Delta=\mathbb{P}(N_{D/Y_-}\oplus \oO_D)$ is the bubble with two anti-canonical divisors $D_0, D_\infty$, $ \Hilb^{-}(\Delta)^{\sim}$ denotes the rubber Hilbert stack (e.g.~\cite[Def.~A.1]{CZZ}), 
and 
$$Q_0|_D=Q_1|_{D_0}=R_0, \,\,\,  Q_1|_{D_\infty}=Q_2|_{D_0}=R_1, \,\,\, \cdots, \,\,\, Q_k|_{D_\infty}=Q_\infty|_{D}=R_k, $$
$$P_0=\sum_{i=0}^kQ_i+Q_\infty-\sum_{i=0}^kR_i.$$
We denote 
\begin{equation}\label{equ on class delta}\delta=(Q_0,Q_1,\cdots,Q_k,Q_\infty,R_0,R_1,\cdots,R_k) \end{equation} 
to be the collection of those numerical classes, with $k(\delta):=k$ recording the number of bubble components, and  
$$P(\delta):=\sum_{i=0}^kQ_i+Q_\infty-\sum_{i=0}^kR_i.$$
Under assumptions in Lemma \ref{lem on exi of reso2}, critical pullbacks define \textit{relative virtual classes} of $\Hilb^{Q_0}(Y_{-},D)$, $\Hilb^{Q_\infty}(Y_{+},D)$:
\begin{equation*} \Phi_{Y_{-},D}^{Q_{0}}:
K_0^T\left(\phi_{Q_{0}|_D}\right) \to K_0^T\left(\Hilb^{Q_{0}}(Y_{-},D)\right)_{ }, \end{equation*}
\begin{equation*} \Phi_{Y_{+},D}^{Q_{\infty}}:
K_0^T\left(\phi_{Q_{\infty}|_D}\right) \to K_0^T\left(\Hilb^{Q_{\infty}}(Y_{+},D)\right)_{ }, \end{equation*}
where $K_0^T\left(\phi_{Q_{(-)}|_D}\right)$ is the shorthand for the critical $K$-groups of regular functions $\phi_{Q_{(-)}|_D}$. 

And similarly there are \textit{relative virtual classes} of rubber moduli stacks $\Hilb^{Q_i}(\Delta)^{\sim}$: 
$$\Phi_{\Delta^{\sim}}^{Q_i}: K_0^T\left(\phi_{Q_i|_{D_0}}\right)\otimes K_0^T\left(\phi_{Q_i|_{D_\infty}}\right) \to K_0^T\left( \Hilb^{Q_i}(\Delta)^{\sim}\right)_{ }.$$
Given $\delta$ as \eqref{equ on class delta}, we have
$R_i=Q_i|_{D_\infty}=Q_{i+1}|_{D_0}$ and diagonal class 
$$\oO_{\Delta_i}=\Delta_{i*}\oO\in K_0^T\left(-\phi_{R_i}\right)\otimes K_0^T\left(\phi_{R_i}\right)\cong K_0^T\left(\phi_{R_i}\right)\otimes K_0^T\left(\phi_{R_i}\right), $$ 
where 
$$\Delta_i: W_i/\!\!/ G_i\to W_i/\!\!/ G_i\times W_i/\!\!/ G_i, \quad x\mapsto (x,x)$$
is the diagonal embedding and $W_i/\!\!/ G_i$ is the domain of the function $\phi_{R_i}$, and the  isomorphism is the canonical one \eqref{equ on cano iso on minus}. 

We denote the product diagonal class 
\begin{align}\label{equ on higher diag2}\oO_{\Delta^{k+1}}:=\bigotimes_{i=0}^{k}\oO_{\Delta_i}&  \in  \,
\bigotimes_{i=0}^{k} \left(K_0^T\left(\phi_{R_i}\right)\otimes K_0^T\left(\phi_{R_i}\right)\right) \\ \nonumber
&\cong  K_0^T\left(\phi_{Q_0|_{D}}\right) \otimes \bigotimes_{i=1}^{k}\left(K_0^T\left(\phi_{Q_i|_{D_0}}\right) \otimes K_0^T\left(\phi_{Q_i|_{D_\infty}}\right)\right)
\otimes K_0^T\left(\phi_{Q_\infty|_{D}}\right).
\end{align}
As in the previous section, following the proof of \cite[Thm.~5.21]{CZZ} and \cite[\S 4.2]{Qu}, we have:
\begin{theorem}
[Degeneration formula of $K$-theoretic $\DT_4$ invariants]
\label{thm on glue in DT4}

\begin{align}\label{equ on glue for of DT4} i_0^!(\Phi^{P_t}_{X/\bbA^1})=\sum_{k=0}^{\infty}(-1)^k\sum_{\begin{subarray}{c}\delta \\ k(\delta)=k \\ P(\delta)=P_0   \end{subarray}} \Phi_{Y_{-},D}^{Q_0}\otimes \Phi_{\Delta^{\sim}}^{Q_1}\otimes 
\cdots \otimes\Phi_{\Delta^{\sim}}^{Q_k}\otimes \Phi_{Y_{+},D}^{Q_\infty}(\oO_{\Delta^{k+1}}),  
\end{align}
where $\oO_{\Delta^{k+1}}$ is the class \eqref{equ on higher diag2} and the RHS lies in $K_0^T\left(\Hilb^{P_t}(X/\bbA^1)_0\right)_{}$ by the map
\eqref{equ on glue map}. 
\end{theorem}

\appendix 


\section{}\label{app on sympl}
We recall basics in shifted symplectic structures (ref.~\cite{PTVV, Cal, CPTVV, Park2}). 

\begin{definition} 
A graded mixed complex $E$ (over a base ring $R$) consists of the following data: 
\begin{itemize}
\item A $\mathbb{Z}$-graded complex of $R$-modules 
$$E=\bigoplus_{p\in \mathbb{Z}}E(p), $$
where each $E(p)$ has internal differential $d_{E}|_{E(p)}\equiv d_{E(p)}$. 
\item A map of chain complexes 
$$\epsilon: E(p)\to E(p+1)[-1]$$
such that $\epsilon^2=0$. 
\end{itemize}
\end{definition}

\begin{definition} 
Let $M\to N$ be a map of derived stacks over $R$ and $DR(M/N)$ be the derived de Rham complex \cite[Def.~2.4.2]{CPTVV}, which is a graded mixed complex. 
Consider the \textit{negative cyclic}, \textit{periodic cyclic}, \textit{Hochschild} complexes of weight $p$
$$NC(DR(M/N))(p), \quad PC(DR(M/N))(p), \quad C(DR(M/N))(p), $$
where 
$$NC^k(E)(p):=\prod_{i\geqslant 0} E^{k-2i}(p+i), \quad d_{NC(E)}=d_E+\epsilon, $$
$$PC^k(E)(p):=\prod_{i\in \mathbb{Z}} E^{k-2i}(p+i), \quad d_{PC(E)}=d_E+\epsilon, $$
$$C^k(E)(p):=E^{k}(p), \quad d_{C(E)}=d_E, $$
and the corresponding \textit{negative cyclic}, \textit{periodic cyclic}, \textit{Hochschild} homology 
$$HN^k(DR(M/N))(p), \quad HP^k(DR(M/N))(p), \quad HH^k(DR(M/N))(p). $$
\end{definition}
\begin{remark}
We refer to elements in $HN^{k-p}(DR(M/N))(p)$ (resp.~$HH^{k-p}(DR(M/N))(p)$) as $k$-\textit{shifted} closed $p$-\textit{forms} (resp.~$k$-\textit{shifted} $p$-\textit{forms}). 
It is straightforward to define canonical maps:
$$HN^k(DR(M/N))(p)\to HP^k(DR(M/N))(p), \quad HN^k(DR(M/N))(p)\to HH^k(DR(M/N))(p).  $$
\end{remark}
\begin{definition} 
Let $f: M\to N$ be a map of derived stacks over $R$. An $n$-\textit{shifted symplectic structure} on $f$ is an element
$\Omega\in HN^{n-2}(DR(M/N))(2)$ such that its image in $HH^{n-2}(DR(M/N))(2)$ is non-degenerate,~i.e. inducing a quasi-isomorphism 
$$\bbT_{M/N}\cong \bbL_{M/N}[n]. $$
\end{definition}

\begin{definition} 
Let $f: M\to N$ be a map of derived stacks over a derived stack $B$ and $N\to B$ be the structure map endowed with an $n$-shifted symplectic structure $\Omega$. 
A \textit{Lagrangian structure} on $f$ (relative to $B$) is a nullhomotopy 
$$f^*\Omega \thicksim 0$$
in the space $\calA^{2,cl}(M/B)(n)$ of $n$-shifted closed $2$-forms, such that the induced morphism
$$\bbT_{f}\to \bbL_{M/B}[n-1] $$
is a quasi-isomorphism. 
\end{definition}
Next we recall the notion of orientations on shifted symplectic stacks, as in \cite[\S 2.4]{BJ}, \cite{CL1, CL2}. 
\begin{definition}\label{ori on even cy}
Let $f: M\to N$ be a map of derived stacks which has a (relative) $n$-shifted symplectic structure (with even $n$). The symplectic structure induces an isomorphism
\begin{equation}\label{serre iso even}\det\left(\bbT_{f}|_{t_0(M)}\right)^{\otimes 2}\cong \oO_{t_0(M)} \end{equation}
of the restriction of the tangent complex of $f$ to the classical truncation $t_0(M)$ of $M$. 

An \textit{orientation} on $f$ is a square root of the above isomorphism \eqref{serre iso even}. 
\end{definition}
\begin{definition}\label{def of rel or}
Let $N$ be a $n$-shifted symplectic derived stack with odd $n$ and $f: M\to N$ be a map of derived stacks which has a Lagrangian structure. 
The Lagrangian structure induces an isomorphism 
\begin{equation}\label{rel ori iso}\left(\det\left(\bbT_{f}|_{t_0(M)}\right)\right)^{\otimes 2}\cong 
t_0(f)^*\det\left(\bbT_{N}|_{t_0(N)} \right) \end{equation}
of the restriction of the tangent complexes of $f$ and $N$ to the classical truncation $t_0(M)$ of $M$. 

An \textit{orientation} of $N$ is a choice of a line bundle $\det\left(\bbT_{N}|_{t_0(N)} \right)^{1/2}$ on $t_0(N)$ such that 
$$\det\left(\bbT_{N}|_{t_0(N)} \right)^{1/2}\otimes \det\left(\bbT_{N}|_{t_0(N)} \right)^{1/2}\cong \det\left(\bbT_{N}|_{t_0(N)} \right). $$

An \textit{orientation} of $f$ is a choice of square root $\det\left(\bbT_{N}|_{t_0(N)} \right)^{1/2}$ as above and an isomorphism 
$$\det\left(\bbT_{f}|_{t_0(M)}\right)\cong t_0(f)^*\det\left(\bbT_{N}|_{t_0(N)} \right)^{1/2}, $$
whose square is the isomorphism \eqref{rel ori iso}. 
\end{definition}
\begin{remark}\label{rmk on can ori}
Consider the derived critical locus $\textbf{Crit}(\boldsymbol{\phi})$ of $\boldsymbol{\phi}: \textbf{\emph{B}}\to \C$ as in Definition \ref{defi of derived crit loci}. 
Then we have a canonical isomorphism 
$$\det\left(\bbT_{\textbf{Crit}(\boldsymbol{\phi})}\right)\cong \det\left(\bbT_{\textbf{\emph{B}}}|_{\textbf{Crit}(\boldsymbol{\phi})}\right)^{\otimes 2}. $$
Therefore $\det\left(\bbT_{\textbf{\emph{B}}}|_{\Crit(\phi)}\right)$ provides a canonical choice of orientation for $\textbf{Crit}(\boldsymbol{\phi})$. 
\end{remark}
 \begin{remark}
By compositing a Lagrangian $f: \textbf{\emph{M}}\to  \textbf{Crit}(\boldsymbol{\phi})$ and the Lagrangian fibration $p: \textbf{Crit}(\boldsymbol{\phi})\to \textbf{\emph{B}}$,
we obtain a $(-2)$-shifted symplectic structure on map $p\circ f: \textbf{\emph{M}}\to \textbf{\emph{B}}$ \cite[Lem.~3.6]{CZZ}. One can easily check 
that an orientation of $f$ and the canonical orientation in the previous remark induces an orientation of $p\circ f$ (e.g.~\cite[Prop.~4.11]{CZZ}). 
\end{remark}

\section{}\label{app b}
\begin{prop}\label{prop on ex sat set}
We have 

\begin{itemize}
\item Example \ref{ex of papers} (1) satisfies all conditions in Setting \ref{setting of lag}. 
\item Example \ref{ex of papers} (2) satisfies all conditions in Setting \ref{setting of lag} but the (Resolution) property. 
\end{itemize}
\end{prop}
\begin{proof}
We first fix the map $\pi_X$ in both examples. 

In (1), let $\overline{\mathcal{M}}_{g,n}$ be the moduli stack  of genus $g$, $n$-pointed stable curves. It is a smooth proper 
Deligne-Mumford stack whose coarse moduli space is projective \cite{Kn}. 
Therefore $\overline{\mathcal{M}}_{g,n}\cong [Y/K]$ for 
a quasi-projective scheme $Y$ and a linear algebraic group $K$ \cite[Thm.~4.4]{Kre2}. 
Then we can take 
$$X=\overline{\mathcal{M}}_{g,n}\times (W/\!\!/ G)^n  $$ 
with the regular function $\phi_X$ being the pullback of the regular function $\phi_{W/\!\!/ G}^{\boxplus n}$ on $(W/\!\!/ G)^n$. 
And 
$$\pi_X: \fBun_{H_R,g,n}^{R_{\chi}=\omega_{\mathrm{log}}}\times_{[\pt/G]^n}(W/\!\!/ G)^n\to 
\fBun_{H_R,g,n}^{R_{\chi}=\omega_{\mathrm{log}}}\times (W/\!\!/ G)^n\to \overline{\mathcal{M}}_{g,n}\times (W/\!\!/ G)^n$$
is the composition, where the last map is the composition of the (smooth) forgetful map to the moduli stack $\fM_{g,n}$ of prestable curves 
with the (flat) stabilization map \cite{Beh}:
\begin{equation*}\fM_{g,n}\to \overline{\mathcal{M}}_{g,n} \end{equation*}
to the moduli of stable curves.

In (2), we take $X=(W/\!\!/ G)^n$, with projection 
$$\pi_X: \calA^P\times (W/\!\!/ G)^n\to (W/\!\!/ G)^n. $$
In both cases, the domain of $\emph{\textbf{f}}$ is a derived stack whose classical truncation is separated, Deligne-Mumford, finite type over $\C$ (e.g.~\cite[Thm.~2.11]{CZ}, \cite[Prop.~2.25]{CZZ}), 
and the map $f$ is quasi-projective (e.g.~\cite[\S 4.2]{CiKM}, \cite[Proof of Prop.~4.4]{LW}).
 The target is a smooth quasi-separated algebraic stack, locally of finite type over $\C$. Although it is not quasi-compact, but we can replace it by 
 a quasi-compact open subset where the image of $f$ lies in. The affine stabilizer property is proved in Lemma \ref{lem on aff stab}. 
 This finishes the proof of (LG), (DM) conditions. 

The remaining properties are verified in Lemma \ref{lem on ori lag}, Lemma \ref{lem on exi of reso} and Lemma \ref{lem on exi of reso2}.  
\end{proof}

\begin{lemma}\label{lem on aff stab}
The targets of ${\textbf{f}}$ in Example \ref{ex of papers} have affine stabilizer.
\end{lemma}
\begin{proof}
For (1), it is enough to show that $\fBun_{H_R,g,n}^{R_{\chi}=\omega_{\mathrm{log}}}$ has affine stabilizer ($H_R=G\times \C^*$).  
Recall that we have a Cartesian diagram 
$$
\xymatrix{
A:=\fBun_{H_R,g,n}^{R_{\chi}=\omega_{\mathrm{log}}} \ar[r]^{ } \ar[d]_{}  \ar@{}[dr]|{\Box}    &  \{\omega_{\mathrm{log}}\} \ar[d]^{ } \\
B:=\fBun_{H_R,g,n} \ar[r]^{ } & C:=\fBun_{\C^*,g,n}. 
}
$$
Let $x\in A$ be a closed point, then 
$$1\to \Aut_A(x)\to \Aut_B(x)\to \Aut_C(x)$$
is exact. Therefore, it is enough to show $\Aut_B(x)$ is affine. 

Take a closed point $\{C\}$ of the moduli stack $\mathfrak{M}_{g,n}$ of prestable genus $g$ curves with $n$-marked points, 
denote $\fBun_{H_R}(C)$ to be the moduli stack of principle $H_R$-bundle over $C$. 
There is a Cartesian diagram 
$$
\xymatrix{
B':=\fBun_{H_R}(C) \ar[r]^{ } \ar[d]_{}  \ar@{}[dr]|{\Box}    &  \{C\} \ar[d]^{ } \\
B:=\fBun_{H_R,g,n} \ar[r]^{ } & \mathfrak{M}:=\mathfrak{M}_{g,n},
}
$$
and an exact sequence
$$1\to \Aut_{B'}(x)\to \Aut_B(x)\to \mathrm{Im}(\Aut_B(x)\to \Aut_{\mathfrak{M}}(x))\to 1. $$
Note that $\fBun_{H_R}(C)$ and $\mathfrak{M}_{g,n}$ (when $(g,n)\neq (1,0)$)
have affine stabilizers (e.g.~\cite[Rmk.~2.2.6]{DG2}, \cite[Prop.~3.1]{BS}). As an extension of affine algebraic group, we know 
$\Aut_B(x)$ is also affine \cite[pp.~379, $\mathrm{VI_B}$, Prop.~9.2 (viii)]{DeGr}.

For (2), this is \cite[Prop.~2.2]{CZZ}.
\end{proof}

\begin{lemma}\label{lem on ori lag}
(Oriented Lagrangian) and (Compatibility) of Setting \ref{setting of lag} hold for Example \ref{ex of papers}.
\end{lemma}
\begin{proof}
For (1), shifted symplectic structures with 
image in the periodic cycle homology given by $\phi$ (in the sense of Setting \ref{setting of lag}) are called $\phi$-locked symplectic forms in \cite[Notation~1.1.2]{Park2}. 
Note that the AKSZ type mapping stack construction works for $\phi$-locked symplectic forms \cite[\S 6.2]{Park2}).
By running through the construction in \cite[\S 3.3, \S 3.5]{CZ}, one can show the claim, where orientation is constructed as \cite[Rmk.~4.2]{CZ}.

For (2), by \cite[Thm.~3.2]{CZZ}, we have a Lagrangian structure (relative to $\calA^P$) on
$$\emph{\textbf{f}}: \textbf{Hilb}^P(Y,D) \to  \calA^P\times \textbf{Crit}(\phi)^n, $$
whose composition with $\textbf{Crit}(\phi)^n\to (W/\!\!/ G)^n$ has a canonically induced 
$(-2)$-shifted symplectic structure $\Omega_{\emph{\textbf{f}}}$ \cite[Prop.~3.7]{CZZ}. 
Orientation on $f$ is constructed in \cite[Prop.~4.11]{CZZ}. 
Note that \cite[Rmk.~4.1.7]{Park2} implies that the image of $\Omega_{\emph{\textbf{f}}}$ in the periodic cycle homology 
is given by the pullback function of $\phi+c$ for a locally constant function $c$. By \cite[Prop.~5.6]{CZZ}, we know $c=0$, 
this proves the compatibility with the function $\phi$.
\end{proof}

\begin{lemma}\label{lem on exi of reso}
(Resolution) of Setting \ref{setting of lag} holds for Example \ref{ex of papers}\,(1). 
\end{lemma}
\begin{proof}
Let $B=\fM_{g,n}$ and let $\pi: \calC\to B$ be the universal curve. 
Recall that the quasimap derived stack $\textbf{QM}:=\textbf{QM}_{g,n}^{R_{\chi}=\omega_{\mathrm{log}}}(\bCrit(\phi_{W/\!\!/ G}),\beta)$ is an open substack of the twisted mapping stack 
$\bMap^{R_{\chi}=\omega_{\mathrm{log}}}_{g,n}([\bCrit^{}(\phi_W)/H_R])$ (where $\phi_{W/\!\!/ G}$ is the descent of $\phi_W: W\to \C$). 
Let 
$$\pi: \bMap^{R_{\chi}=\omega_{\mathrm{log}}}_{g,n}([\bCrit^{}(\phi_W)/H_R])\times_B \calC\to \bMap^{R_{\chi}=\omega_{\mathrm{log}}}_{g,n}([\bCrit^{}(\phi_W)/H_R])$$ be the induced 
map, $\calP$ be the universal $H_R$-bundle and $\calW:=\calP\times _{H_R}W$ be the universal $W$-bundle.
By \cite[Prop.~3.26, Lem.~4.8]{CZ}, the relative tangent complex of 
\begin{equation}\label{eqn:f}
    \textbf{\emph{f}}:=ev^n\times_{[\pt/H_R]^n}\mu:\bMap^{R_{\chi}=\omega_{\mathrm{log}}}_{g,n}([\bCrit^{}(\phi_W)/H_R]) \to [W/H_R]^n\times_{[\pt/H_R]^n}\fBun_{H_R,g,n}^{R_\chi=\omega_{\mathrm{log}}}
\end{equation}
is given by 
$$\bbT_\textbf{\emph{{f}}}\cong \big(\pi_*\left(\calW\boxtimes (\omega^\vee_{\mathrm{log}}\otimes\omega_{\calC/B})\right)\to 
\pi_*\left(\calW^\vee\boxtimes\omega_{\mathrm{log}}\right)\big) $$ 
where the map is induced by $\mathrm{Hess}(\phi_W)$.
For simplicity, we write this complex as $\pi_*(\calV_0\to\calV_1)$.

When restricted to the classical truncation $\mathrm{QM}$ of $\textbf{QM}$,  the complex $\bbT_\textbf{\emph{{f}}}$ above has a resolution following the construction in \cite[\S 5.2]{CiKM}, which we briefly summarize below.
Let $\calO(1)$ be a $\pi$-ample line bundle with a preferred section $s$. We have  for $i=0,1$, an injective map
$s_n:\calV_i\to \calV_i(n)$ whose cokernel $\calB_i^n$ has support $Z(s_n)$ finite over  $\mathrm{QM}$. We choose $m$ large enough so that both $\calV_i(m)$ and the cokernel $\calB_i^m$ are $\pi$-acyclic. Notice that the map  $\mathrm{Hess}(\phi_W)$ induces maps on the resolutions, making the diagram commutative
\[\xymatrix{\calV_0(m)\ar[r]\ar[d]&\calB_0^m\ar[d]   \\
\calV_1(m)\ar[r]&\calB_1^m.
}
\]
Applying $\pi_*$ to the diagram above, and taking the  associated total complex, we obtain a 3-step resolution of $\bbT_\textbf{\emph{f}}$.

Notice that the Serre duality map $\bbT_\textbf{\emph{f}}\to \bbT_\textbf{\emph{{f}}}^\vee[2]$ on the 3-step resolution above is defined on the level of complexes. 
This follows from the fact that 
$\calV_1\cong(\calV_0)^\vee\otimes\omega_{\calC/B}$ is given by definitions of $\calV_0$ and $\calV_1$ above,  $\calB_0^m=\calV_0(m)|_{Z(s_m)}$, and the residue map (e.g.~\cite[Thm.~3.2.5]{Nee}): 
\[\pi_*(\calV_0(m)|_{Z(s_m)}\otimes \calV_1(m)[1])\cong \pi_*(\calV_0^\vee(m)\otimes\omega_{\calC/B}\otimes\calV_0(m)|_{Z(s_m)}[1]) \to \calO_{\mathrm{QM}}.\] 
Similarly for \[\pi_*(\calV_1(m)|_{Z(s_m)}\otimes \calV_0(m)[1])\to \calO_{\mathrm{QM}}.\] 
Finally, a construction following \cite[Proof of Prop.~4.1]{OT} symmetrizes the 3-step complex into a self-dual complex.
\end{proof}

We do not know whether the resolution property (of Setting \ref{setting of lag}) holds for Example \ref{ex of papers}\,(2) in general. The following gives a 
positive answer in interesting examples. 
\begin{lemma}\label{lem on exi of reso2}
Let $D_0$ be a Calabi-Yau 3-fold with a $\C^*$-action such that the holomorphic volume form has weight $1$, $\mathring{\omega}_{\mathbb{P}^1}(n)$ be the $\C^*$-bundle obtained by removing the zero section of $\oO_{\mathbb{P}^1}(n-2)$.  

Take $Y=\mathring{\omega}_{\mathbb{P}^1}(n)\times_{\C^*}D_0$, and consequently  the anti-canonical  divisor is a disjoint union $D=\bigsqcup_{i=1}^n D_0$. 
Assume that $\textbf{\emph{Hilb}}^{P|_{D}}(D)=\bigsqcup_{i=1}^n \textbf{\emph{Crit}}(\phi)$ for a regular function $\phi: W/\!\!/ G\to \C$
on a linear GIT quotient, and that $\textbf{\emph{Hilb}}^P(Y,D)$ is contained in the open substack in  
\eqref{equ on ope imm} below. 
Then (Resolution) of Setting \ref{setting of lag} holds in this case.
\end{lemma}

\begin{proof}
Let $\mathfrak{M}$ be the moduli stack of $n$-pointed trees of rational curves $C$ with a distinguished component $C_0\cong \mathbb{P}^1$ such that every irreducible component of $C$ (except $C_0$) has exactly two special points (nodes or marked points) \cite[\S 7.2]{CiK}.
Let $\pi: \calC\to \mathfrak{M}$ be the universal curve. 

Define $\bMap^{\mathrm{twist}}_{\mathfrak{M}}(\calC,\bCrit^{}(\phi)\times \mathfrak{M})$ by the following homotopy pullback diagram 
\begin{equation}\label{diag on twist mapp stac}\xymatrix{
\bMap^{\mathrm{twist}}_{\mathfrak{M}}(\calC,\bCrit^{}(\phi)\times \mathfrak{M}) \ar[r] \ar[d] \ar@{}[dr]|{\Box} & \{\mathring{\omega}_{\mathrm{log}} \} \ar[d]  \\
\bMap^{}_{\mathfrak{M}}(\calC,[\bCrit^{}(\phi)/\C^*]\times \mathfrak{M}) \ar[r]& \bMap^{ }_{\mathfrak{M}}(\calC,[\pt/\C^*]\times \mathfrak{M}). 
}\end{equation}
Here $\bMap$ means derived mapping stack and $\mathring{\omega_{\mathrm{log}}}$ is the principal $\C^*$-bundle over $\calC$ corresponding to the log-canonical bundle.
We claim that there is an open immersion 
\begin{equation}\label{equ on ope imm}\bMap^{\mathrm{twist}}_{\mathfrak{M}}(\calC,\bCrit^{}(\phi)\times \mathfrak{M}) \hookrightarrow   
\bMap_{\mathfrak{M}}(\mathring{\omega_{\mathrm{log}}}\times_{\C^*} D_0,\textbf{Perf}\times \mathfrak{M})_{0} \end{equation}
with the target the derived moduli stack of perfect complexes with trivial determinant on $(\mathring{\omega_{\mathrm{log}}}\times_{\C^*} D_0)/\mathfrak{M}$. 
Indeed, we have 
\begin{align*}\bMap^{}_{\mathfrak{M}}(\calC,[\bCrit^{}(\phi)/\C^*]\times \mathfrak{M})&\cong 
\bMap^{\C^*\emph{-}\mathrm{equi}}_{\mathfrak{M}}(\calP_{\C^*},\bCrit^{}(\phi)\times \mathfrak{M}) \\
& \stackrel{\mathrm{open}}{\hookrightarrow}  \bMap^{\C^*\emph{-}\mathrm{equi}}_{\mathfrak{M}}(\calP_{\C^*},\bMap(D_0,\textbf{Perf})_0\times \mathfrak{M}) \\
&\cong  \bMap_{\mathfrak{M}}(\calP_{\C^*}\times_{\C^*} D_0,\textbf{Perf}\times \mathfrak{M})_0.
\end{align*}
By base change through diagram \eqref{diag on twist mapp stac}, we obtain \eqref{equ on ope imm}. 

We have a commutative diagram 
$$\xymatrix{
\bMap^{\mathrm{twist}}_{\mathfrak{M}}(\calC,\bCrit^{}(\phi)\times \mathfrak{M}) \ar[r]^{\mathrm{open}\quad \, \, \, \, } \ar[d]   &\bMap_{\mathfrak{M}}(\mathring{\omega_{\mathrm{log}}}\times_{\C^*} D_0,\textbf{Perf}\times \mathfrak{M})_0 \ar[d]  \\
(\bCrit^{}(\phi))^{\times n}\times \mathfrak{M} \ar[r]^{\mathrm{open}\quad\quad \, \, \, \,} & \bMap^{ }_{\mathfrak{M}}(\bigsqcup_{i=1}^nD_0,\textbf{Perf})_0\times \mathfrak{M},
}$$
where the left vertical map is evaluation at marked points, and the right vertical map is induced by 
the inclusion $\bigsqcup_{i=1}^n (D_0\times \mathfrak{M})\hookrightarrow (\mathring{\omega_{\mathrm{log}}}\times_{\C^*} D_0)$ of the anti-canonical divisor relative to $\mathfrak{M}$.  

Let 
$\pi: \bMap^{\mathrm{twist}}_{\mathfrak{M}}(\calC,\bCrit^{}(\phi)\times \mathfrak{M})\times_{\mathfrak{M}} \calC \to 
\bMap^{\mathrm{twist}}_{\mathfrak{M}}(\calC,\bCrit^{}(\phi)\times \mathfrak{M})$ be the induced projection, 
$ad_{\fg}\calP_G:=\calP_{G}\times_{G}\fg$ be the adjoint bundle, $\calW:=\calP_{G}\times _{G}W$ the universal $W$-bundle,  
$\oO(S)=\omega_{\log}\otimes \omega_{\calC/\mathfrak{M}}^{-1}$ be the line bundle defined by the divisor $S\subset \calC$ of marked points.
Using the isomorphism below \cite[Eqn.~(3.10) of pp.~17]{CZ}, 
we know the tangent complex of 
$$\bMap^{\mathrm{twist}}_{\mathfrak{M}}(\calC,\bCrit^{}(\phi)\times \mathfrak{M})\to (W/\!\!/ G)^{\times n}\times \mathfrak{M}$$
can be computed as 
\begin{equation}\label{equ on ta cp}\bbT_\textbf{\emph{{f}}}\cong \left(\pi_*(ad_{\fg}\calP_G\boxtimes\oO(-S)) \to \pi_*(\calW\boxtimes\oO(-S)) \to 
\pi_*\left(\calW^\vee\boxtimes\omega_{\log}\right)\to \pi_*\left(ad_{\fg}\calP_G\boxtimes\omega_{\log}\right) \right). \end{equation} 
By assumption, $\textbf{{Hilb}}^P(Y,D)$ is a finite type open substack of $\bMap^{\mathrm{twist}}_{\mathfrak{M}}(\calC,\bCrit^{}(\phi)\times \mathfrak{M})$. Therefore 
its tangent complex can be computed using \eqref{equ on ta cp}. Same argument as the proof of Lemma \ref{lem on exi of reso} gives the construction. 
\end{proof}
\begin{remark} We note that Lemma \ref{lem on exi of reso2} has straightforward generalizations when $D_0$ has more symmetries than $\bbC^*$-action.
For example, when $D_0=\C^3$,  the torus $(\C^*)^3$ acts coordinate-wise. 
Fix a principal $(\C^*)^3$-bundle $P$ on the distinguished component $C_0\cong \mathbb{P}^1$ whose associated rank three bundle is
$$\oO_{\mathbb{P}^1}(a)\oplus \oO_{\mathbb{P}^1}(b) \oplus \oO_{\mathbb{P}^1}(c), \quad \mathrm{with} \,\, a+b+c=n-2. $$ 
Via the contraction $\calC\to \mathbb{P}^1$, we can pullback $P$ to a bundle $\calP$ on $\calC$ and consider $P$-twisted maps as in \eqref{diag on twist mapp stac}
whose parametrizing stack is denoted by $\bMap^{\mathrm{twist}}_{\mathfrak{M}}(\calC,\bCrit^{}(\phi)\times \mathfrak{M})$. 

Let $Y=P\times_{(\C^*)^3}\C^3$, whose anti-canonical  divisor is a disjoint union $D=\bigsqcup_{i=1}^n D_0$. 
Assume that $\textbf{{Hilb}}^{P|_{D}}(D)=\bigsqcup_{i=1}^n \textbf{{Crit}}(\phi)$ for a regular function $\phi: W/\!\!/ G\to \C$
on a linear GIT quotient, and that $\textbf{{Hilb}}^P(Y,D)$ is contained in the open substack: 
\begin{equation}\label{equ on ope imm2}\bMap^{\mathrm{twist}}_{\mathfrak{M}}(\calC,\bCrit^{}(\phi)\times \mathfrak{M}) \hookrightarrow   
\bMap_{\mathfrak{M}}(\calP\times_{(\C^*)^3} D_0,\textbf{Perf}\times \mathfrak{M})_{0} \end{equation}
Then (Resolution) of Setting \ref{setting of lag} holds in this case by the same arguement as in Lemma \ref{lem on exi of reso2}.

In particular, the moduli stack of relative PT stable pairs on a log Calabi-Yau local curve $(Y=\oO_{\mathbb{P}^1}(a,b,c),D=\bigsqcup_{i=1}^n \C^3)$ with $n=a+b+c+2$ 
satisfies (Resolution) of Setting \ref{setting of lag}.
\end{remark}

\providecommand{\bysame}{\leavevmode\hbox to3em{\hrulefill}\thinspace}
\providecommand{\MR}{\relax\ifhmode\unskip\space\fi MR }
\providecommand{\MRhref}[2]{%
 \href{http://www.ams.org/mathscinet-getitem?mr=#1}{#2}}
\providecommand{\href}[2]{#2}

\end{document}